\documentclass[10pt]{article}
\usepackage{amsthm}
\usepackage{amsfonts, amssymb, amsmath, amscd, latexsym}
\usepackage{mathrsfs,epsfig}
\usepackage[latin1]{inputenc}
\usepackage[all]{xy}
\usepackage{cite}
\usepackage{subcaption}
\usepackage[dvipsnames]{xcolor}

\newtheorem{knowntheorem}{Theorem}

\def\today{\ifcase\month\or
January\or February\or March\or April\or May\or June\or July \or
August\or September\or October\or November\or December\fi
\space\NNumber\day, \NNumber\year}
\numberwithin{equation}{section}

\def\eu{ {\, \textrm{\rm e} }}

\def\sbwd{\sigma^{\textrm{bwd}}}

\def\sfwd{\sigma^{\textrm{fwd}}}
\def\sbwd{\sigma^{\textrm{bwd}}}
\def\sTfwd{\Sigma^{\textrm{fwd}}}
\def\sTbwd{\Sigma^{\textrm{bwd}}}
\def\dist{{\mathscr{D}}}
\def\al{\alpha}

\def\ga{\vec{\gamma}}

\def\ep{\varepsilon}
\def\la{\lambda}
\def\si{\sigma}

\def\Om{\Omega}
\def\na{\vec{\nabla}}

\def\N{{\mathbb N}}
\def\Z{{\mathbb Z}}
\def\R{{\mathbb R}}

\def\ee{{\mathrm e}}
\def\EE{{\mathcal E}}

\def\MM{{\mathcal M}}

\def\SS{{\mathcal S}}
\def\TT{{\mathcal T}}

\def\DDD{\mathscr{D}}
\def\TTT{\mathscr{T}}
\def\PPP{\mathscr{P}}

\def\dd{\delta}
\def\aup{a^{\uparrow}}
\def\adown{a^{\downarrow}}
\def\lzero{\Lambda^0}
\def\l1{\Lambda^{1}}

\def\Q{\vec{Q}}
\def\P{\vec{P}}

\def\zz{\vec{\zeta}}
\def\xxi{\vec{\xi}}

\def\x{\vec{x}}
\def\y{\vec{y}}
\def\f{\vec{f}}
\def\g{\vec{g}}

\newcommand{\lit}{{\lim_{t \to +\infty}}}
\newcommand{\lito}{{\lim_{t \to -\infty}}}

\def\Ae(t){\boldsymbol{A^\ep}}

\def\({\left(}
\def\){\right)}
\DeclareMathOperator{\tr}{tr}
\DeclareMathOperator{\lspan}{span}
\newcommand\bs[1]{{\boldsymbol{#1}}}

\def\ft{\footnote}
\def\und{\underline}
\def\ov{\overline}

\def\ogap{\Theta}   
\def\ngap{\Theta}    

\newcommand\assump[1]{{\rm\textbf{#1}}}
\def\inn{\textrm{in}}

\newtheorem{theorem}{Theorem}[section]
\newtheorem{proposition}[theorem]{Proposition}
\newtheorem{lemma}[theorem]{Lemma}
\newtheorem{cor}[theorem]{Corollary}
\theoremstyle{definition}

\theoremstyle{remark}
\newtheorem{remark}[theorem]{Remark}
\begin{document}
\title{Some contributions on Melnikov chaos for smooth and piecewise-smooth planar systems:
``trajectories chaotic in the future"}
\pagestyle{myheadings}
\date{}
\markboth{}{Melnikov chaos for planar systems}

\author{	
A. Calamai\thanks{Dipartimento di Ingegneria Civile, Edile e Architettura,
Universit\`a Politecnica delle Marche, Via Brecce Bianche 1, 60131 Ancona -
Italy. Partially supported by G.N.A.M.P.A. - INdAM (Italy) and
PRIN 2022 - Progetti di Ricerca di rilevante Interesse Nazionale, \emph{Nonlinear differential
problems with applications to real phenomena} (Grant Number: 2022ZXZTN2).
},
M. Franca\thanks{Dipartimento di Matematica e Informatica, Universit\`a di Firenze,  50134 Firenze, Italy.
 Partially supported by G.N.A.M.P.A. - INdAM (Italy) and PRIN 2022 - Progetti di Ricerca di rilevante Interesse Nazionale, \emph{Nonlinear
 differential
problems with applications to real phenomena} (Grant Number: 2022ZXZTN2). },
M. Posp\' i\v sil\thanks{Department of Mathematical Analysis and Numerical Mathematics, Faculty of Mathematics, Physics and Informatics,
Comenius University in Bratislava,
Mlynsk\'a dolina, 842 48 Bratislava, Slovakia; Mathematical Institute, Slovak Academy of Sciences, \v Ste\-f\'a\-ni\-ko\-va 49, 814 73
Bratislava, Slovakia.
Partially supported by the Slovak Research and Development Agency under the Contract no. APVV-23-0039, and by the Grants VEGA 1/0084/23 and
VEGA-SAV 2/0062/24.}
}

\maketitle
\begin{abstract}
We consider a $2$-dimensional autonomous system subject to a $1$-periodic perturbation, i.e.
$$ \dot{\x}=\f(\x)+\ep\g(t,\x,\ep),\quad \x\in\Om .$$
We assume that for $\ep=0$ there is a trajectory $\ga(t)$  homoclinic to the origin which is
a critical point: in this context Melnikov theory provides a sufficient condition for the insurgence of a
chaotic pattern when $\ep \ne 0$.

In this paper we show that for any line $\Xi$ transversal to $\{\ga(t) \mid t \in \R \}$ and any
$\tau \in [0,1]$  we can find a set $\aleph^+(\Xi,\tau)$ of initial conditions, located in $\Xi$ at $t=\tau$, giving rise
to a pattern chaotic just in the future, i.e.\ for $t\geq\tau$. Further diam$(\aleph^+(\Xi,\tau)) \le \ep^{(1+\nu)/ \und{\sigma}}$
where  $\und{\sigma}>0$ is a constant and $\nu>0$ is a parameter that can be chosen as large as we wish.

The same result holds true for $t\leq\tau$: we show that there is a set $\aleph^-(\Xi,\tau)$ of initial conditions giving rise
to a pattern chaotic just in the past.
In fact all the results are developed in a piecewise-smooth context, assuming that $\vec{0}$ lies on the discontinuity  curve: we recall
that in this setting chaos is not possible if we have sliding phenomena close to the origin.
This paper can also be
considered as
the first part of the project to show the existence of classical chaotic phenomena when sliding close to the origin is not present.
\end{abstract}

{\bf Keywords: Melnikov theory, piecewise-smooth systems,  homoclinic trajectories, non-autonomous dynamical systems,
non-autonomous perturbation}

{\bf 2020 AMS Subject Classification: Primary 34A36, 37G20, 37G30; Secondary 34C37, 34E10, 37C60.}

\section{Introduction}
In this paper we study the chaotic behavior of a $2$-dimensional piecewise-smooth  system  (possibly discontinuous) subject
to a non-autonomous perturbation, by means of the Melnikov method.
Let us start for illustrative purposes to consider the smooth case, i.e.,

\begin{equation}\label{eq-smooth}\tag{S}
	\dot{\x}=\f(\x)+\ep\g(t,\x,\ep),\quad \x\in\Om ,
\end{equation}
where  $\Omega \subset \R^2$ is an open set, $\ep\in\R$ is a small parameter, $\f$ and $\g$ are $C^r$, $r \ge 2$.

We assume that the origin $\vec{0} \in\Om\subset \R^2$ is a critical point for \eqref{eq-smooth} for any $\ep \ge 0$ and that
for $\ep=0$ there is a trajectory $\ga(t)$ homoclinic to $\vec{0}$.
The classical Melnikov theory provides a condition which is sufficient for the insurgence
of a chaotic pattern for $\ep \ne 0$.
Namely it is sufficient to require that $\g$ is $1$-periodic in $t$ and that a computable function $\MM(\tau)$, see
\eqref{melni-disc} below, has a non-degenerate zero:
this leads to the formation of a Smale horseshoe and of chaotic dynamics.

More precisely, let $\x(t,\tau; \xxi)$ denote the trajectory leaving from $\xxi$ at $t=\tau$, and denote by $\Phi(t; \tau) \xxi= \x(t,\tau;
\xxi)$, so that
$\Phi(t; \tau)$ is a $C^r$-diffemorphism for any $t,\tau \in \R$; let $\tau_0 \in [0,1]$ be such that $\MM(\tau_0)=0 \ne \MM'(\tau_0)$.
Let $\bs{\Gamma}= \{\ga(t) \mid t \in \R \} \cup \{\vec{0} \}$.
Then, there is $\ogap(\ep)$ large enough so that for any positive integer $k>\ogap(\ep)$ we can find  a set  $\aleph$ which is invariant
for $\Phi(k+\tau_0,\tau_0)$ and which has the following properties.
\emph{For any sequence $\ee \in \EE= \{0,1 \}^{\Z}$ there is a unique $\xxi \in \aleph$ such that
the trajectory $\x(t, \tau_0; \xxi)$ is either ``close to
 $\bs{\Gamma}$'' or ``close to the origin'';} i.e., either $\|\x(t, \tau_0; \xxi)-\ga(t-T_{2j})\|=O(\ep)$ or $\|\x(t,
 \tau_0; \xxi)\|=O(\ep)$
 whenever $t \in [T_{2j-1}, T_{2j+1}]$, where $T_j= jk+\tau_0$.

   This kind of results started from the work of Melnikov
  \cite{Me}, but an important step forward
was performed by Chow et al.\ in \cite{CHM}, and a big progress is due to
 \cite{Pa84} where Palmer  addressed the
$n$-dimensional case where $n \ge 2$.
This theory has been generalized in several directions, in particular it has been extended to the almost periodic case,
see e.g. \cite{Sch, PS},  and to the case where the zeros of $\MM(\tau)$ are degenerate, see e.g. \cite{BF02C, BF02M}.
Afterwards
the so-called perturbation approach has been widely developed by many authors.
 Melnikov theory is by now well-established for smooth systems, and there are many works devoted to it.
For example we refer to \cite{CHM,G1,GH, Pa84, Pa00, PS,Sch, St,MS,WigBook, Wig99}.

In this paper we review completely the construction of the chaotic pattern, motivated by the project to extend the classical theory to
a discontinuous setting, and to relax further the assumptions on the recurrence properties and on non-degeneracy of the zeros of the Melnikov
function.
 In particular all the proofs are written directly in the piecewise-smooth discontinuous setting.

However, in this paper we present some results of intrinsic mathematical interest, which, as far as we are aware, are new even in the
\emph{smooth} case. In particular we develop a new iterative scheme which allows us to select sets
$\aleph^+$ and $\aleph^-$  giving rise to solutions performing a  possibly infinite number of loops either in the
future or in the past,
following any prescribed sequence of two symbols.
Roughly speaking     the trajectories leaving
from $\aleph^+$ are ``chaotic in the future'', while the ones leaving from $\aleph^-$ are ``chaotic in the past'': we find
some information concerning $\aleph^{\pm}$ which have not appeared previously in literature, as far as we are aware.
More precisely let us  consider, as before, the case in which
 $\g$ is $1$-periodic in $t$.
Let $\Xi$ be a  line (or more generally, any curve)
 transversal to any point in $\bs{\Gamma} \setminus \{ \vec{0} \}$,  say $\ga(s)$, $s \in \R$ for definiteness,
 let $\tau_0 \in [0,1]$ and let $T_j=kj+\tau_0$ as above and set
$\TT^+=(T_j)$, $j \ge 1$, $\TT^-=(T_j)$, $j \le -1$. Then
we construct the sets $\aleph^+= \aleph^+(\tau, \TT^+, \Xi)$ and
$\aleph^-= \aleph^-(\tau, \TT^-, \Xi)$ with the following properties: if $\xxi \in \aleph^+$ then
$\x(t,  \tau ; \xxi)$ is either ``close to
 $\bs{\Gamma}$'' or ``close to the origin'' for $t \ge \tau$, while if $\xxi \in \aleph^-$,
 $\x(t,  \tau ; \xxi)$ has this property for $t \le \tau$  (see Section \ref{fractal.space} for more details).

  In fact we may  also get a  better localization: if we replace the constant $\ngap(\ep)$ defined above by
 $(\nu+1) \ngap(\ep)$ where $\nu \ge \nu_0 \ge 1$ is arbitrary high,
   the set $\aleph^+$ gets smaller. More precisely, let
  us
  denote by diam$(\aleph^+)$
  the diameter of $\aleph^+$. Then we get
  $$ \textrm{diam}(\aleph^+):= \sup \{ \|\vec{P}-\vec{Q}\| \mid \vec{P}, \vec{Q} \in \aleph^+ \} =O(\ep^{ (1+\nu)/\underline{\sigma}}),$$
  where $0<\underline{\sigma}<1$ is a constant which depends only on the eigenvalues of $\bs{f_x}(\vec{0})$, see \eqref{defsigma}.
 Whence $\textrm{diam}(\aleph^+)$ can be chosen arbitrarily small, even with respect to $\ep$ which is the size of the perturbation,
just paying the prize of a larger time spent by the trajectories to perform a loop; further
$\aleph^+$ is located  in a one-dimensional set.

 However, $\aleph^+$ oscillates with $\tau$ within an $\ep$-neighborhood of $\ga(s)$; so even if we know that its
  diameter is $O(\ep^{ (1+\nu)/\underline{\sigma}})$, and that it is $O(\ep^{ (1+\nu)/\underline{\sigma}})$ close to $\P_s(\tau)$,
 the intersection point between the stable manifold $\tilde{W}^s(\tau)$ and $\Xi$,
  we know its position just with $O(\ep)$ accuracy.

    In fact we have all the analogous results for the set $\aleph^-$ giving rise to patterns chaotic in past.

     As  discussed in \cite{CFPdiscAr},  this $O(\ep^{ (1+\nu)/\underline{\sigma}})$
      precision on the size of $\aleph^+$ and $\aleph^-$ is lost when we look at the set of initial
conditions
$\aleph$ giving rise to the classical pattern: chaotic both in the past and in the future. In fact
$\aleph$ will lie $O(\ep^{ (1+\nu)/\underline{\sigma}})$ close to the stable and unstable manifold but it
 will spread in an $O(\ep)$ neighborhood along the direction
of $\bs{\Gamma}$.
So we cannot locate $\aleph$ in a one-dimensional set.

 On the contrary, we
 emphasize  once again that the sets $\aleph^\pm$ are located in a $1$-dimensional object, $\Xi$, further we have a different set for any
 $\tau \in [0,1]$ and
 any  line (or more generally, any curve)
$\Xi$.
Indeed we conjecture that for any $\ee^+ \in \EE^+=\{0;1 \}^{\Z^+}$ and any $\tau \in [0,1]$ we may find a manifold of initial conditions of
trajectories
which mimic the
sequence $\ee^+$ in forward time.

 Our approach, inspired by the work by Battelli and Fe\v ckan, see \cite{BF08, BF10, BF11, BF12}, allows us to obtain a great flexibility
in the choice of the sets $\aleph^{\pm}$, and to weaken the recurrence properties required from the Melnikov
function $\MM(\tau)$,  see  \eqref{melni-disc} below.
\\
Let us consider again the smooth case.
In fact, even if the system is $1$-periodic we can choose aperiodic sequences $\tilde{\TT}=(\tilde{T}_j)_{j \in \Z}$.
 For example we can choose $\tilde{\TT}$ so that the gap $  \tilde{T}_j- \tilde{T}_{j-1}$
varies periodically between $\big[ \ngap +1 \big]$ and $2 \big[ \ngap +1 \big]$ (here $[\cdot]$ denotes the integer part) or it
  becomes unbounded.
Further, if there is $\tau_1 \ne \tau_0$ such that $\MM(\tau_1)=0 \ne \MM'(\tau_1)$ and   $\tau_1 \in [0,1]$, we may choose
a sequence $\bar{\TT}=(\bar{T}_j)_{j \in \Z}$ jumping randomly from a translate of $\tau_0$ to a translate of $\tau_1$.

Anytime we change either one of $\tau$, $\TT^{\pm}$  we obtain different sets $\aleph^{\pm}$ (in fact infinitely many of them) giving rise to
different chaotic
patterns and
they are all contained in the same $O(\ep^{ (1+\nu)/\underline{\sigma}})$ one-dimensional neighborhood, and they vary also if we change $\Xi$.

 In our opinion all these results concerning $\aleph^+$ and $\aleph^-$ increase our knowledge on the sensitive dependence of the system  on
 initial conditions.

We stress the fact that we assume much weaker recurrence properties than the classical ones. In particular we do not require any
non-degeneracy of the zeros of
$\MM(\tau)$
 (we just need $\MM$ to change sign): this allows us to consider systems in which $\g$ is made up by
 a sum of a periodic component and a noise,  not necessarily small, and on which we have little information, see Remark \ref{remP1}.
Notice that for this reason we choose to formulate our assumptions in terms of the Melnikov
function instead of  in terms of the maps $\f$ and $\g$.

 Further we show that, in the periodic case, the action of $\Phi(T_{i+1}, T_i)$ on $\aleph^+$ and $\aleph^-$ is semi-conjugated to the action
 of the Bernoulli
 shift on
$\EE^+= \{0;1 \}^{\Z^+}$ and $\EE^-= \{0;1 \}^{\Z^-}$ respectively, and we obtain analogous results in the aperiodic case.

In fact in literature we usually find conjugation (or semi-conjugation as e.g.\ in \cite{BF02C, SW}) with the Bernoulli shift on two symbols,
but using a fixed time gap $T_{j+1}-T_j=k>  \ngap(\ep)$ for any $j$ if the system is periodic or almost periodic (see
assumption \assump{P2} in Section \ref{fractal.space})
and $k$ is a multiple of a period or a quasi-period.
Here we can reproduce the classical situation (i.e. constant $T_{j+1}-T_j$), but we can also consider sequences where
$T_{j+1}-T_j=k_j>  \ngap(\ep)$ varies, see Section \ref{proof.Bernoulli}. In the former case we find semi-conjugation with the shift in two
symbols
with constant time gaps while in the latter with variable time gaps.
We use always   the same endomorphism in the periodic case, and  slightly different endomorphisms  in the almost periodic case
as we classically find in literature, see e.g.~\cite{PS}.

This allows us to consider the aperiodic case with variable time gaps and a different sequence  of endomorphisms
producing the semi-conjugation. In fact we think that   in all the cases we might find positive topological entropy
in the spirit of \cite{BW}, even if an actual proof is beyond the purpose of this paper.

In fact all the results of this paper are obtained already in a piecewise-smooth context. More precisely
we consider the piecewise-smooth system
\begin{equation}\label{eq-disc}\tag{PS}
	\dot{\x}=\f^\pm(\x)+\ep\g(t,\x,\ep),\quad \x\in\Om^\pm ,
\end{equation}
where $\Om^{\pm} = \{ \x\in\Om \mid \pm
G(\x)>0\}$, $\Om^{0}= \{ \x\in\Om \mid G(\x) = 0 \}$,   $G$ is a $C^{r}$-function on $\Om$ with $r> 1$
such that $0$ is a regular value of
$G$. Next,  $\f^\pm \in C^{r}_b(\Om^\pm \cup \Omega^0, \R^2)$, $\g\in C_b^r(\R\times\Om\times\R,\R^2)$ and
$G\in C_b^r(\Om,\R)$, i.e., the
derivatives of $\f^\pm$, $\g$ and $G$ are uniformly continuous and bounded up to the $r$-th order, respectively,
 if $r\in \N$ and up to $r_0$ if $r=r_0+r_1$, with $r_0 \in \N$ and $0< r_1 <1$ and the $r_0$-th derivatives are $r_1$
H\"older continuous.

In the last 20 years many authors addressed the problem of generalizing Melnikov theory
to a discontinuous (piecewise-smooth) setting, but assuming that $\vec{0} \not \in \Om^0$: among the other papers, let us mention e.g.,
 \cite{K2,KRG, LSZH16, BF08, BF11, CaFr, CDFP, DE13, GZ23, HuLi}.
Most of the cited papers concern the $2$-dimensional case, but Battelli  and Fe\v ckan managed to address the $n \ge 2$ case, and
they proved the persistence of a homoclinic \cite{BF08}, the insurgence of chaos
when the unperturbed homoclinic exhibits no sliding phenomena \cite{BF11} and when it does \cite{BF10}.

If, instead, we assume that the critical point \emph{lies on the discontinuity hypersurface} (in the case $n=2$, the curve $\Omega^0$),
the problem of detecting a chaotic behavior becomes even more challenging.
In this setting,
as a preliminary step, in \cite{CaFr, HuLi, HL24}  it was shown that the Melnikov condition found by Battelli and Fe\v ckan together with a
further
(generic) transversality requirement
(always satisfied in two dimensions)  are enough to prove the persistence of the homoclinic trajectory in the $n \ge 2$ case.

Quite surprisingly, in   \cite{FrPo},
which is still set in the two-dimensional case, it was shown
that the Melnikov condition,
which ensures the persistence of the homoclinic \cite{CaFr} and the transversality of
the crossing between stable and unstable leaves, \emph{does not guarantee} the insurgence of chaos
differently from the smooth setting, and also from the piecewise-smooth setting considered, e.g., in  \cite{K2,KRG, LSZH16, BF08, BF11,
CaFr,CDFP, LPT,MV21}.
In fact,  a natural       geometrical obstruction prevents the formation of chaotic patterns, and new bifurcation
 phenomena take place, scenarios which may exist just in a discontinuous context.

 This paper can also be regarded as an intermediate step to complete the picture   in the two-dimensional  discontinuous  case showing
 that, when    this obstruction is removed,
  system \eqref{eq-disc} exhibits a chaotic behavior
as in the smooth setting: the complete result is discussed in \cite{CFPdiscAr}.

 Summing up the main achievements of the paper   are the following:
  \begin{description}
 \item[1] We construct new sets $\aleph^-$ and $\aleph^+$ of initial conditions giving rise
 respectively  to patterns chaotic in the past and chaotic in the future.
 These sets lie on  \emph{any arbitrarily chosen  $1$-dimensional transversal} to  $\bs{\Gamma} \setminus \{\vec{0}\}$, say $\Xi$, and their
 diameters
 may be chosen of order $O(\ep^{(1+\nu)/\underline{\sigma}})$, where
   $\nu \ge \nu_0 \ge 1$ is arbitrary high and $\und{\sigma}>0$ is a constant (independent of $\ep$ and $\nu$),
 see \eqref{defsigma}, even if  the size $\ep$ of
 the perturbation is fixed.
 Further if we fix the sequence $\TT^{\pm}$ we get a new $\aleph^{\pm}$
 whenever we let either $\tau$ or $\Xi$ vary.
  \item[2] We   weaken the recurrence requirements on the function $\MM(\tau)$ and we allow its zeros to be degenerate.
   \item[3] The results are proved in a piecewise-smooth setting, and they  are a first step to prove the existence of a
   classical chaotic pattern (i.e., a pattern taking place both in the past and in the future) in the context where the critical point
 $\vec{0}$  lies on the discontinuity   level  $\Om^0$. This analysis will be completed in a forthcoming paper.
\end{description}

The paper is divided as follows: in Section \ref{S.WuWs} we  list  the main assumptions and recall some basic constructions such as stable and
unstable manifold; in Section \ref{fractal.space} we state the main results of the paper, i.e., Theorems \ref{restatefuture},
\ref{restatepast}
and \ref{restatebis};
in Section 4, using the results of
\cite{CFPRimut}, we construct the  Poincar\'e map
going from a transversal to $\bs{\Gamma}$ at  $\ga(s)$ (but we set $s=0$ for definiteness)  back to itself and in particular we state the
crucial results,
Theorems \ref{key} and \ref{keymissed},
which evaluate space displacement and flight time in performing a loop; in Section \ref{S.pf.main}
we prove our main results: we develop the iterative scheme needed to construct $\aleph^{+}$ in \S \ref{proof.future} in the setting with
weaker assumptions on
the zeros of $\MM(\tau)$, while in \S
 \ref{proof.futurebis} we adapt the argument to a setting where the zeros of $\MM(\tau)$
are non-degenerate, so that we get sharper results. Then in \S \ref{proof.past} we use a classical inversion of time argument
to construct $\aleph^-$.
 In Section \ref{proof.Bernoulli}, following a classical scheme, we show that
under the assumptions of Theorems \ref{restatefuture} and \ref{restatepast},
 the action of the forward (resp.\ backward) flow of
(PS) on the set $\aleph^+$ (resp.\ $\aleph^-$) is semi-conjugated with the
forward (resp.\ backward) Bernoulli shift.
Finally in the appendix we correct a minor error  in the formula for the $2$-dimensional Melnikov function in the
piecewise-smooth case, appeared in \cite{CaFr} and then repeated in \cite{CFPRimut}, which however does not affect the main argument of those
papers.

\section{Preliminary constructions and notation}\label{S.WuWs}

 In this preliminary section
 we specify the notion of solution
of the system \eqref{eq-disc}
and we collect the basic assumptions which we assume through the whole paper.
By a \emph{solution} of
\eqref{eq-disc} we mean a continuous, piecewise $C^r$ function
$\x(t)$ that satisfies
\begin{align}
\dot{\x}(t)=\f^+(\x(t)) + \ep \g(t,\x(t),\ep), \quad \textrm{whenever } \x(t) \in \Om^+,  \label{eq-disc+} \tag{PS$+$}\\
\dot{\x}(t)=\f^-(\x(t)) + \ep \g(t,\x(t),\ep), \quad \textrm{whenever } \x(t) \in \Om^-. \label{eq-disc-}\tag{PS$-$}
\end{align}
 Moreover, if  $\x(t_{0})$ belongs to
$\Om^{0}$ for some $t_0$,   then  we assume
either $\x(t)\in\Om^{-}$ or $\x(t)\in\Om^{+}$  for $t$ in some left neighborhood of
$t_{0}$, say $]t_{0}-\tau,t_{0}[$ with $\tau>0$.
In the first case, the
left derivative of $\x(t)$ at $t=t_0$ has to satisfy $\dot
\x(t_{0}^{-}) = \f^{-}(\x(t_0))+ \ep \g(t_0,\x(t_0),\ep)$; while in the
second case, $\dot \x(t_{0}^{-}) = \f^{+}(\x(t_0))+ \ep
\g(t_0,\x(t_0),\ep)$. A~similar condition is required for the right
derivative $\dot \x(t_{0}^{+})$.
We stress that, in this paper, we do not consider solutions of equation
\eqref{eq-disc} that belong to $\Om^{0}$ for $t$ in some
nontrivial interval, i.e., sliding solutions.

\subsubsection*{Notation}
Throughout the paper we will use the following notation. We denote scalars by small letters,
e.g.\ $a$, vectors in $\R^2$ with an arrow, e.g.\ $\vec{a}$, and $n \times n$ matrices by bold letters, e.g.\ $\bs{A}$.
By $\vec{a}^*$ and $\bs{A}^*$ we mean the transpose of the vector $\vec{a}$ and of the matrix $\bs{A}$, resp.,
so that $\vec{a}^*\vec{b}$ denotes the scalar product of the vectors $\vec{a}$, $\vec{b}$.
We denote by $\|\cdot\|$ the Euclidean norm in $\R^2$, while for matrices
we use the functional norm $\|\bs{A}\|= \sup_{\|\vec{w}\| \le 1} \|\bs{A} \vec{w}\|$.
We will use the shorthand notation $\bs{f_x}=\bs{\frac{\partial f}{\partial x}}$ unless this may cause confusion.\\
Further we denote by $B(\xxi,\delta):=\{\Q \mid \|\Q-\xxi\| < \delta \}$.

\medskip

We list here some hypotheses which we assume through the whole paper.

\begin{description}
	\item[\assump{F0}]  $\vec{0}\in\Omega^0$,  $\f^\pm(\vec{0})=\vec{0}$, and the eigenvalues $\la_s^\pm$, $\la_u^\pm$ of
$\bs{f_x^\pm}(\vec{0})$ are such that
$\la_s^\pm<0<\la_u^\pm$.
\end{description}

Denote by $\vec{v}_s^{\pm}$, $\vec{v}_u^{\pm}$ the normalized eigenvectors of $\bs{f_x^\pm}(\vec{0})$ corresponding to $\la_s^\pm$,
$\la_u^\pm$.
 We assume that the eigenvectors $\vec{v}_s^{\pm}$, $\vec{v}_u^{\pm}$ are not orthogonal to $\na G(\vec{0})$. To fix the ideas we require
\begin{description}
	\item[\assump{F1}]  $[\na G(\vec{0})]^*\vec{v}^{-}_u<0<[\na G(\vec{0})]^*\vec{v}^{+}_u \, ,
\quad  [\na G(\vec{0})]^*\vec{v}^{-}_s <0< [\na G(\vec{0})]^*\vec{v}^{+}_s $.
\end{description}

 Moreover we require a further condition on the mutual positions
 of the
directions spanned by
$\vec{v}_s^{\pm}$, $\vec{v}_u^{\pm}$.

More precisely, set $\mathcal{T}_u^{\pm}:=\{ c\vec{v}_u^{\pm} \mid c \ge 0 \}$, and
denote by $\Pi_u^1$ and $\Pi_u^2$ the disjoint open sets
in which $\R^2$ is divided by the polyline $\mathcal{T}^u:=\mathcal{T}_u^+ \cup \mathcal{T}_u^-$.
We require that $\vec{v}_s^+$ and $\vec{v}_s^-$ lie on ``opposite sides'' with respect to
	$\mathcal{T}^u$.
Hence, to fix the ideas, we assume: \begin{description}
	\item[\assump{F2}]   $\vec{v}_s^+ \in \Pi_u^1$ and $\vec{v}_s^- \in \Pi_u^2$.
\end{description}

We emphasize that if \assump{F2} holds, there is no sliding on $\Om^0$ close to $\vec{0}$.
On the other hand,  sliding might occur  when both  $\vec{v}_s^{\pm}$ lie in
$\Pi_u^1$, or they both lie in  $ \Pi_u^2$, see \cite[\S 3]{FrPo}.

\begin{remark} \label{rem:settings}
We point out that it is the \emph{mutual position of the eigenvectors} that plays a role in the argument. In the paper we fix a particular
situation
for definiteness; however, by
reversing all the directions, one may obtain equivalent results.
Moreover, in the continuous case $\mathcal{T}^u$ is a line and $\Pi_u^1$, $\Pi_u^2$ are halfplanes.
In fact, all
smooth systems satisfy \assump{F2}.
On the other hand, if assumption \textbf{F2} is replaced with the opposite condition, that is $\vec{v}_s^+$ and $\vec{v}_s^-$ lie ``on the
same side'' with respect to
$\mathcal{T}^u$, then it was shown in \cite{FrPo} that  generically  chaos cannot occur,  while new bifurcation
phenomena, involving continua of sliding homoclinic trajectories, may arise.
\end{remark}
\begin{description}
	\item[\assump{K}] For $\ep=0$ there is  a unique solution $\ga(t)$ of \eqref{eq-disc} homoclinic to the origin such that
	$$\ga(t)\in\begin{cases}\Om^-,& t<0,\\\Om^0,&t=0,\\\Om^+,&t>0.\end{cases}$$
	Furthermore, $(\na G(\ga(0)))^*\f^\pm(\ga(0))>0$.
\end{description}

 Sometimes it will be convenient to denote by $\ga^-(s)=\ga(s)$ and  by $\ga^+(t)=\ga(t)$ when $s \le 0 \le t$ so that $\ga^{\pm}(\tau) \in
 (\Omega^\pm \cup
 \Omega^0)$ for any $\tau\in\R$.

Recalling  the orientation of $\vec{v}_s^{\pm}$, $\vec{v}_u^{\pm}$ chosen in \assump{F1}, we assume w.l.o.g.\  that
\begin{equation}\label{ass.scenario}
\lito\frac{\dot{\ga}(t)}{\|\dot{\ga}(t)\|}=\vec{v}_u^-
\quad \mbox{ and } \quad
\lit\frac{\dot{\ga}(t)}{\|\dot{\ga}(t)\|}=-\vec{v}_s^+.
\end{equation}

Concerning the perturbation term $\g$, we assume the following:
\begin{description}
	\item[\assump{G}]  $\g(t,\x, \ep)$ is bounded together with its derivatives for any $t \in \R$, $\x \in B(\bs{\Gamma}, 1)$,  $\ep \ge 0$
and
$\g(t,\vec{0},\ep)=\vec{0}$ for any $t\in\R$, $\ep \ge 0$.
\end{description}
Hence, the origin is a critical point for the perturbed problem too.

\begin{figure}[t]
\centerline{\epsfig{file=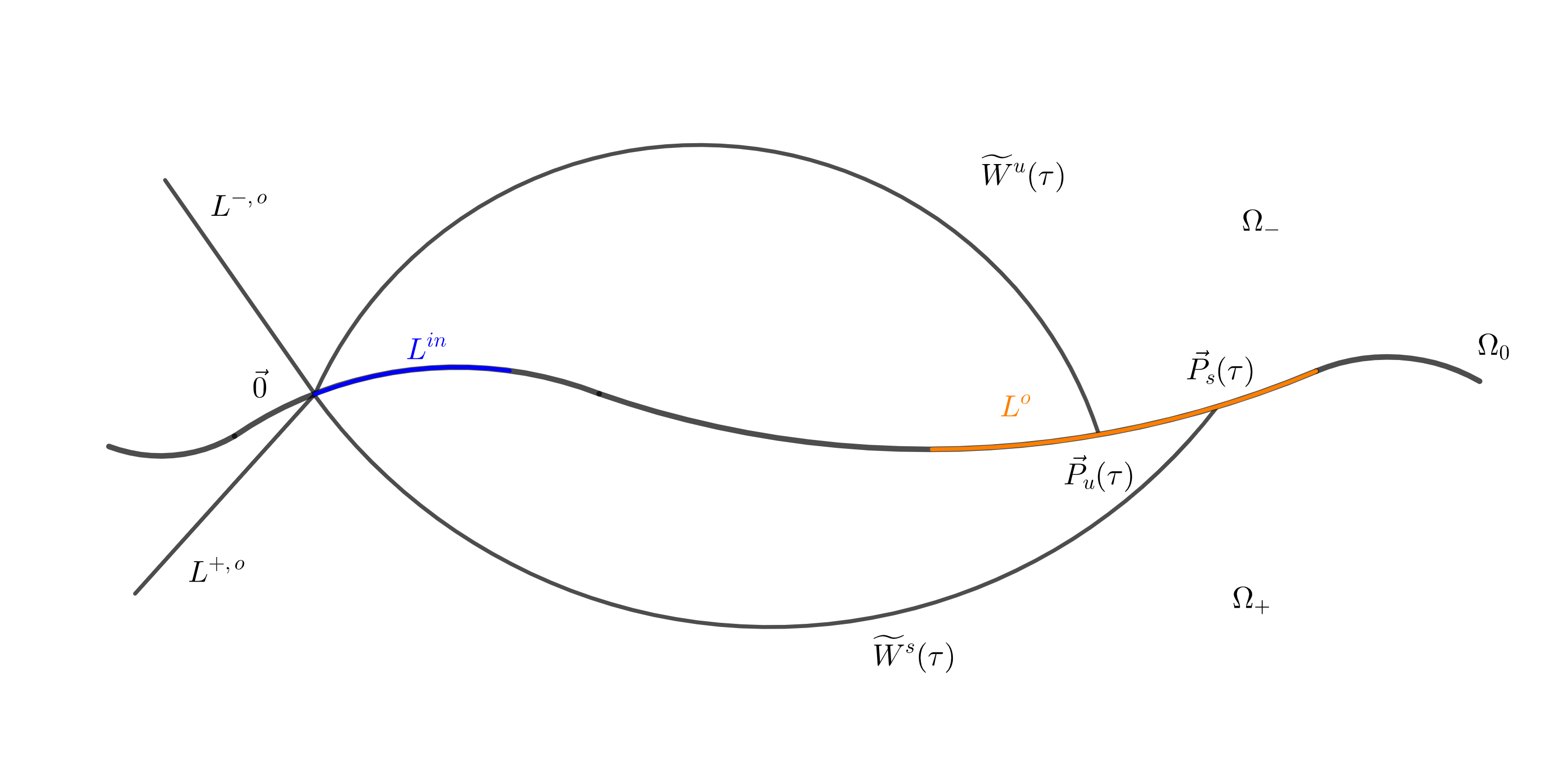, width = 10 cm} }
\caption{Stable and unstable leaves (the superscript ``$^{\textrm{out}}$"
	is denoted as ``$^{\textrm{o}}$" for short).}
\label{LinL0}
\end{figure}
We recall that
$\bs{\Gamma}:= \{\ga(t) \mid t \in \R \} \cup \{\vec{0}\}$; let us denote  by $E^{\textrm{in}}$ the open set enclosed by $\bs{\Gamma}$,
and by $E^{\textrm{out}}$ the open set complementary to $E^{\textrm{in}}\cup\bs{\Gamma}$.

Further, for any fixed $\delta>0$, we set
 \begin{equation}\label{L0}
 \begin{split}
   L^0 =  L^0(\delta) & := \{ \vec{Q} \in \Omega^0 \mid \|\vec{Q}- \ga(0) \| <\delta \}, \\
   L^{\textrm{in}}=  L^{\textrm{in}}(\delta) &:= \{ \vec{Q} \in (\Omega^0 \cap E^{\textrm{in}}) \mid \|\vec{Q} \| <\delta\}.
  \end{split}
\end{equation}
We denote by $\x(t,\tau;\P)$   the trajectory of \eqref{eq-disc} which is in $\P$ at $t=\tau$.
Now, we define  the stable and the unstable leaves $W^s(\tau)$ and $W^u(\tau)$ of \eqref{eq-disc}.

Assume first for simplicity that the system is
\textit{smooth} and suppose that \assump{F0} holds true.
 Then, following \cite[\S 2]{CFPRimut}, which is based on \cite[Theorem 2.16]{Jsell},  we can define the global stable and unstable
 leaves as follows:
\begin{equation}\label{mm}
\begin{array}{c}
W^u(\tau) := \{ \P \in \R^2 \mid    \lito \x(t,\tau;\P)=\vec{0} \}, \\
W^s(\tau) := \{ \P \in \R^2 \mid   \lit  \x(t,\tau;\P)=\vec{0} \}.
\end{array}
\end{equation}
In fact  $W^u(\tau)$ and $W^s(\tau)$ are $C^r$ immersed $1$-dimensional manifolds,
i.e., they are the image of $C^r$ curves;  they also have $C^r$ dependence on $\tau$ and $\ep$ but we leave this last dependence unsaid.
Notice that if $\P \in W^u(\tau)$ then $\x(t,\tau;\P)\in W^u(t)$
for any $t, \tau \in \R$. Analogously for $W^s(\tau)$.

Assume further \assump{F1}, \assump{K} and follow $W^u(\tau)$ (respectively $W^s(\tau)$) from the origin towards
$L^0(\sqrt{\ep})$: then it intersects $L^0(\sqrt{\ep})$ transversely in a point denoted by $\P_u(\tau)$
(respectively by $\P_s(\tau)$).
In fact, $\P_u(\tau)$ and $\P_s(\tau)$ are $C^r$ functions of $\ep$ and $\tau$; hence
$\P_u(\tau)=\P_s(\tau)= \ga(0)$ if $\ep=0$ for any $\tau \in \R$.

We denote by $\tilde{W}^u(\tau)$ the branch of $W^u(\tau)$ between the origin and $\P_u(\tau)$
(a path), and by $\tilde{W}^s(\tau)$ the branch of $W^s(\tau)$ between the origin and $\P_s(\tau)$, in both the cases including the endpoints.
Since $\tilde{W}^u(\tau)$ and $\tilde{W}^s(\tau)$ coincide with $\bs{\Gamma} \cap (\Om^- \cup \Om^0)$
and $\bs{\Gamma} \cap (\Om^+ \cup \Om^0)$ if $\ep=0$, respectively, and vary in a $C^r$ way
 we find $\tilde{W}^u(\tau)\subset  (\Om^- \cup \Om^0)$ and
$\tilde{W}^s(\tau)\subset  (\Om^+ \cup \Om^0)$, for any $\tau \in \R$ and any $0\le \ep \le \ep_0$, see again \cite{CFPRimut},
\cite[Theorem 2.16]{Jsell}
or  \cite[Appendix]{mFaS}.
Further
\begin{equation}\label{proprieta}
\begin{gathered}
 \Q_u \in \tilde{W}^u(\tau)   \Rightarrow  \x(t,\tau; \Q_u) \in \tilde{W}^u(t) \subset (\Om^- \cup \Om^0) \quad \textrm{for any $t\le
 \tau$},\\
\Q_s \in \tilde{W}^s(\tau) \Rightarrow  \x(t,\tau; \Q_s) \in \tilde{W}^s(t) \subset (\Om^+ \cup \Om^0) \quad \textrm{for any $t\ge \tau$}.
\end{gathered}
\end{equation}

Now, we go back to the general case where \eqref{eq-disc} is piecewise-smooth but discontinuous.

\begin{remark}\label{Wloc}
Consider \eqref{eq-disc} and assume
 \assump{F0}, \assump{F1},  \assump{F2}.
 Following \cite{CFPRimut} we can again define
$\tilde{W}^{u}(\tau)$ and  $\tilde{W}^{s}(\tau)$,
and  we get that  they are $C^r$ and
 have the property~\eqref{proprieta}, see \cite{CFPRimut} for details.

 Moreover, if \textbf{K} holds then
  $\P_u(\tau)$ and $\P_s(\tau)$ are again $C^r$ in $\ep$ and $\tau$,
 and $\P_u(\tau)=\P_s(\tau)= \ga(0)$ if $\ep=0$ for any $\tau \in \R$.
\end{remark}
We set
\begin{equation}\label{tildeW}
\tilde{W}(\tau):= \tilde{W}^u(\tau) \cup \tilde{W}^s(\tau).
\end{equation}


\begin{figure}[t]
\begin{center}
	\epsfig{file=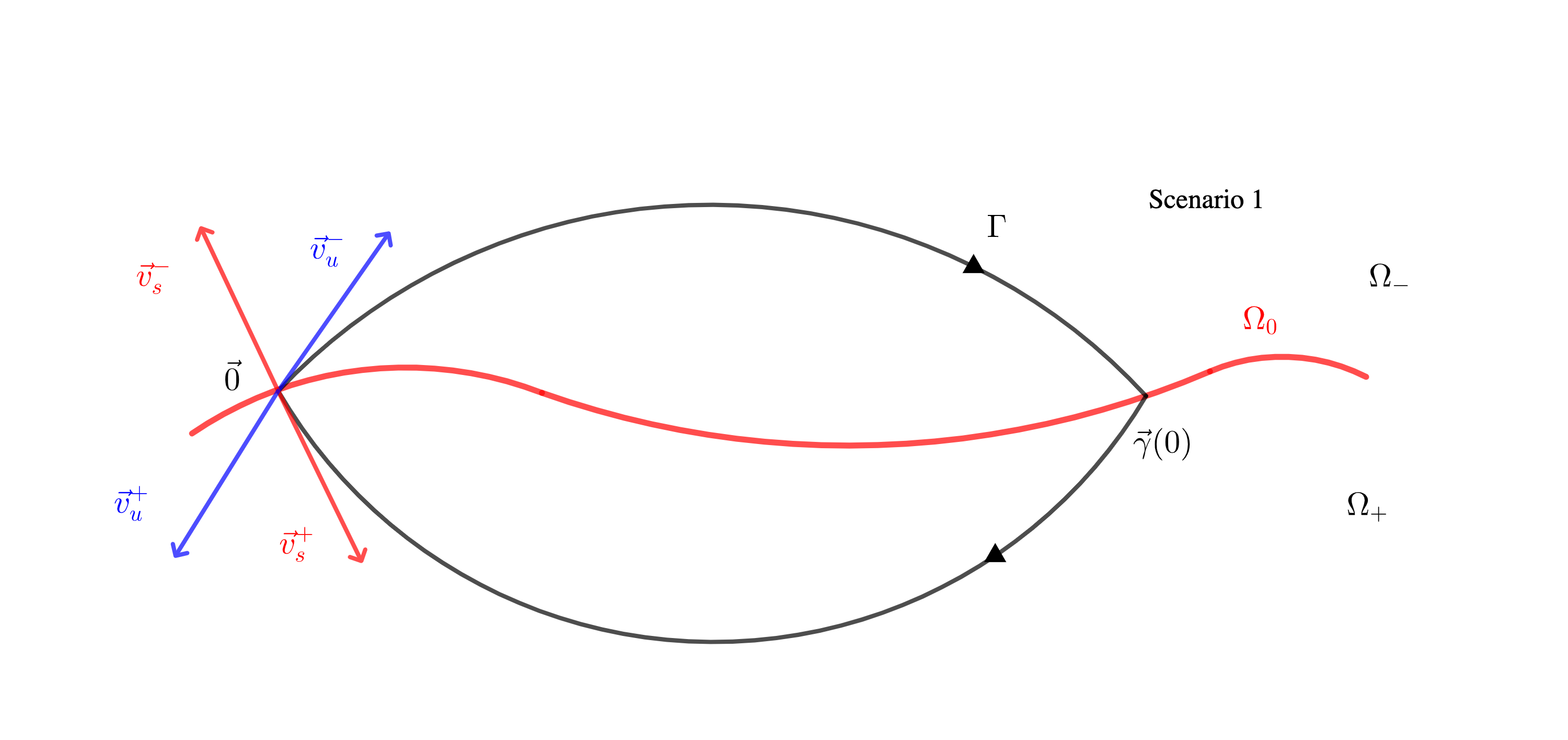, width = 8 cm}\\
	\epsfig{file=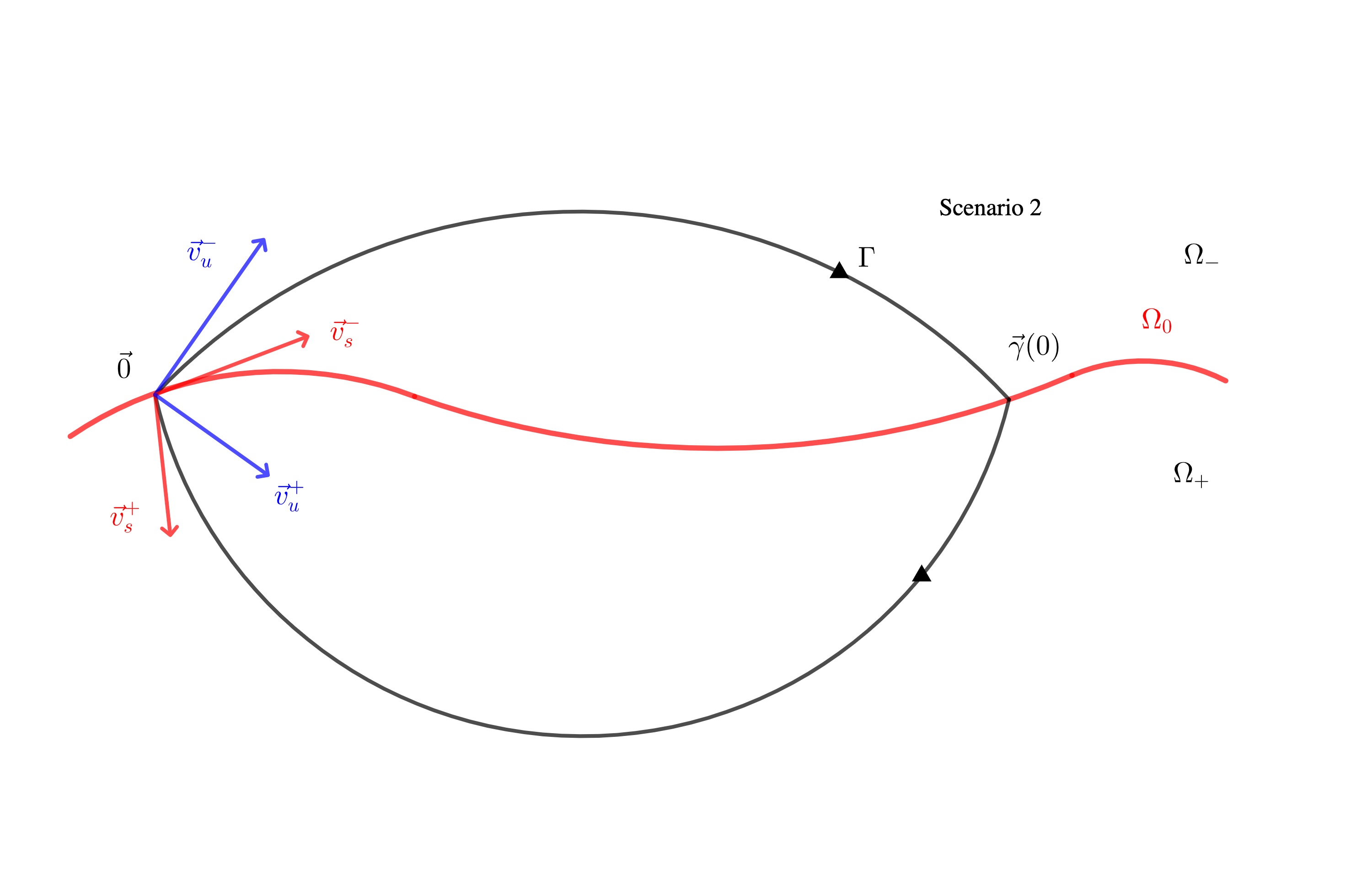, width = 8 cm}
\end{center}
\caption{Scenarios 1 and 2. In these two settings there is no sliding close to the origin. Further Melnikov theory guarantees persistence of
the homoclinic
trajectories
  \cite{CaFr}, and we conjecture that we may have chaotic phenomena.
}
\label{scenario123}
\end{figure}

At this point, we need to distinguish between four possible scenarios, see Figures~\ref{scenario123} and \ref{scenario45}.
 \begin{description}
\item[Scenario 1] Assume \assump{K} and that there is $\rho>0$ such that $d\vec{v}_u^+ \in E^{\textrm{out}}$, $d\vec{v}_s^- \in
    E^{\textrm{out}}$ for any $0<d<\rho$.
\item[Scenario 2] Assume \assump{K} and that there is $\rho>0$ such that $d\vec{v}_u^+ \in E^{\textrm{in}}$, $d\vec{v}_s^- \in
    E^{\textrm{in}}$ for any $0<d<\rho$.
\item[Scenario 3]
Assume \assump{K} and that there is $\rho>0$ such that $d\vec{v}_u^+ \in E^{\textrm{in}}$, $d\vec{v}_s^- \in E^{\textrm{out}}$ for any
$0<d<\rho$, so
    \assump{F2} does not hold.
\item[Scenario 4]
Assume \assump{K} and that there is $\rho>0$ such that $d\vec{v}_u^+ \in E^{\textrm{out}}$, $d\vec{v}_s^- \in E^{\textrm{in}}$ for any
$0<d<\rho$, so
    \assump{F2} does not hold.
\end{description}
Notice that  \assump{F2} holds in both Scenarios 1 and 2, and our results apply to both the cases. Further, let
$$W^{u,+}_{loc}(\tau) :=W^u(\tau)\cap\Omega^+\cap B(\vec{0},\rho),\quad  W^{s,-}_{loc}(\tau):=W^s(\tau)\cap\Omega^-\cap B(\vec{0},\rho)$$ for
some $\rho>0$;
in Scenario 1
 both $W^{u,+}_{loc}(\tau)$ and $W^{s,-}_{loc}(\tau)$ lie in $E^{\textrm{out}}$ for any $\tau \in \R$,
while in Scenario 2 $W^{u,+}_{loc}(\tau)$ and $W^{s,-}_{loc}(\tau)$ both lie in $E^{\textrm{in}}$ for any $\tau \in
    \R$.

    In Scenarios 3 and 4  sliding generically occurs in $\Om^0$ close to the origin, \assump{F2} does not hold,   and
    $W^{u,+}_{loc}(\tau)$ and  $W^{s,-}_{loc}(\tau)$ lie on the opposite sides with respect to $\bs{\Gamma}$.
    Notice that Scenarios 1 and 2 have a smooth counterpart while Scenarios 3 and 4 may take place just if the system is discontinuous.
     We recall once again that
    in all the four scenarios the existence of a non degenerate zero of the Melnikov function guarantees the persistence of the homoclinic
    trajectory, cf.~\cite{CaFr}, but  generically  chaos is not possible in Scenarios 3 and 4; we conjecture that chaos is still possible in
    Scenarios 1 and 2:
    this will
    be the object of future investigation which will use the results of this article.

    \emph{In this paper we will just consider Scenario 1 to fix the ideas although Scenario~2 can be handled in a similar way};
 accordingly we introduce the following notation, see Figure \ref{LinL0},
     \begin{equation*}\label{L+-}
 \begin{split}
   L^{-,\textrm{out}}=L^{-,\textrm{out}}(\delta) &:= \{ \vec{Q}= d (\vec{v}_u^-+  \vec{v}_s^-) \mid 0 \le d \le \delta \},\\
   L^{+,\textrm{out}}= L^{+,\textrm{out}}(\delta) &:= \{ \vec{Q}= d (\vec{v}_u^++  \vec{v}_s^+) \mid 0 \le d \le \delta \}.
  \end{split}
\end{equation*}

\begin{figure}[t]
\begin{center}
	\epsfig{file=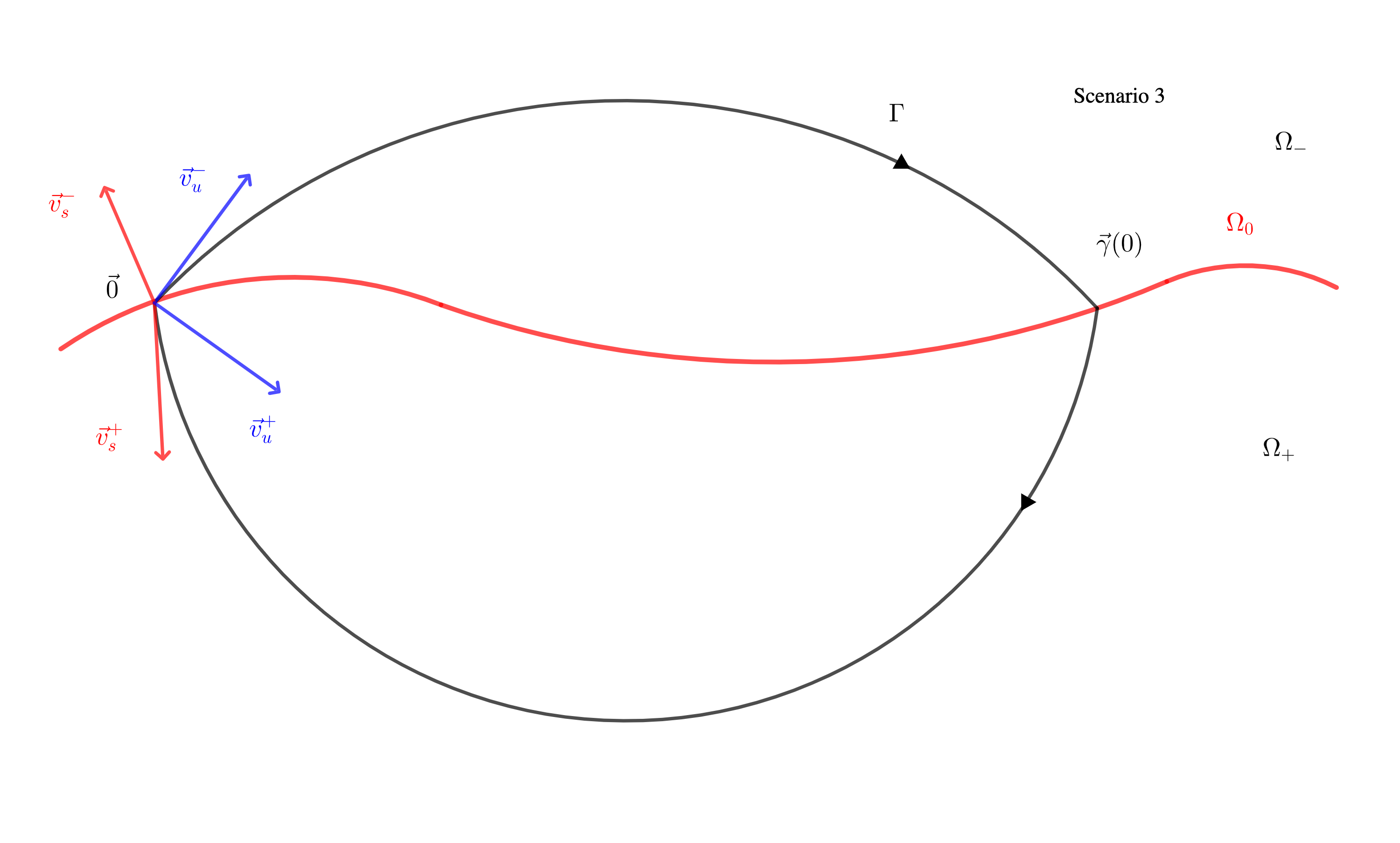, width = 8 cm}\\
	\epsfig{file=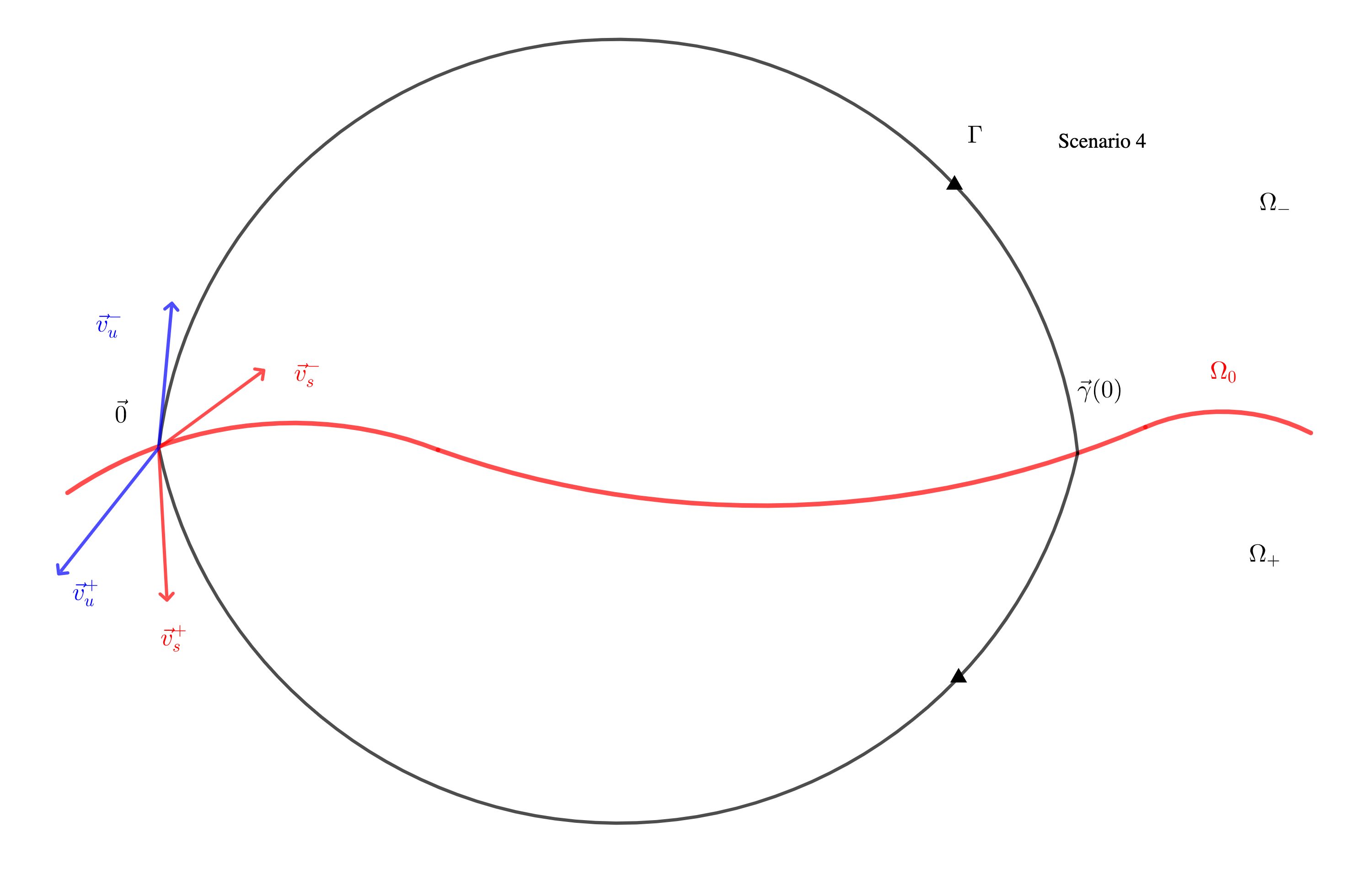, width = 8 cm}
\end{center}
\caption{Scenarios 3 and 4. In these settings we have persistence of the homoclinic trajectories  but  sliding might occur  close to the
origin. Here our
analysis does not apply directly. Further Melnikov theory guarantees persistence of the homoclinic trajectories,
  cf.\ \cite{CaFr}, but chaos is forbidden  in general, cf.~\cite{FrPo}.}
\label{scenario45}
\end{figure}

\section{Statement of the main results}\label{fractal.space}
In this section we state the main results proved in this article. First
 we collect here for convenience of the reader and future reference,  the main constants which
will play a role in our argument:
\begin{equation}\label{defsigma}
\displaystyle
\begin{array}{ccc}
\sfwd_+= \frac{|\la_s^+|}{\la_u^++|\la_s^+|}, & \sfwd_-= \frac{\la_u^-+|\la_s^-|}{\la_u^-}, &
\sfwd=  \sfwd_+ \sfwd_-,  \\
  \sbwd_+=   \frac{1}{\sfwd_+}, & \sbwd_- =\frac{1}{\sfwd_-}, &  \sbwd=   \sbwd_+ \sbwd_-, \\
 \underline{\sigma} = \min \{ \sfwd_+ , \sbwd_-  \} , &  & \overline{\sigma}= \max \{ \sfwd_+ , \sbwd_-  \},
  \end{array}
\end{equation}
 \begin{equation*}
\displaystyle
\begin{array}{ccc}
  \sTfwd_+= \frac{1}{\la_u^++|\la_s^+|}, & & \sTbwd_- = \frac{1}{\la_u^- +|\la_s^-|},
  \\
\sTfwd=\frac{\la_u^-+|\la_s^+|}{\la_u^-(\la_u^++|\la_s^+|)} ,&  & \sTbwd= \frac{\la_u^{-}+|\la_s^+|}{|\la_s^+|(\la_u^-+|\la_s^-|)},\\
\underline{\Sigma}= \min \{ \sTfwd , \sTbwd  \} , &  & \overline{\Sigma}= \max \{ \sTfwd , \sTbwd  \}.
\end{array}
\end{equation*}
 Notice that, in
particular,
$\underline{\sigma}\leq\overline{\sigma}<1$.

\begin{remark}
In the smooth setting  we have $\la_u^+=\la_u^-$, and $\la_s^+= \la_s^-$ so we have the following simplifications
\begin{equation}\label{defsigmabis}
\begin{array}{c}
\sfwd= \frac{|\la_s|}{\la_u}, \quad \sbwd= \frac{\la_u}{|\la_s|}; \qquad  \sTfwd= \frac{1}{\la_u} , \quad \sTbwd= \frac{1}{|\la_s|}.
  \end{array}
\end{equation}
\end{remark}

We denote by $\Z^-= \{k \in \Z \mid k \le -1\}$, $\Z^+= \{k \in \Z \mid k \ge 1\}$, and then by
$\EE^-= \{0,1 \}^{\Z^-}$, $\EE^+= \{0,1 \}^{\Z^+}$, respectively, the space of sequences   from   $\Z^-$,
$\Z^+$ to $ \{0,1 \}$.

  Let us define the Melnikov function $\mathcal{M}:\R\to\R$  which, for planar
 piecewise-smooth systems as \eqref{eq-disc}, takes the following form, see  Appendix, where we correct a small
 error appeared in \cite{CaFr}:

\begin{equation}\label{melni-disc}
\begin{split}
   \mathcal{M}(\al) &= c^-_{\perp}\int_{-\infty}^{0} \eu^{-\int_0^t
			\tr\boldsymbol{\f_x^-}(\ga^-(s))ds} \f^-(\ga(t))
		\wedge \g(t+\al,\ga^-(t),0) dt\\
		&\quad {}+ c^+_{\perp} \int_{0}^{+\infty} \eu^{-\int_0^t
			\tr\boldsymbol{\f_x^+}(\ga^+(s))ds} \f^+(\ga(t))
		\wedge \g(t+\al,\ga^+(t),0) dt,\\
 c_{\perp}^{\pm} &  =
         \frac{\|\na G(\ga(0))\| }{(\na G(\ga(0)))^*\f^\pm(\ga(0))} >0,
\end{split}
	\end{equation}
where ``$\wedge$" is the wedge product in $\R^2$ defined by $\vec{a} \wedge \vec{b}= a_1b_2-a_2b_1$
for any vectors $\vec{a}=(a_1,a_2)$, $\vec{b}=(b_1,b_2)$.
In fact, also in the piecewise-smooth case the function $\mathcal{M}$ is $C^r$.

  More precisely we will assume that the Melnikov function, defined in \eqref{melni-disc}, verifies the following hypothesis.
 \begin{description}
   \item[P1] There is a constant $\bar{c}>0$ and an increasing sequence $(b_i)$, $i \in \Z$, such that
   $b_{i+1}-b_i \ge 1/10$   and
   $$\MM(b_{2i})<-\bar{c}<0<\bar{c}<\MM(b_{2i+1}),$$
   for any $i \in \Z$.
 \end{description}
\begin{remark}
Notice that $b_i \to \pm \infty$ as $i \to \pm \infty$. Further the assumption $b_{i+1}-b_i \ge  1/10$ could be replaced by
$b_{i+1}-b_i \ge c$ where $c>0$ is a constant independent of $i$.
\end{remark}
If \assump{P1} holds, the intermediate value theorem implies that  $\MM(\tau)$ has at least one zero in $]b_k,b_{k+1}[$ for any $k
\in \Z$.

We begin by stating the consequences of our results for \eqref{eq-disc}, but replacing \assump{P1} by 
more restrictive
recurrence properties,
for clarity and in order to illustrate the novelties introduced by our approach   with more usual periodicity
assumptions.
 \begin{description}
   \item[P2] There are  $b_0$, $b_1$ such that $\MM(b_0)<0<\MM(b_1)$
    and $\g(t,\x,\ep)$ is almost periodic in $t$,  i.e., for any  $\varsigma>0$ there is  $\bar{L}=\bar{L}(\varsigma)>0$
   such that any interval of length $\bar{L}$ contains at least a, so called, quasi-period $\bar{N}$, such that
   $$ \|\g(t,\x,\ep)- \g(t+\bar{N},\x,\ep)\| \le \varsigma \,, \quad \textrm{for any } (t,\x,\ep) \in\R\times\Om\times\R.$$
     \item[P3]  There are $b_0$, $b_1$ such that $\MM(b_0)<0<\MM(b_1)$ and $\g(t,\x,\ep)$ is  $1$-periodic in $t$, i.e
   $$ \g(t,\x,\ep)= \g(t+1,\x,\ep)  \, ,\quad \textrm{for any } (t,\x,\ep) \in\R\times\Om\times\R.$$
 \end{description}

It is classically known that the periodicity and the almost periodicity of $\g$ imply the periodicity and the almost periodicity
of $\MM$, so   \assump{P3} implies \assump{P2}, which implies \assump{P1}.  Further we recall that quasi-periodicity implies
almost periodicity.

We want to consider an increasing sequence of times $\mathcal{T}=(T_j)$, $j \in \Z$ such that   $\MM(T_{2j})=0$
and  $T_{j+1}-T_j$ becomes larger and larger as $\ep \to 0$, see \eqref{TandKnu} and \eqref{TandKnunew} below.
Correspondingly, we find a subsequence $\beta_j:=b_{n_j}$ such that
\begin{equation}\label{defkj}
  \beta_j:=b_{n_j} <T_j < \beta'_j:= b_{n_j +1} \, , \qquad  B_{j}=\beta'_j- \beta_j= b_{n_j+1}-b_{n_j} >0,
\end{equation}
for any $j \in \Z$.

When dealing with assumption \assump{P1} we follow two different settings of assumptions;  the first is slightly more
restrictive but  includes the cases  where  $(B_j)$  is bounded, so in particular it is satisfied if we assume either  \assump{P2} or
\assump{P3}:  it allows
to obtain more precise
(and clearer) results,
i.e., Theorems \ref{restatefuture} and \ref{restatepast}; in the second approach we
ask for weaker assumptions but we obtain weaker results, i.e., Theorem \ref{restatebis} (where we have a very weak control of the size of
$\alpha^{\ee}_j(\ep)$ introduced below).

We consider now the first setting of assumptions.

We assume first that   there may be some $j \in \Z$ such that $T_{2j}$ is an accumulation point of the zeros of $\MM(\tau)$,
so there are  $0<\dd<1 $, $1/2>\l1 >2\lzero \ge 0$  and increasing sequences $(\aup_j)$, $(\adown_j)$,
$j\in \Z$   such that
$$\beta_{2j}<\aup_j<T_{2j}-\lzero<T_{2j}<T_{2j}+\lzero<\adown_{j}<\beta'_{2j}$$  and
\begin{equation}\label{minimal}
\begin{split}
 & |\MM(\aup_j)|=|\MM(\adown_j)|= \dd \bar{c}, \qquad \MM(\aup_j) \MM(\adown_j) <0, \\ & \MM(\tau) \ne 0 \quad \forall \tau \in [\aup_j,
  \adown_j] \setminus [T_{2j}-\lzero, T_{2j}+\lzero],\\
  & \textrm{$0<\adown_j-\aup_j \le \l1$,
for any $j \in \Z$.}
  \end{split}
\end{equation}
We emphasize that  both $\l1$ and $\lzero$ are independent of $j \in \Z$.

In fact \eqref{minimal} is enough to construct a chaotic pattern, but
we obtain sharper results  if we assume that there is a strictly monotone increasing and continuous function
$\omega_M(h):[0,\l1] \to [0,+\infty[$ (independent of $j$) such that
$\omega_M(0)=0$ and
\begin{equation}\label{isogen}
\begin{split}
      |\MM(T_{2j}- \lzero - h)| & \ge \omega_M(h)  \quad  \textrm{for any } \;  (T_{2j}- \lzero - h)\in  [\aup_j, T_{2j}- \lzero],  \\
  |\MM(T_{2j}+ \lzero + h)| & \ge \omega_M(h)  \quad  \textrm{for any } \;  (T_{2j}+ \lzero + h)\in  [ T_{2j}+ \lzero, \adown_j],
\end{split}
\end{equation}
and any $j \in \Z$.

Notice that \eqref{minimal} and \eqref{isogen} are always satisfied if \assump{P2} or \assump{P3} holds.
In fact \eqref{minimal} and \eqref{isogen} can be satisfied even if the sequence $(B_j)$  becomes unbounded, but they require  a control
from below on the ``slope" of the Melnikov function ``close to its zeros".

  We stress that in the easier and most significant case where $T_{2j}$ is an isolated zero for $\MM(\cdot)$ for any $j$ we
  can assume  $\lzero=0$ so that \eqref{minimal} simplifies as follows
  \begin{equation}\label{defbetaj}
\begin{split}
    &  |\MM(\aup_j)|=|\MM(\adown_j)|= \dd \bar{c}, \qquad \MM(\aup_j) \MM(\adown_j) <0, \\
     & \MM(\tau) \ne 0 \qquad \forall \tau \in [\aup_j,
  \adown_j] \setminus \{T_{2j} \}.
\end{split}
\end{equation}
  Further in this case   \eqref{isogen} reduces to
\begin{equation}\label{isolated}
  |\MM(T_{2j}+ h)| \ge \omega_M(|h|)  \quad  \textrm{for any }  (T_{2j}+ h)\in [\aup_j, \adown_j]\, , \quad h \ne 0,
\end{equation}
and any $j \in \Z$.

When we assume the classical hypothesis that
there is   $C >0$ such that
\begin{equation}\label{nondeg}
  \MM(T_{2j})=0   \quad   \text{and} \quad |\MM'(T_{2j})|>C  \;  \text{for any $j \in \Z$,}
\end{equation}
then \eqref{minimal} and \eqref{isolated} hold and we may assume  simply $ \omega_M(h) = \tfrac{C}{2} h$
 and $\l1 = 2\frac{\bar{c}}{C}  \delta<\tfrac{1}{2}$.

\begin{remark}\label{toaddinPrestate}
  Classically the recurrence conditions such as periodicity or almost periodicity, are required directly on $\g$ and
  then inherited by $\MM$. Here we prefer to ask for conditions on $\MM$ (which in general are more difficult to be verified on $\g$)
  in order to include the ``large perturbations" as $k(\tau)$ in \eqref{e.remP1}. We think this might be useful for application since
  it is possible that the perturbation $\g$ is made up by a periodic part of which we have a control and a ``noise", possibly not small, see
  Remark \ref{remP1} below
\end{remark}

Now we need to define the following absolute constants $K_0$ and $\nu_0$, which depend only on the eigenvalues in \assump{F0}:
  \begin{equation}\label{defK0-new}
  K_0:=    \frac{3\overline{\Sigma} }{2  \underline{\sigma}} , \qquad  \nu_0:= \max \left\{  3\overline{\sigma}-1 ; 1   \right\}.
\end{equation}
Following \cite{CFPRimut} we introduce  a further   parameter $\nu \ge \nu_0$  which is used to tune the distance of the points in
$\xxi \in \aleph^+$ from $\vec{P}_s(\tau)$, paying the price of a longer time needed by the trajectory
$\x(t,\tau; \xxi)$   to perform a loop in forward time (and analogously for $\xxi \in \aleph^-$ in backward time).

When  \assump{P1} and \eqref{minimal} hold we need the following condition concerning the sequence $(T_j)$ and the time gap
$T_{j+1}-T_j$:
  \begin{equation}\label{TandKnu}
\begin{aligned}
   & \MM(T_{2j})=0,  \qquad
  T_{j+1}-T_j >   \l1 + K_0(1+\nu)   |\ln(\ep)| , \quad \textrm{ for any $j \in \Z$}.
  \end{aligned}
\end{equation}

\begin{remark}
 If \assump{P3} holds and there is $t_0$ such that $\MM(t_0)=0 \ne \MM'(t_0)$,
  then we can choose  $\TT=(T_j)$, $T_j=t_0+j \ogap$, where $\ogap=     \left[  K_0(1+\nu) |\ln(\ep)|    +2 \right]$, where $[\cdot]$
  denotes the integer part,
  so that it satisfies \eqref{TandKnu} and \eqref{nondeg}.
\end{remark}

In the second approach we drop \eqref{minimal}, but
we need to ask for a slightly more restrictive condition on  the time gap
$T_{j+1}-T_j$, i.e.:
\begin{equation}\label{TandKnunew}
\begin{split}
  &  \MM(T_{2j})=0,
 \qquad  T_{j+1}-T_{j} >   \max\{  B_{j+1}; B_{j}   \} + K_0(1+\nu)   |\ln(\ep)| ,
  \end{split}
\end{equation}
for any $j \in \mathbb{Z}$ and $\nu \ge \nu_0$.

 Let us denote by $\tilde{V}(\tau)$   the compact connected set enclosed by
$\tilde{W}(\tau)$ and by the branch of $\Om^0$ between $\P_u(\tau)$ and $\P_s(\tau)$.
Now we are ready to state the main results of the paper, i.e., Theorems \ref{restatefuture}, \ref{restatepast},
 \ref{restatebis} and Corollaries  \ref{r.allbis}, \ref{c.localization},  \ref{localize1};
 the proofs of all these results are postponed to Section \ref{S.pf.main}. More precisely Theorem  \ref{restatefuture} is proved in \S
 \ref{proof.futurebis},
 Theorem  \ref{restatebis} \textbf{a)} is proved at the end of \S \ref{proof.future}, while Theorem  \ref{restatebis} \textbf{b)} together
 with all the other
 results
 are proved  in \S \ref{proof.past}.

\begin{theorem}\label{restatefuture}
	Assume  that  $\f^{\pm}$ and $\g$ are $C^r$, $r>1$ and that \assump{F0}, \assump{F1}, \assump{F2}, \assump{K} and \assump{G} hold true;
assume further
\assump{P1},
and fix $\nu \ge \nu_0$ for $\nu_0$ as in \eqref{defK0-new} and   $\tau \in [b_0, b_1]$.
Then we can choose $\ep_0$ small enough so that for any $0< \ep \le \ep_0$
and any increasing sequence $\mathcal{T}^+=(T_j)$ satisfying
 \eqref{TandKnu} for $j \in \Z^+$ and
\begin{equation}\label{TandKnu_tau+}
	 T_1-b_1 >   \l1 + K_0(1+\nu)   |\ln(\ep)| ,
\end{equation}
there is $c^*>0$ so that the following holds.
   For any sequence
   $\ee^+ \in \EE^+$ there is a compact set $X^+= X^+(\ee^+,\tau, \TT^+)$ and a sequence $\alpha_j(\ep)= \al_j^{\ee^+}(\ep,\tau,\TT^+)$
such that   for any $\xxi \in X^+$
the trajectory $\x(t,\tau ;  \xxi) \in \tilde{V}(t)$ for any $t \ge \tau$
  and satisfies the property $\bs{C_{\ee^+}^+}$, i.e.
 \begin{description}
  \item[$\bs{C_{\ee^+}^+}$]
	if $\ee_j=1$ then
	\begin{equation}\label{ej+=1w}
	\begin{split}
	&   \|\x(t,\tau; \xxi )-\ga(t-T_{2j}-\al_j(\ep)) \| \le c^*\ep \quad \textrm{when $t \in [T_{2j-1}, T_{2j+1}]$},
	\end{split}
	\end{equation}
	while if $\ee_j=0$ we have
	\begin{equation}\label{ej+=0w}
	\|\x(t,\tau; \xxi) \| \le c^*\ep  \quad \textrm{when $t \in [T_{2j-1}, T_{2j+1} ] $}
	\end{equation}
	for any $j \in \Z^+$. Further
	\begin{equation}\label{ej+=tau}
	\|\x(t,\tau; \xxi)-\ga(t-\tau) \| \le c^*\ep  \quad \textrm{when $t \in [\tau, T_1] $}.
	\end{equation}
\end{description}
  Moreover for any $j \in \mathbb{Z}^+$ we have the following estimates
    \begin{description}
    \item[a)] if \eqref{nondeg} holds and $r \ge 2$   there is $c_{\alpha}> 0$  such that $|\al_{j}^{\ee^+}(\ep)| \le c_{\alpha}\ep$;
    \item[b)] if \eqref{isolated} holds and $r>1$ then
    $$|\al_{j}^{\ee^+}(\ep)| \le \omega_{\alpha} (\ep) ,$$
    where $\omega_{\alpha}(\cdot)$ is an increasing continuous function such that $\omega_{\alpha}(0)=0$;
    \item[c)] if \eqref{minimal} holds and $r>1$ then $|\al_{j}^{\ee^+}(\ep)| \le \l1$.
  \end{description}
  The constants $\ep_0$, $c^*$, $c_{\alpha}$  and the function  $\omega_{\alpha}$ are  independent of
   $\ep$,
    $\nu$, $\tau$, $\TT^+$,
  $\ee^+$.
\end{theorem}

\begin{theorem}\label{restatepast}
	Assume that   $\f^{\pm}$ and $\g$ are $C^r$, $r>1$ and that \assump{F0}, \assump{F1}, \assump{F2}, \assump{K} and \assump{G} hold true;
assume further
\assump{P1}
and fix $\nu \ge \nu_0$ for $\nu_0$ as in \eqref{defK0-new} and $\tau \in [b_0 , b_1]$.
Then we can choose $\ep_0$ small enough so that for any $0< \ep \le \ep_0$
and any increasing sequence $\mathcal{T}^-=(T_j)$  satisfying
 \eqref{TandKnu} for $j \le -2$ and
\begin{equation}\label{TandKnu_tau-}
	  b_0 -T_{-1} >   \l1 + K_0(1+\nu)   |\ln(\ep)| ,
\end{equation}
there is $c^*>0$ so that the following holds.
   For any sequence
   $\ee^-\in \EE^-$ there is a compact set $X^-= X^-(\ee^-,\tau, \TT^-)$ and a sequence $\alpha_j(\ep)= \al_j^{\ee^-}(\ep,\tau,\TT^-)$
such that   for any $\xxi \in X^-$
   the trajectory $\x(t,\tau ;  \xxi) \in \tilde{V}(t)$ for any $t \le \tau$  and satisfies   the property $\bs{C_{\ee^-}^-}$, i.e.
 \begin{description}
  \item[$\bs{C_{\ee^-}^-}$]
	if $\ee_j=1$, then
	\begin{equation}\label{ej-=1w}
	\begin{split}
	&   \|\x(t,\tau; \xxi )-\ga(t-T_{2j}-\al_j(\ep)) \| \le c^*\ep \quad \textrm{when $t \in [T_{2j-1}, T_{2j+1} ]$},
	\end{split}
	\end{equation}
	while if $\ee_j=0$, we have
	\begin{equation}\label{ej-=0w}
	\|\x(t,\tau; \xxi) \| \le c^*\ep  \quad \textrm{when $t \in [T_{2j-1}, T_{2j+1} ] $}
	\end{equation}
	for any $j \in \Z^-$. Further
	\begin{equation}\label{ej-=tau}
	\|\x(t,\tau; \xxi)-\ga(t-\tau) \| \le c^*\ep  \quad \textrm{when $t \in [T_{-1}, \tau ] $}.
	\end{equation}
\end{description}
  Moreover for any $j \in \mathbb{Z}^-$ we have the following estimates
    \begin{description}
    \item[a)] if \eqref{nondeg} holds and $r \ge 2$,  there is $c_{\alpha}> 0$  such that $|\al_{j}^{\ee^-}(\ep)| \le
        c_{\alpha} \ep$;
    \item[b)] if \eqref{isolated} holds and $r>1$, then
    $$|\al_{j}^{\ee^-}(\ep)| \le \omega_{\alpha} (\ep),$$
    where $\omega_{\alpha}(\cdot)$ is an increasing continuous function such that $\omega_{\alpha}(0)=0$;
    \item[c)] if \eqref{minimal} holds and $r>1$ then $|\al_{j}^{\ee^-}(\ep)| \le  \l1 $.
  \end{description}
  The constants $\ep_0$, $c^*$, $c_{\alpha}$ and the function $\omega_{\alpha}$ are independent of
   $\ep$,
    $\nu$, $\tau$, $\TT^-$,
  $\ee^-$.
\end{theorem}

\begin{theorem}\label{restatebis}\
\begin{description}
  \item[\textbf{a)}]  Assume that the hypotheses of Theorem \ref{restatefuture} hold  but
replace \eqref{TandKnu}  by \eqref{TandKnunew}  and   \eqref{TandKnu_tau+} by
\begin{equation}\label{TandKnunew_tau+}
	T_1-b_1 >  \max\{  B_{1}; B_{0}   \} + K_0(1+\nu)   |\ln(\ep)| ,
\end{equation}
      then we obtain the same result as in Theorem \ref{restatefuture} but we just have
     $$|\al_{j}^{\ee^+}(\ep)| \le B_{2j}.$$
  \item[\textbf{b)}] Assume that the hypotheses of Theorem
    \ref{restatepast} hold but
replace \eqref{TandKnu}  by \eqref{TandKnunew}  and   \eqref{TandKnu_tau-} by
	\begin{equation}\label{TandKnunew_tau-}
		b_0-T_{-1} >  \max\{  B_{0}; B_{-1}   \} + K_0(1+\nu)   |\ln(\ep)| ,
\end{equation}
      then we obtain the same result as in Theorem \ref{restatepast} but we just have
     $$|\al_{j}^{\ee^-}(\ep)| \le B_{2j}.$$
\end{description}
     \end{theorem}

 In Theorems \ref{restatefuture} and \ref{restatepast} we can suppress the dependence on $\al_j^{\ee^{\pm}}(\ep)$  of the relations in
\eqref{ej+=1w},
\eqref{ej-=1w},
possibly paying the price of a further loss of precision of the estimates.
\begin{cor}\label{r.allbis}
Let the assumptions of Theorems \ref{restatefuture}  and \ref{restatepast} hold.

If \textbf{a)} holds then we can replace  \eqref{ej+=1w},
\eqref{ej-=1w}, by
\begin{equation*} 
	\begin{split}
	&   \|\x(t,\tau; \xxi)-\ga(t-T_{2j}) \| \le \tilde{c}^* \ep \quad \textrm{when $t \in [T_{2j-1}, T_{2j+1} ]$}
	\end{split}
	\end{equation*}
 where  $\tilde{c}^* $ is the constant in \eqref{deftildecstar}.

If \textbf{b)} holds then we can replace  \eqref{ej+=1w},
\eqref{ej-=1w}, by
	\begin{equation*} 
	\begin{split}
	&   \|\x(t,\tau; \xxi)-\ga(t-T_{2j}) \| \le \omega(\ep) \quad \textrm{when $t \in [T_{2j-1}, T_{2j+1} ]$}
	\end{split}
	\end{equation*}
where  $\omega(\ep)= c^*\ep+ \tilde{c}\omega_{\al} (\ep)$  and $\tilde{c} >0$ is a constant,  cf.~\eqref{tildecstarbis}.
\end{cor}

\begin{remark}\label{r.unique}
  Notice that if we have  two different sequences $\hat{\ee}^{\pm}$ and $\check{\ee}^{\pm}$ in $\EE^\pm$ then
  $X^{\pm}(\hat{\ee}^{\pm}, \tau, \TT^{\pm}) \cap X^{\pm}(\check{\ee}^{\pm}, \tau, \TT^{\pm})= \emptyset$.
  We conjecture that in the setting of Theorem \ref{restatefuture}  \textbf{a)}  and Theorem \ref{restatepast} \textbf{a)}
   $X^{+}(\ee^+, \tau, \TT^{+})$ and  $X^{-}(\ee^-, \tau, \TT^{-})$ are singletons for any $\ee^+ \in \EE^+$ and any
   $\ee^- \in \EE^-$.
\end{remark}

  \begin{remark}\label{remP1}
We stress that \assump{P1} does not require an upper bound on $B_j$ when $|j| \to +\infty$.
Further, the self-similarity required on $\MM$ is very weak and we do not need any non-degeneracy of the zeros of the Melnikov function
$\MM(\tau)$, which may be null in some interval. Hence with our results we can deal with functions like, e.g.:
\begin{equation}\label{e.remP1}
\begin{split}
&     \MM_1(\tau)= 3\sin ( \tau)+ k(\tau), \\
    & \MM_2(\tau)= 3\sin (\sqrt[3]{1+\tau^2})+ k(\tau)
\end{split}
\end{equation}
where $|k(\tau)| \le 2$ is arbitrary.
\\
Notice that $\MM_1(\tau)$  satisfies \eqref{minimal} and \eqref{isogen},
  while
$\MM_2(\tau)$ does not even satisfy
\eqref{minimal}. So in the former case
we can apply Theorems \ref{restatefuture} and \ref{restatepast} while in the latter just Theorem \ref{restatebis}.

Further notice that if we assume $k(\tau) \equiv 0$, in the former case we can apply Corollary \ref{r.allbis} since
the hypotheses of Theorems \ref{restatefuture} \textbf{a)} and \ref{restatepast} \textbf{a)} are satisfied,
 while in the latter we can apply again just Theorem \ref{restatebis}, since we do not have a
uniform lower bound on     $|\MM'(T_j)|$.

 We remand the reader to \cite[\S 7]{CFPdiscAr} for the construction of examples of systems of type \eqref{eq-disc}
such that the corresponding Melnikov function $\MM$ is not even (a priori) almost periodic but has properties similar  to the functions
$\MM_1$ and $\MM_2$ in \eqref{e.remP1}.
\end{remark}

 Via Theorem \ref{restatefuture} or \ref{restatebis} \textbf{a)} we define  the sets
\begin{equation}\label{XandSigma+}
\begin{split}
    \aleph^+(\tau, \TT^+):= &\cup_{\ee^+ \in \EE^+}X^+( \ee^+, \tau, \TT^+).
\end{split}
\end{equation}
Analogously via Theorems \ref{restatepast} or \ref{restatebis} \textbf{b)}
we define  the sets
\begin{equation}\label{XandSigma-}
   \aleph^-(\tau, \TT^-):= \cup_{\ee^- \in \EE^-}X^-( \ee^-, \tau, \TT^-).
\end{equation}

\begin{remark}\label{r.spatial.mult1}
 Assume that we are in the setting either of Theorem \ref{restatefuture} or of Theorem \ref{restatebis} \textbf{a)} but
 assume that the system is smooth.
 Then we can replace $\Omega^0$ by any line (or curve) $\Xi$ transversal to $\bs{\Gamma} \setminus \{\vec{0} \}$ (cf. point 1 of
 Introduction),
 so for any fixed couple $(\tau, \TT^+)$ we find uncountably many distinct chaotic sets $\aleph^+(\tau, \TT^+)$,
 one for each transversal $\Xi$. The analogous result holds for $\aleph^-(\tau, \TT^-)$ when either the assumptions of Theorem
 \ref{restatepast} or of  Theorem  \ref{restatebis} \textbf{b)} are satisfied
\end{remark}
\begin{remark}\label{r.spatial.mult2}
 Assume that we are in the setting either of Theorem \ref{restatefuture} or of Theorem \ref{restatebis} \textbf{a)}, this time
 allow the system to be non-smooth.
 For any $T \in \R$ we can define
 $$\aleph^+_T(\tau, \TT^+):= \{ \x(T,\tau; \xi) \mid \xi \in \aleph^+(\tau, \TT^+)\}.$$
 Now fix $T \in \R$ and a point $\vec{\zeta}   \in \bs{\Gamma}$, say $\vec{\zeta}=\ga (s) $, $s \in \R$.
 Then the set $\aleph^+_T(T-s, \TT^+)$ is $O(\ep)$ close to $\vec{\zeta}$.
 This way for any fixed $T\in\R$ we construct uncountably many distinct sets $\aleph^+_T(T-s, \TT^+)$
 parametrized by $s$,
 each of them $O(\ep)$ close to  a point $\zeta= \ga(s) \in \bs{\Gamma}$.
 An analogous result holds for $\aleph^-$.
In this way we extend the results of Remark \ref{r.spatial.mult1} to the piecewise-smooth but discontinuous case.
\end{remark}

\begin{cor}\label{c.localization}
Let either the assumptions of Theorems \ref{restatefuture} and \ref{restatepast} or the assumptions of Theorem \ref{restatebis} hold. Then
  the sets $\aleph^{\pm}(\tau,\mathcal{T}^{\pm})= \cup_{\ee^{\pm} \in \EE^{\pm}}
  X^{\pm}(\ee^{\pm},\tau,\mathcal{T}^{\pm})$
  are compact   and $\aleph^{\pm}(\tau,\mathcal{T}^{\pm}) \subset L^0(c^* \ep)$. Further
\begin{equation}\label{eq.diam}
 \textrm{diam}(\aleph^{\pm}):= \sup \left\{ \|\Q_2- \Q_1 \| \mid \Q_1, \Q_2 \in \aleph^{\pm}(\tau, \TT^{\pm}) \right\} \le
\ep^{\frac{1+\nu}{\und{\sigma}}}.
\end{equation}
Moreover $\P_s(\tau)$    is an extremal point in $\aleph^+(\tau, \TT^+)$, while
$\P_u(\tau)$    is an extremal point in $\aleph^-(\tau, \TT^-)$, and   both  $\P_s(\tau)$ and $\P_u(\tau)$  correspond  to the null
sequence.
\end{cor}
From Corollary \ref{c.localization} we see that the sets $\aleph^{\pm}(\tau, \TT^{\pm})$ are located   in a one-dimensional set.
  Further they are contained in a set whose diameter   becomes arbitrarily small  as $\nu$ increases,
  leaving unaltered the size $\ep$ of the perturbation.  The drawback is that the minimum
  gap $T_{j+1}-T_j$ increases
  linearly with $\nu$.
However the positions of $\P_{s}(\tau)$ and $\P_{u}(\tau)$
are known just with a precision of order $O(\ep)$ since they both oscillate in $L^0(c^* \ep)$.

 Further, in the setting of Theorems \ref{restatefuture} and \ref{restatepast} the flow of \eqref{eq-disc} on $\aleph^+$ and on $\aleph^-$ is
 semi-conjugated
 to the Bernoulli shift on
$\EE^+$ and on $\EE^-$ respectively, see Theorem
\ref{thm.bernoulli}.

In fact we get an even better localization of the initial conditions giving rise to chaos if we evaluate them at $t=T_1$.

 \begin{cor}\label{localize1}
  Let either the assumptions of Theorem \ref{restatefuture} or the assumptions of Theorem \ref{restatebis} \textbf{a)} be satisfied;
   then
$$\aleph^{+}_{1} := \{\x(T_1, \tau ; \xxi) \mid \xxi \in \aleph^+(\tau, \TT^+) \}   \subset \left( B(\vec{0},  \ep^{ \frac{\nu+1}{2}})\cap
\tilde{V}(T_1)
\right).$$
Analogously
let either the assumptions of Theorem \ref{restatepast} or the assumptions of Theorem \ref{restatebis} \textbf{b)} be satisfied;
   then
$$\aleph^{-}_{-1} := \{\x(T_{-1}, \tau ; \xxi) \mid \xxi \in \aleph^-(\tau,\TT^-) \}   \subset \left( B(\vec{0},  \ep^{ \frac{\nu+1}{2}})\cap
\tilde{V}(T_{-1})
\right).$$
\end{cor}
The proof of this result is postponed to the end of Section \ref{S.pf.main}.

As far as we are aware these facts are new even when $\f$ is smooth
 and stress the sensitive dependence of this perturbed equation on initial conditions.

We close this section  by observing that, in the smooth context,   assumption \assump{G} is not
restrictive.
In fact we may replace it by
\begin{description}
	\item[\assump{G'}]
	there is $C_g>0$ such that $\|\g(t,\vec{0},\ep)\| \le C_g$ for any $t \in \R$ and any $0\le \ep \le 1$.
\end{description}
Note that, in the smooth case and under condition \assump{G'},
by standard arguments relying on the exponential
dichotomy theory, see e.g.\ \cite[\S 4.2]{CDFP},
it can be proved that
 \eqref{eq-smooth}  admits a unique solution, say $\x_b(t,\ep)$, which is bounded for any $t \in \R$ and such that $\|\x_b(t,\ep)\|=O(\ep)$
uniformly in $t \in\R$.
In fact $\x_b(t,\ep)$ emanates from the origin and,  roughly speaking, replaces its role; i.e., $\ga(t)$ is perturbed on a trajectory
homoclinic to
$\x_b(t,\ep)$ as $|t| \to \infty$. Further  we have the following.

\begin{remark} \label{r.translated}
In the smooth case, replace \assump{G} by \assump{G'}; then we can still apply our methods to  \eqref{eq-smooth}, obtaining a result
analogous
to
Corollary \ref{localize1}.
\end{remark}

\begin{proof}[Proof of Remark \ref{r.translated}]
If \assump{G'} holds we may set
$\y(t,\ep)=\x-\x_b(t,\ep)$ so that  \eqref{eq-smooth}
is changed into
\begin{equation}\label{eq-translated}
   \dot{\y}= \f(\y)+\ep\g_T(t,\y,\ep)
\end{equation}
where
\begin{equation*}
  \begin{split}
     \g_T(t,\y,\ep):= & \g(t,\y+ \x_b(t,\ep),\ep)- \g(t,\x_b(t,\ep),\ep) \\
       & + \frac{\f(\y+ \x_b(t,\ep))- \f(\y)- \f(\x_b(t,\ep))}{\ep}.
  \end{split}
\end{equation*}
Using the fact that $\x_b(t,\ep)/\ep$ is uniformly bounded, one can check that $\g_T(t,\y,\ep)$ is bounded
when $t  \in \R$, $\y$ is in a compact neighborhood of $\bs{\Gamma}$ and $\ep \in [0,1]$. Further $\g_T(t,\vec{0},\ep) \equiv 0$
so \eqref{eq-translated} satisfies \assump{G}. So we can apply the results of this section;  then going back to the original coordinates we
prove the
remark.
\end{proof}
 Classically with the change of variable used in Remark \ref{r.translated} we have a loss of regularity
 with respect to the $\ep$ variable, which is important to obtain $C^{r-1}$ functions $\alpha_j(\ep)$. However our approach  does not
 allow a good control on the functions $\alpha_j(\ep)$ which are at most continuous and in fact  usually they are not uniquely defined,
  see Theorems
 \ref{restatefuture},  \ref{restatepast} and  \ref{restatebis}; so we have no problems with the loss of regularity and we can ask for $r>1$
 instead of $r \ge
 2$.

\section{Construction of the  Poincar\'e  map}
In this section we borrow the results of \cite{CFPRimut} and we construct a  Poincar\'e  map from a subset of $L^0$ back to $L^0$
both   forward and backward in time, see Figure~\ref{nextstable}.

\begin{figure}[t]
\centerline{\epsfig{file=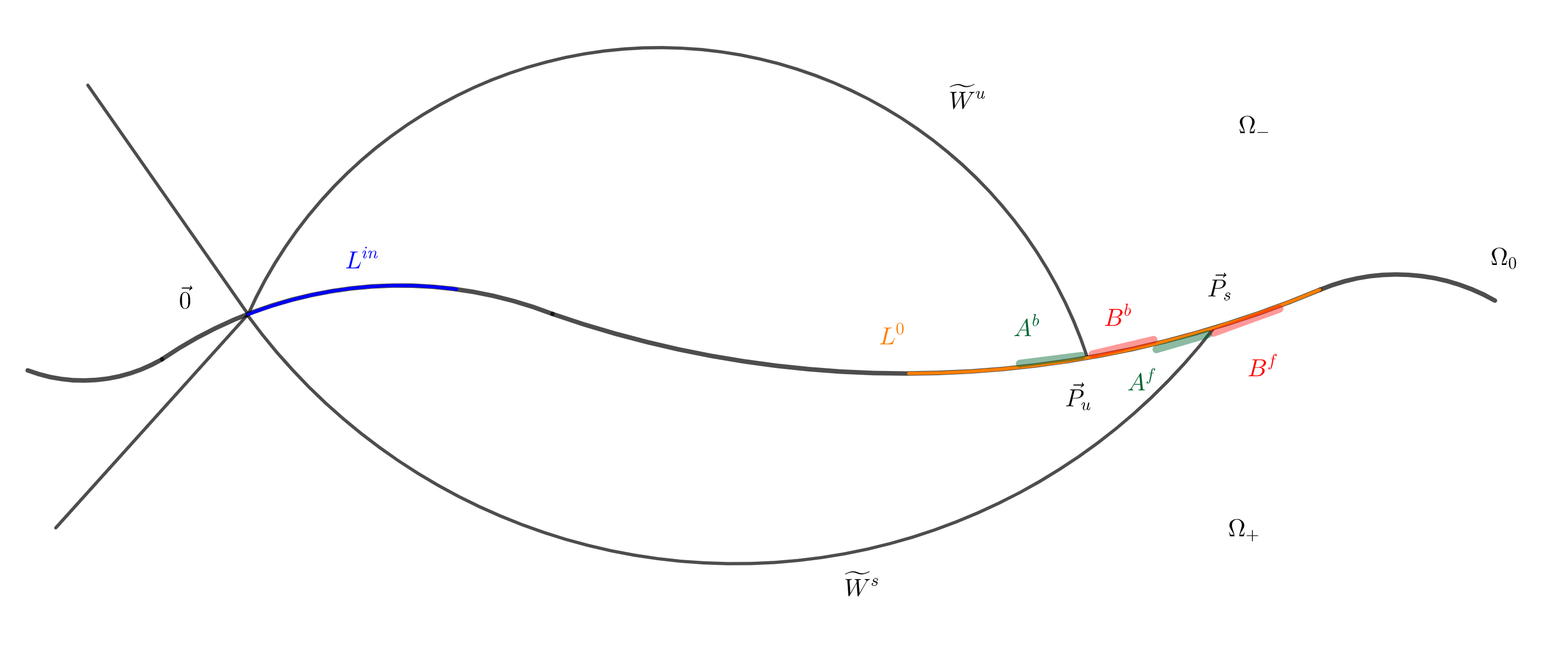, width = 10 cm}}
\caption{An explanation of the maps $\mathscr{P}^{\textrm{fwd}}(\cdot, \tau)$ and $\mathscr{P}^{\textrm{bwd}}(\cdot, \tau)$. Here the
dependence on $\tau$ of the points and sets is left unsaid.}
\label{nextstable}
\end{figure}

Fix $\tau \in \R$ and consider $L^0(\delta)$  given by \eqref{L0}, where $\delta>0$ is a small parameter, independent of $\ep>0$.

The point $\P_s(\tau)$ ($\P_u(\tau)$) splits $L^0$ in two parts, say $A^{\textrm{f}}(\tau)$ and $B^{\textrm{f}}(\tau)$
($A^{\textrm{b}}(\tau)$ and
$B^{\textrm{b}}(\tau)$), respectively
 ``inside'' and ``outside''.
Following \cite{CFPRimut} we define a $C^r$ Poincar\'e map
using the flow of \eqref{eq-disc} from $A^{\textrm{f}}(\tau)$ back to $L^0$ remaining close to $\bs{\Gamma}$; i.e.\
$\mathscr{P}^{\textrm{fwd}}(\cdot, \tau): A^{\textrm{f}}(\tau) \to L^0$ and a $C^r$ time map $\mathscr{T}^{\textrm{fwd}}(\cdot,\tau):
A^{\textrm{f}}(\tau) \to \R$ such that
for any $\P \in A^{\textrm{f}}(\tau)$ the trajectory
$\x(t,\tau; \P)$ will stay close to $\bs{\Gamma}$  (in fact close to $\tilde{W}(t)$) for any
$t \in [\tau, \mathscr{T}^{\textrm{fwd}}(\P,\tau)]$ and it will cross transversely
$L^0$ for the first time at $t=\mathscr{T}^{\textrm{fwd}}(\P,\tau)>\tau$ in the point $\mathscr{P}^{\textrm{fwd}}(\P, \tau)\in
A^{\textrm{b}}(\mathscr{T}^{\textrm{fwd}}(\P,\tau))$.
\\
Notice that if $\P \in B^{\textrm{f}}(\tau)$ then there is some
 $\mathscr{T}^{\textrm{out}}=\mathscr{T}^{\textrm{out}}(\P,\tau)>\tau$ such that the trajectory will leave a neighborhood of $\bs{\Gamma}$ at
 $t \ge
 \mathscr{T}^{\textrm{out}}$.

Using the flow of \eqref{eq-disc} (but now going backward in time) we   construct a  $C^r$ Poincar\'e map
$\mathscr{P}^{\textrm{bwd}}(\cdot, \tau): A^{\textrm{b}}(\tau) \to L^0$ and a $C^r$ time  map $\mathscr{T}^{\textrm{bwd}}(\cdot,\tau):
A^{\textrm{b}}(\tau) \to \R$
such that
for any $\P \in A^{\textrm{b}}(\tau)$ the trajectory
$\x(t,\tau; \P)$ will stay close to $\bs{\Gamma}$ (in fact  close to $\tilde{W}(t)$) for any
$t \in [\mathscr{T}^{\textrm{bwd}}(\P,\tau), \tau]$ and it will cross transversely
$L^0$ for the first time at $t=\mathscr{T}^{\textrm{bwd}}(\P,\tau)<\tau$ in the point $\mathscr{P}^{\textrm{bwd}}(\P, \tau)\in
A^{\textrm{f}}(\mathscr{T}^{\textrm{bwd}}(\P,\tau))$.
\\
Again if $\P \in B^{\textrm{b}}(\tau)$ then there is some $\mathscr{T}^{\textrm{out}}(\P,\tau)<\tau$
 such that the trajectory will leave a neighborhood of $\bs{\Gamma}$ at $t \le \mathscr{T}^{\textrm{out}}$.

    \begin{lemma}\label{L.loop}\cite[Lemma 3.6]{CFPRimut}
  Assume \assump{F0}, \assump{F1}, \assump{F2}, \assump{K}, \assump{G}.
    Let $\Q \in A^{\textrm{f}}(\tau)$. Then there are $\mathscr{T}^{\textrm{fwd}}(\Q,\tau)>\tau_1(\Q,\tau)>\tau$ such that
    the trajectory $\x(t,\tau; \Q) $
      crosses transversely $\Om^0$ at $t\in\{\tau, \tau_1(\Q,\tau), \mathscr{T}^{\textrm{fwd}}(\Q,\tau)\}$.
     Hence,
     $$\mathscr{P}^{\textrm{fwd}}_+(\Q,\tau):=\x(\tau_1(\Q,\tau),\tau; \Q) \in L^{\inn},$$
     $$\mathscr{P}^{\textrm{fwd}}(\Q,\tau):=\x(\mathscr{T}^{\textrm{fwd}}(\Q,\tau),\tau; \Q) \in
     A^{\textrm{b}}(\mathscr{T}^{\textrm{fwd}}(\Q,\tau)).$$

    Analogously, let $\Q \in A^{\textrm{b}}(\tau)$. Then there are $\mathscr{T}^{\textrm{bwd}}(\Q,\tau)<\tau_{-1}(\Q,\tau)<\tau$ such that
    $\x(t,\tau; \Q) $,
    crosses transversely $\Om^0$ at $t\in\{\tau, \tau_{-1}(\Q,\tau), \mathscr{T}^{\textrm{bwd}}(\Q,\tau)\}$. Hence
     $$ \mathscr{P}^{\textrm{bwd}}_-(\Q,\tau):=\x(\tau_{-1}(\Q,\tau),\tau; \Q) \in L^{\inn},$$
     $$ \mathscr{P}^{\textrm{bwd}}(\Q,\tau):=\x(\mathscr{T}^{\textrm{bwd}}(\Q,\tau),\tau; \Q) \in
     A^{\textrm{f}}(\mathscr{T}^{\textrm{bwd}}(\Q,\tau)).$$

     Further the functions $\mathscr{P}^{\textrm{fwd}}_+(\Q,\tau)$, $\mathscr{P}^{\textrm{fwd}}(\Q,\tau)$, $
     \mathscr{P}^{\textrm{bwd}}_-(\Q,\tau)$
     and $ \mathscr{P}^{\textrm{bwd}}(\Q,\tau)$ are $C^r$ in both the variables.
  \end{lemma}
The smoothness of the functions in Lemma \ref{L.loop} follows from the following observation, borrowed again from
\cite{CFPRimut}.

\begin{remark}\label{data.smooth}
  Let $A$ be an open, connected and bounded subset of $\Omega$, let $\tau_2>\tau_1$ and
  denote by
  $$B(t)= \{ \x(t,\tau_1; \Q) \mid \Q \in A \}.$$
  Assume that in $\Omega^0 \cap B(t)$ there are no sliding phenomena for any $t \in [\tau_1, \tau_2]$.
  Then the functions
  \begin{gather*}
  	\Phi_{\tau_2,\tau_1} : A \to B(\tau_2), \qquad \quad  \Phi_{\tau_1,\tau_2} :   B(\tau_2)\to A,\\
  	\Phi_{\tau_2,\tau_1}(\Q)= \x(\tau_2,\tau_1; \Q)
  ,\qquad \Phi_{\tau_1,\tau_2}=\Phi_{\tau_2,\tau_1}^{-1}
  \end{gather*}
  are  homeomorphisms.

  Assume further that $A \cap \Omega^0=\emptyset$, $B(\tau_2) \cap \Om^0= \emptyset$, and
  that for any $\Q\in A$, if $\x(\bar{t},\tau_1; \Q) \in \Om^0$ for some $\bar{t} \in (\tau_1, \tau_2)$, then it crosses $\Om^0$
  transversely. Then $\Phi_{\tau_2,\tau_1} $ and $\Phi_{\tau_1,\tau_2}$  are $C^r$ diffeomorphisms.
\end{remark}

\begin{remark}\label{r.in1}
Let us recall that $\tilde{V}(\tau)$ is the compact connected set enclosed by
$\tilde{W}(\tau)$ and by the branch of $\Om^0$ between $\P_u(\tau)$ and $\P_s(\tau)$, and denote by
$\tilde{V}^-(t)=\tilde{V}(t) \cap \Om^-$ and by $\tilde{V}^+(t)=\tilde{V}(t) \cap \Om^+$.
If $\Q \in A^{\textrm{f}}(\tau)$, then $\x(t,\tau; \Q) \in \tilde{V}^+(t)$ for any
$t \in  ]\tau, \tau_1(\Q,\tau)[$.
\\
Analogously if $\Q \in A^{\textrm{b}}(\tau)$, then $\x(t,\tau; \Q) \in \tilde{V}^-(t)$ for any
$t \in  ] \tau_{-1}(\Q,\tau), \tau[$.
\end{remark}

Now, following \cite[\S 4]{CFPRimut}, we estimate the space displacement with respect to $\tilde{W}(\cdot)$ and the fly time
of the maps introduced in Lemma \ref{L.loop}.

For this purpose we need to define a directed distance in $\Om^0$ using arc length, this is possible since $\Om^0$ is a regular curve.
So, for any $\Q\in L^0$ we define
$\ell(\Q)=\int_{\Om^0(\vec{0},\Q)}ds>0$ where $\Om^0(\vec{0},\Q)$ is the (oriented) path of $\Om^0$ connecting $\vec{0}$ with $\Q$,
and we define the directed distance
\begin{equation}\label{dist}
    \dist(\Q,\P):= \ell(\P)-\ell(\Q)
\end{equation}
for $\Q,\P\in L^0$.
Notice that $\dist(\Q,\P)>0$
 means that $\Q$ lies on $\Om^0$ between $\vec{0}$ and $\P$.
Now, we introduce some further crucial notation.

 \textbf{Notation.}
We denote by $\Q_s(d,\tau)$ the point in $A^{\textrm{f}}(\tau)$ such that
\[
\dist(\Q_s(d,\tau), \P_s(\tau))=d>0,
\]
and by $\Q_u(d,\tau)$ the point in $A^{\textrm{b}}(\tau)$ such that
\[
\dist(\Q_u(d,\tau), \P_u(\tau))=d>0.
\]

From Lemma \ref{L.loop}, we see that for any $\tau \in \R$ and any    $0<d \le \delta$,
we can define the maps
\begin{equation}\label{Q1T1}
\begin{gathered}
\mathscr{T}_1(d,\tau):=   \mathscr{T}^{\textrm{fwd}}(\Q_s(d,\tau),\tau) \, , \qquad \mathscr{P}_1(d,\tau):=
\mathscr{P}^{\textrm{fwd}}(\Q_s(d,\tau),\tau),\\
\mathscr{T}_{-1}(d,\tau):=   \mathscr{T}^{\textrm{bwd}}(\Q_u(d,\tau),\tau) \, , \qquad \mathscr{P}_{-1}(d,\tau):=
\mathscr{P}^{\textrm{bwd}}(\Q_u(d,\tau),\tau).
\end{gathered}
\end{equation}
Sometimes we will also make use of the maps
\begin{equation}\label{Q1T1half}
\begin{gathered}
\mathscr{T}_{\frac{1}{2}}(d,\tau):=   \tau_1(\Q_s(d,\tau),\tau) \, , \qquad \mathscr{P}_{\frac{1}{2}}(d,\tau):=
\mathscr{P}^{\textrm{fwd}}_+(\Q_s(d,\tau),\tau),\\
\mathscr{T}_{-\frac{1}{2}}(d,\tau):=   \tau_{-1}(\Q_u(d,\tau),\tau) \, , \qquad \mathscr{P}_{-\frac{1}{2}}(d,\tau):=
\mathscr{P}^{\textrm{bwd}}_-(\Q_u(d,\tau),\tau).
\end{gathered}
\end{equation}

Let us    define
\begin{equation}\label{mu0}
\mu_0= \frac{1}{4} \min \left\{ \sTfwd_+ , \sTbwd_- ,  \und{\sigma}^2    \right\}.
\end{equation}
We introduce a new parameter $\mu \in ]0, \mu_0]$, which gives an upper bound for the estimate of the errors
in the evaluations of the maps defined in \eqref{Q1T1} and \eqref{Q1T1half}.

 \begin{knowntheorem}\label{key}\cite[Theorem 4.2]{CFPRimut}
Assume \assump{F0}, \assump{F1}, \assump{F2}, \assump{K}, \assump{G}  and    let $\f^\pm$ and $\g$ be $C^r$ with $r>1$.
 There are $\ep_0>0$, $\delta>0$, such that for any $0<\ep \le \ep_0$, $0<d \le \delta$,  $\tau \in \R$,
	the functions $\TTT_{\pm 1}(d,\tau)$, $\PPP_{\pm 1}(d,\tau)$
	are $C^r$. Furthermore, for any $0<\mu<\mu_0$, we can find $\ep_0>0$, $\delta>0$, such that for any $0<\ep \le \ep_0$, $0<d \le \delta$,
$\tau \in \R$,
	\begin{equation}\label{D1T1}
	\begin{split}
	 	   d^{\sfwd +\mu} & \le \dist (\PPP_1(d,\tau),\P_u(\TTT_1(d,\tau)))
 \le d^{\sfwd-\mu}, \\
	 d^{\sbwd+\mu} & \le \dist (\PPP_{-1}(d,\tau), \P_s (\TTT_{-1}(d,\tau)))
 \le d^{\sbwd-\mu},\\
 \|\PPP_{\frac{1}{2}}(d,\tau)\| & \le d^{\sfwd_+ -\mu}     , \qquad    \|\PPP_{-\frac{1}{2}}(d,\tau)\| \le d^{\sbwd_- -\mu},
	\end{split}
	\end{equation}
	\begin{equation}\label{T1T-1}
	\begin{split}
	 	 \left[\sTfwd -\mu \right] |\ln(d)|& \le  (\TTT_1(d,\tau)-\tau)    \le   \left[\sTfwd+\mu \right] |\ln(d)|, \\
	 \left[\sTbwd -\mu \right] |\ln(d)|
& \le  \tau-\TTT_{-1}(d,\tau)    \le    \left[\sTbwd+\mu \right] |\ln(d)|, \\
 \left[\sTfwd_+ -\mu \right] |\ln(d)|& \le  (\TTT_\frac{1}{2}(d,\tau)-\tau)   \le   \left[\sTfwd_++\mu \right] |\ln(d)|, \\
	 \left[\sTbwd_- - \mu \right] |\ln(d)|
& \le  \tau -\TTT_{-\frac{1}{2}}(d,\tau)     \le    \left[\sTbwd_-+\mu \right] |\ln(d)|,
	\end{split}
	\end{equation}
	and all the expressions in \eqref{D1T1} are uniform with respect to any   $\tau \in \R$ and $0< \ep \le \ep_0$.
\end{knowntheorem}

\begin{knowntheorem}\label{keymissed}\cite[Theorem 4.3]{CFPRimut}
Assume \assump{F0}, \assump{F1}, \assump{F2}, \assump{K}, \assump{G}  and    let $\f^\pm$ and $\g$ be $C^r$ with $r>1$.
 For any $0< \mu <\mu_0$ we can find $\ep_0>0$, $\delta>0$,  such that for any $0<\ep \le \ep_0$,   $0<d \le \delta$,      and any $\tau
\in \R$
   we find
  \begin{equation}\label{keymissed.es-}
   \|\x(t,\tau; \Q_s(d,\tau))-\x(t,\tau; \P_s(\tau))\| \le  d^{\sfwd_+ - \mu }
  \end{equation}
  for any $\tau \le t \le \TTT_{\frac{1}{2}}(d,\tau)$,   and
  \begin{equation}\label{keymissed.es+}
  \|\x(t,\tau; \Q_s(d,\tau))-\x(t,\TTT_1(d,\tau); \P_u(\TTT_1(d,\tau)))\| \le  d^{\sfwd_+ - \mu }
  \end{equation}
    for any $\TTT_{\frac{1}{2}}(d,\tau) \le t \le \TTT_1(d,\tau)$.
    Further   $\x(t,\tau; \Q_s(d,\tau))$ is in  $\tilde{V}^+(t)$  for any $\tau<t<
\mathscr{T}_{\frac{1}{2}}(d,\tau)$ and it is
in
$\tilde{V}^-(t)$ for any $\mathscr{T}_{\frac{1}{2}}(d,\tau)<t< \mathscr{T}_{1}(d,\tau)$.

Similarly, for any $0<\ep \le \ep_0$,   $0<d \le \delta$,   $0< \mu <\mu_0$   and any $\tau \in \R$  we find
\begin{equation}\label{keymissed.eu+}
	\|\x(t,\tau; \Q_u(d,\tau))-\x(t,\tau; \P_u(\tau))\| \le  d^{\sbwd_- - \mu}
\end{equation}
  for any $   \TTT_{-\frac{1}{2}}(d,\tau)  \le t \le \tau$,    and
\begin{equation}\label{keymissed.eu-}
	\|\x(t,\tau; \Q_u(d,\tau))-\x(t,\TTT_{-1}(d,\tau); \P_s(\TTT_{-1}(d,\tau)))\| \le  d^{\sbwd_- - \mu}
\end{equation}
    for any $\TTT_{-1}(d,\tau) \le t \le \TTT_{-\frac{1}{2}}(d,\tau)$.
    Further $\x(t,\tau; \Q_u(d,\tau))$ is in  $\tilde{V}^-(t)$  for any $
\mathscr{T}_{-\frac{1}{2}}(d,\tau)<t<\tau$
and it is in  $\tilde{V}^+(t)$  for any $\mathscr{T}_{-1}(d,\tau)<t< \mathscr{T}_{-\frac{1}{2}}(d,\tau)$.
\end{knowntheorem}

 \begin{remark}\label{est.ga}
  Assume \assump{F0}, \assump{K}, then there is  a constant $c_0^*>0$ such that
  $\|\ga^-(t)\| \le \frac{c_0^*}{4} \eu^{\la_u^- t}$ for any $t \le 0$ and
  $\|\ga^+(t)\| \le \frac{c_0^* }{4}\eu^{\la_s^+ t}$ for any $t \ge 0$.
\end{remark}

We state now two classical results concerning the possibility to estimate the position of the trajectories of the unstable manifold
$\tilde{W}^u(\tau)$ and of the stable manifold $\tilde{W}^s(\tau)$ using   $\ga^-(t)$ and $\ga^+(t)$ respectively.
The proof is omitted, see, e.g., the nice introduction of \cite{JPY},  or \cite[\S 4.5]{GH}.
\begin{remark}\label{rMelni}
	Assume that $\f^{\pm}$ and $\g$ are $C^r$, $r>1$.
Observe that if $0<\ep \ll 1$ then,  for any fixed $\tau \in \R$,
	there are $c>0$ (independent of $\ep$ and $\tau$) and a monotone increasing and continuous function
$\bar{\omega}(\ep)$ such that $\bar{\omega}(0)=0$ and
	$$\dist(\P_s(\tau),\P_u(\tau))= c \ep \MM(\tau)+ \ep  \omega(\tau,\ep) \, , \qquad  |\omega(\tau,\ep)| \le
\bar{\omega}(\ep).$$
Further, if $r \ge 2$, we can find $C_M>0$ such that $\bar{\omega}(\ep) \le C_M \ep$.
\end{remark}

\begin{remark}\label{rMelnibis}
 Assume \assump{K} and \assump{F1}, then there is $\ep_0>0$ such that for any $0< \ep \le \ep_0$ we have the following. There is $\bar{c}^*>0$
 such that

 \begin{equation}\label{trajWuWs}
   \begin{split}
      \| \x(t,\tau ; \P_u(\tau)) -\ga^-(t-\tau) \| \le \bar{c}^* \ep & \qquad \qquad \textrm{for any $t \le \tau$}, \\
        \| \x(t,\tau ; \P_s(\tau)) -\ga^+(t-\tau) \| \le \bar{c}^* \ep & \qquad \qquad \textrm{for any $t \ge \tau$}.
   \end{split}
 \end{equation}
\end{remark}
Now we are ready to state the following result, which allows to locate the trajectories in an $\ep$-neighborhood of $\bs{\Gamma}$,
when $t$ is in an appropriate interval.

We introduce now two further time values, independent of $d$, $\tau$, $\ee$:
\begin{equation}\label{TaTb}
 T_a:=  \frac{1}{ \la_u^- }  | \ln(\ep)|  \, , \qquad
 T_b:=  \frac{1}{| \la_s^+ |}| \ln(\ep)|.
\end{equation}

  \begin{remark} \label{rem:TaTb}
    Using Remark \ref{est.ga} and \eqref{TaTb}
     we find
  $$\|\ga(t)\| \le\frac{c^*_0}{4} \ep  \quad  \textrm{for any $t \le -T_a$ and any $t\ge T_b$},$$
  Then, using the exponential behavior of trajectories converging to the origin we find $c^*_1,\kappa>0$ (cf. \cite[\S 4.1,
  \S 4.2]{CDFP}) such that
  $$\|\x(t,\tau; \P_u(\tau))\| \le \frac{c^*_1}{5} \eu^{( \la_u^- -\kappa\ep)(t-\tau)} \le \frac{c^*_1}{4}    \ep  \quad  \textrm{for
  any $t-\tau
  \le -T_a$},$$
  $$\|\x(t,\tau; \P_s(\tau))\| \le \frac{c^*_1}{5} \eu^{-(| \la_s^+ |-\kappa\ep)(t-\tau)} \le \frac{c^*_1}{4} \ep  \quad  \textrm{for
  any $t-\tau
  \ge T_b$}.$$
  Further from \eqref{defK0-new} we see that $\max \{T_a, T_b \} \le  2 K_{0}|\ln(\ep)| \le  K_{0}(1+\nu_0)|\ln(\ep)|$.
\end{remark}

Let us set
\begin{equation}
\begin{split}
     J_0  & :=  \left\{ d \in \R \mid 0< d \le \ep^{\frac{1+\nu}{\und{\si}}} \right\}.
  \label{J0}
\end{split}
\end{equation}

\begin{proposition}\label{forPropC}
  Assume \assump{F0}, \assump{F1}, \assump{F2}, \assump{K}, \assump{G}, then we can find $\ep_0$ such that for
  any $0< \ep \le \ep_0$ we have the following.

  Fix $\tau \in \R$, $d \in J_0$ and $\nu \ge \nu_0$, then there is $c^*>0$
  (independent of $d$, $\nu$ and $\ep$) such that
  $$
  \| \x(t,\tau; \Q_s(d,\tau))-\ga^+(t-\tau)\| \le \frac{c^*}{2} \ep $$
  for any $\tau \le t \le \TTT_{\frac{1}{2}}(d,\tau)$, and
   $$
  \| \x(t,\tau; \Q_s(d,\tau))-\ga^-(t-\TTT_1(d,\tau))\| \le \frac{c^*}{2} \ep $$
  for any $\TTT_{\frac{1}{2}}(d,\tau) \le t \le \TTT_1(d,\tau)$.

  Further,
  $$
  \| \x(t,\tau; \Q_u(d,\tau))-\ga^-(t-\tau)\| \le \frac{c^*}{2} \ep $$
  for any $  \TTT_{-\frac{1}{2}}(d,\tau)\le t \le \tau$, and
   $$
  \| \x(t,\tau; \Q_u(d,\tau))-\ga^+(t-\TTT_{-1}(d,\tau))\| \le \frac{c^*}{2} \ep $$
  for any $\TTT_{-1}(d,\tau) \le t \le \TTT_{-\frac{1}{2}}(d,\tau)$.
\end{proposition}
\begin{proof}
Let $\tau \in \R$, $d \in J_0$ and $\nu \ge \nu_0$ be fixed, and let
$\bar{c}^*>0$ be as in Remark~\ref{rMelnibis}. Then,
 from \eqref{keymissed.es-} and Remark \ref{rMelnibis} we find
\begin{equation*}
\begin{split}
  & \| \x(t,\tau; \Q_s(d,\tau))-\ga^+(t-\tau)\| \le \| \x(t,\tau; \Q_s(d,\tau)) - \x(t,\tau; \P_s(\tau))\|
  \\
  & + \| \x(t,\tau; \P_s(\tau))-\ga^+(t-\tau)\| \le  d^{\sfwd_+-\mu} +  \bar{c}^* \ep \le
   \ep + \bar{c}^* \ep \le  \frac{c^*}{2} \ep
\end{split}
\end{equation*}
for any  $\tau \le t \le \TTT_{\frac{1}{2}}(d,\tau)$, and
where
\begin{equation}\label{cstar}
  c^*= \max\{2(\bar{c}^*+1);   c_0^* ;  c^*_1\}>0
\end{equation}
is a constant independent of $\ep$, $d$, $\nu$, $\tau$, and
 $ c_0^*,c_1^* >0$ are as in Remarks \ref{est.ga} and \ref{rem:TaTb}.

The other inequalities can be proved in the same way using  Remark \ref{rMelnibis}
together with, respectively,
\eqref{keymissed.es+}, \eqref{keymissed.eu+}, \eqref{keymissed.eu-}.
\end{proof}

\begin{lemma}\label{forgottentimes}
  Let  $d \in J_0$    and $c^*>0$ be as in Proposition \ref{forPropC}. Then
  \begin{equation}\label{T.half}
  \begin{split}
       & \tau +T_b  \le \TTT_{\frac{1}{2}}(d,\tau) \le \TTT_1(d,\tau)-T_a  \qquad  \textrm{and}   \\
       &  \|x(t,\tau ; \Q_s(d,\tau))\| \le  \frac{3}{4}  c^* \ep  \qquad  \textrm{for any $t \in [\tau +T_b ,\TTT_1(d,\tau)-T_a]$.}
  \end{split}
  \end{equation}
  Further,
   \begin{equation}\label{T.halfbis}
  \begin{split}
       & \TTT_{-1}(d,\tau) +T_b  \le \TTT_{-\frac{1}{2}}(d,\tau) \le \tau-T_a  \qquad  \textrm{and}   \\
       &  \|x(t,\tau ; \Q_u(d,\tau))\| \le   \frac{3}{4}  c^* \ep  \qquad  \textrm{for any $t \in [\TTT_{-1}(d,\tau) +T_b ,\tau-T_a]$.}
  \end{split}
  \end{equation}
\end{lemma}
\begin{proof}
Using \eqref{J0}, \eqref{T1T-1} and \eqref{TaTb},   for any $d \in  J_0$ we get
\begin{equation*}
  \begin{split}
     \TTT_{\frac{1}{2}}& (d,\tau)-\tau \ge  (\sTfwd_+ -\mu) |\ln(d)| \ge \frac{\sTfwd_+ -\mu}{\und{\sigma} }(1+ \nu)|\ln(\ep)| \\
     &   \ge   \frac{\sTfwd_+}{\und{\sigma}} |\ln(\ep)|
     \ge T_b,
  \end{split}
\end{equation*}
 since $\mu \le \mu_0  \leq \frac{1}{2}\sTfwd_+$ by \eqref{mu0} and $\nu \ge \nu_0 \ge 1$.

 Let us set $D_1=\DDD (\PPP_1(d,\tau),\P_u(\TTT_1(d,\tau)))$ for short; since  $0<d \le \ep^{(1+\nu)/\und{\si}}$,
 from Theorem \ref{key} we see that
\begin{equation}\label{twice}
\begin{split}
    &   D_1 \le d^{ \sfwd-\mu} \le \ep^{\frac{1+\nu}{\und{\si}}  \cdot \frac{3\sfwd}{4}} \le
    \ep^{ \frac{3\sfwd_-}{4}(1+\nu)}.
\end{split}
\end{equation}
 Then,
  using the last line in \eqref{T1T-1}, we find
  \begin{equation*}
  \begin{split}
  & \TTT_1(d,\tau)-\TTT_{-\frac{1}{2}}(D_1, \TTT_1(d,\tau))
  \ge (\sTbwd_-- \mu) |\ln( D_1)|   \\
   & \ge   \sTbwd_- \sfwd_-\frac{9(1+ \nu)}{16}|\ln(\ep)| = \frac{9(1+ \nu)}{16  \la_u^- }|\ln(\ep)|  \ge   \frac{9}{8  \la_u^- }|\ln(\ep)|
   \ge
   T_a.
  \end{split}
\end{equation*}

Then, using  the fact  that $\TTT_{\frac{1}{2}}(d,\tau) = \TTT_{-\frac{1}{2}}(D_1, \TTT_1(d,\tau))$, we conclude the proof of the
first line in \eqref{T.half}.

Now, from Remark \ref{rem:TaTb} we see that
$\| \x(t,\tau; \P_s(\tau))   \| \le \frac{c^*}{4} \ep $ for any $t \ge T_b+\tau$ and
$\|\x(t, \TTT_1(d,\tau); \P_u(\TTT_1(d,\tau)))\| \le \frac{c^*}{4} \ep $ for any $t \le \TTT_1(d,\tau)-T_a$.
Hence, arguing as in the proof of Proposition~\ref{forPropC} for any $t \in [T_b+\tau , \TTT_{\frac{1}{2}}(d,\tau)]$ we find
\begin{equation*}
\begin{split}
    &
  \|\x(t,\tau; \Q_s(d,\tau)) \| \le  \|\x(t,\tau; \Q_s(d,\tau))- \x(t,\tau; \P_s(\tau))\| +\|\x(t,\tau; \P_s(\tau))\| \le  \frac{3}{4} c^*
  \ep,
\end{split}
\end{equation*}
and similarly for any  $t \in [\TTT_{\frac{1}{2}}(d,\tau),\TTT_1(d,\tau)-T_a]$
we find
\begin{equation*}
\begin{aligned}
  \|\x(t,\tau; \Q_s(d,\tau)) \| &\le \|\x(t,\tau; \Q_s(d,\tau))- \x(t,\TTT_1(d,\tau); \P_u(\TTT_1(d,\tau)))\|\\
  & \quad {}+\|\x(t,\TTT_1(d,\tau); \P_u(\TTT_1(d,\tau)))\| \le
  \frac{3}{4}   c^* \ep.
\end{aligned}
\end{equation*}
So, the second inequality in \eqref{T.half} is proved.

 The proof of \eqref{T.halfbis} is analogous and it is omitted.
\end{proof}

\begin{lemma}\label{forgotten.origin}
  Let  $d \in J_0$   and $c^*>0$ be as in Proposition \ref{forPropC}.
  If $t \in [\tau +T_b ,\TTT_1(d,\tau)-T_a]$ we get
  \begin{eqnarray}
   &  \|x(t,\tau ; \Q_s(d,\tau))\| \le  c^* \ep \label{est.origin},  \\
  &  \|x(t,\tau ; \Q_s(d,\tau))- \ga(t-\tau)\| \le  c^* \ep  \label{est.origin-}, \\
  &  \|x(t,\tau ; \Q_s(d,\tau))- \ga(t-\TTT_1(d,\tau))\| \le  c^* \ep \label{est.origin+}.
  \end{eqnarray}
\end{lemma}
\begin{proof}
Let $t \in [\tau +T_b ,\TTT_1(d,\tau)-T_a]$; then \eqref{est.origin} follows directly from  \eqref{T.half} in Lemma \ref{forgottentimes}.

Further, from the same result we see that $\TTT_{\frac{1}{2}}(d,\tau) \in [\tau +T_b ,\TTT_1(d,\tau)-T_a]$. Assume first
$t \in [\tau +T_b , \TTT_{\frac{1}{2}}(d,\tau)]$, then \eqref{est.origin-} follows from Proposition \ref{forPropC}.
Moreover
 $$\|\ga(t-\TTT_1(d,\tau))\| \le \frac{c^*}{4} \ep$$
by Remark \ref{rem:TaTb}.
Hence using  again Lemma \ref{forgottentimes} we find
\begin{equation*}
  \begin{split}
     \|x(t,\tau ; \Q_s(d,\tau))- \ga(t-\TTT_1(d,\tau))\| & \le  \|x(t,\tau ; \Q_s(d,\tau))\|+ \|\ga(t-\TTT_1(d,\tau))\|    \\
       & \le \frac{3c^*}{4} \ep + \frac{c^*}{4} \ep = c^* \ep.
  \end{split}
\end{equation*}
So \eqref{est.origin+} is proved.
Assume now
$t \in [\TTT_{\frac{1}{2}}(d,\tau), \TTT_1(d,\tau)-T_a]$, then \eqref{est.origin+} follows from Proposition \ref{forPropC}.
Moreover $t-\tau \ge \TTT_{\frac{1}{2}}(d,\tau)-\tau \ge T_b$ hence
 $$\|\ga(t-\tau)\| \le  \frac{c^*}{4} \ep$$ by Remark \ref{rem:TaTb}.
So again from  Lemma \ref{forgottentimes} we find
\begin{equation*}
  \begin{split}
     \|x(t,\tau ; \Q_s(d,\tau))- \ga(t-\tau)\| & \le  \|x(t,\tau ; \Q_s(d,\tau))\|+ \|\ga(t-\tau)\|    \\
       & \le \frac{3c^*}{4} \ep + \frac{c^*}{4} \ep = c^* \ep
  \end{split}
\end{equation*}
and the proof is concluded.
\end{proof}

The following is a consequence of Theorem \ref{key}.

\begin{lemma}\label{key1}
Assume \assump{F0}, \assump{F1}, \assump{F2}, \assump{K}, \assump{G}.
	Fix $\nu \ge \nu_0$ and let $d\in J_0$;
	fix $\tau \in \R$.
If there is $c_1>0$ such that
	$\MM(\TTT_{\pm 1}(d,\tau))\le - 3c_1$ then
 \begin{equation*}
  \begin{split}
     d_1 (d,\tau) & :=\dist (\PPP_1(d,\tau),\P_s(\TTT_1(d,\tau))) >  cc_1 \ep>0, \\
        d_{-1} (d,\tau) & :=\dist (\PPP_{-1}(d,\tau),\P_u(\TTT_{-1}(d,\tau))) <-cc_1 \ep <0 ,
  \end{split}
\end{equation*}
 while if  $\MM(\TTT_{\pm 1}(d,\tau))\ge
3c_1$ then
$\pm d_{\pm1}(d,\tau)<- cc_1\ep<0$.
\end{lemma}

\begin{proof}
We just prove the assertions concerning $\MM(\TTT_1(d,\tau))$,  the other case being analogous.
So we assume $\MM(\TTT_1(d,\tau))\le  - 3c_1$.
Then by Theorem \ref{key} and Remark \ref{rMelni},
\begin{equation*}
  \begin{split}
     d_1(d,\tau) & = \dist( \PPP_1(d,\tau), \P_s(\TTT_1(d,\tau))) \\
       & \ge \dist(\P_u(\TTT_1(d,\tau)),\P_s(\TTT_1(d,\tau))) - |\dist(\PPP_1(d,\tau), \P_u(\TTT_1(d,\tau)))| \\
       & \ge c\ep  3c_1 +o(\ep) -  d^{\sfwd-\mu}  \ge  2c c_1\ep-  \ep^{1+\mu_r}> c c_1 \ep
  \end{split}
\end{equation*}
 with  $1+\mu_r= \frac{1+\nu}{\und{\sigma}} (\sfwd - \mu)   \ge
\frac{3}{4}\frac{1+\nu_0}{\und{\sigma}}\sfwd  >1$, see \eqref{mu0}.

Analogously, assuming $\MM(\TTT_1(d,\tau))\ge  3c_1$ we get
\begin{equation*}
	\begin{split}
		d_1(d,\tau) & = \dist(\PPP_1(d,\tau), \P_s(\TTT_1(d,\tau))) \\
		& \le \dist(\P_u(\TTT_1(d,\tau)),\P_s(\TTT_1(d,\tau))) + |\dist( \PPP_1(d,\tau), \P_u(\TTT_1(d,\tau)))| \\
		& \le -c\ep  3c_1 +o(\ep) + d^{\sfwd-\mu} \le  - 2 c c_1 \ep+ \ep^{1+\mu_r}  <- c c_1 \ep .
	\end{split}
\end{equation*}
\end{proof}

\section{Proof of the main results} \label{S.pf.main}

Before giving the proofs in all the (lengthy) details, we sketch the argument and we sum up the main ideas.

 We assume that $0<\ep\leq \ep_0$ is sufficiently small, we fix
 $\tau \in [b_0, b_1]$; we develop the argument twice: in Section \ref{proof.future}
 we let $\TT^+= (T_j)$, $j \in \Z^+$
 be a sequence satisfying \eqref{TandKnunew} for $j \in \Z^+$  and \eqref{TandKnunew_tau+},
 and we develop in details an iterative argument to prove   Theorem \ref{restatebis} \textbf{a)}.
The idea is to find, for any fixed
 $\ee^+ \in \EE^+$, a compact interval $J_{+\infty}:= J_{+\infty}^{\ee^+}(\tau,\TT^+) \subset  J_0$
such that the trajectory $\x(t,\tau ; \Q_s(d,\tau))$ has property   $\bs{C_{\ee^+}^+}$
whenever $d \in  J_{+\infty}$
(we conjecture that if the zeros of $\MM$ are non-degenerate then $J_{+\infty}$ should reduce to a singleton).
Then for any $\ee^+ \in \EE^+$ and any $\TT^+$ as above we define the compact set of initial conditions
$$X^+(\ee^+,\tau,\TT^+)= \{ \Q_s(d,\tau) \mid d \in J^{\ee^+}_{+\infty}(\tau, \TT^+) \}.$$

 In Section \ref{proof.futurebis}  we adapt the argument to the case where
  $\TT^+$ satisfies \eqref{TandKnu} for $j \in \Z^+$  and \eqref{TandKnu_tau+},
  and we show that if \eqref{minimal} or a stronger non-degeneracy condition holds
  then
  we get better estimates on  ``$\alpha_j$", and we prove Theorem~\ref{restatefuture}.

In Section \ref{proof.past} we use an inversion of time argument and we prove   Theorem~\ref{restatebis}  \textbf{b)}
  and     Theorem
\ref{restatepast}. So for any $\ee^- \in \EE^-$
we construct a  compact interval  $J_{-\infty}:= J_{-\infty}^{\ee^-}(\tau,\TT^-)\subset  J_0$
such that the trajectory $\x(t,\tau ; \Q_u(d,\tau))$ has property  $\bs{C_{\ee^-}^-}$
 whenever $d \in  J_{-\infty}$.
Then for any $\ee^- \in \EE^-$ and any $\TT^-$  satisfying \eqref{TandKnu} for $j \le -2$ and \eqref{TandKnu_tau-}  we
define the compact sets of initial conditions
$$X^-(\ee^-,\tau,\TT^-)= \{ \Q_u(d,\tau) \mid d \in  J^{\ee^-}_{-\infty}(\tau, \TT^-) \}.$$

\subsection{Construction of the chaotic patterns in forward time: setting \eqref{TandKnunew}}\label{proof.future}
The aim of this section is to  prove  Theorem  \ref{restatebis} \textbf{a)}, so we aim to show the existence of a chaotic pattern in forward
time: for this
purpose we rely on a  constructive iterative
argument based on Theorem \ref{key}
and Lemma \ref{key1}.

We split $\EE^+= \{0,1\} ^{\Z^+}$ in two subsets $\hat{\EE}^+$ and $\EE^+_0$:
 \begin{equation}\label{efuture}
 \begin{array}{c}
   j^+_*(\ee^+)=\sup\{ j \mid \ee^+_j=1  \},\\
\hat{\EE}^+= \{ \ee^+=(\ee^+_j) \in \EE^+ \mid j^+_*(\ee^+)= +\infty \} ,\\
\EE^+_0= \{ \ee^+=(\ee^+_j) \in \EE^+ \mid  j^+_*(\ee^+)< +\infty \},\\
 j_\SS^+(\ee^+)=\begin{cases}\#\{j\mid 1\leq j\leq j^+_*(\ee^+),\,\ee_j^+=1\},&\text{if }j^+_*(\ee^+)<\infty,\\ \infty,& \text{if }
j^+_*(\ee^+)=\infty,\end{cases}
\end{array}
\end{equation}
 where $\#$ denotes the number of elements. So, $j^+_*(\ee^+)$ is ``the index of the last $1$ of $\ee^+$'' and $j_\SS^+(\ee^+)$ is
``the number of $1$s in $\ee^+$''.

Let us fix $\tau \in [b_0, b_1]$, $\TT^+=(T_j)$, $j \in \Z^+$ satisfying \eqref{TandKnunew},  \eqref{TandKnunew_tau+}
and $\ee^+ \in \EE^+$; adapting
\cite{BF11} we introduce a new sequence
$\SS=(\SS_j)$,  $j=0,1,\dots,j_\SS^+(\ee^+)$, which is a subsequence of $\TT^+$ depending also on $\ee^+$, and we define $(\Delta_j)$,
 $j=1,2,\dots,j_\SS^+(\ee^+)$ as
follows:
\begin{equation}\label{defSS}
\begin{split}
  \SS_0= & \tau \,, \qquad \qquad  \SS_j= \min \{T_{2k} > \SS_{j-1} \mid \ee_k=1 \}  , \qquad 1 \le j \le j^+_{\SS}(\ee^+),\\
  \Delta_j= & \SS_j-\SS_{j-1}  \, , \quad \textrm{ if  $1 \le j \le j^+_{\SS}(\ee^+)$} \,.
\end{split}
\end{equation}
Clearly,  $\SS_j$ has finitely many values if $\ee^+ \in \EE^+_0$, while if $\ee^+ \in \hat{\EE}^+$,   then $\SS_j$ and $\Delta_j$ are
defined for any $j \in \N$ and in $\Z^+$, respectively.\\
We denote by $k_j$ the subsequence such that $\SS_j=T_{2k_j}$, $1 \le j  \le j^+_{\SS}(\ee^+)$, and we set
$\beta_{2k_0}=b_0$, $ \beta'_{2k_0}=b_1$, so that
\begin{equation}\label{defik}
\begin{split}
 & \beta_{2k_j} < \SS_j < \beta'_{2k_j}, \qquad 0 \le j \le j^+_{\SS}(\ee^+)
\end{split}
\end{equation}
where  $\beta_j$  and $\beta'_j$ are the ones defined in \assump{P1} and \eqref{defkj}.
Notice that there is a subsequence $n_k$ such that $\beta_{k}=b_{n_k}$, so
 $\beta_{2k_j} = b_{n_{2k_j}}$ and  $\beta'_{2k_j} = b_{n_{2k_j}+1}$.
Further, setting $B_{2k_0}=B_0=b_1-b_0$, by \eqref{TandKnunew} for all $j \geq 1$  we find the estimate
\begin{equation}\label{defDej}
\Delta_j \ge 2 K_0(1+\nu) |\ln(\ep)| + B_{2k_j}+ B_{2k_{j-1}} .
\end{equation}

 From now until the end of the subsection, we consider $\tau$, $\TT^+$ and $\ee^+ \in \EE^+$ fixed, so we usually leave this dependence
 unsaid.
In fact, we will rely mainly on the sequences $(\SS_j)$ and $(\Delta_j)$.

We start our argument with the more involved case of $\ee^+ \in \hat{\EE}^+$,  then the case where $\ee^+ \in  \EE^+_0$ will follow more
easily.

Let us observe that from  \eqref{defDej},  \eqref{TandKnunew}  and \eqref{defK0-new} we find
\begin{equation}\label{Delta1}
\begin{split}
  & \textrm{exp}\left(-\frac{5(\Delta_1+ B_{2k_1})}{ 2\sTfwd} \right) < \textrm{exp}\left(-\frac{\Delta_1- B_{2k_1}}{ 3\sTfwd}
  \right)
      \\
     &
     <  \eu^{\frac{2(1+\nu)K_0}{3\sTfwd} \ln(\ep)}=
    \ep^{\frac{(1+\nu)\overline{\Sigma}}{\underline{\sigma}\sTfwd}}
    \leq \ep^{\frac{1+\nu}{\und{\si}}}.
\end{split}
\end{equation}
Let us define
\begin{equation}\label{I1}
 I_1=  I_1^{\ee^+}(\tau, \TT^+) := \left[ \textrm{exp}\left(-\frac{5(\Delta_1+ B_{2k_1})}{2 \sTfwd}     \right) ,
 \ep^{(1+\nu)/ \und{\si}}\right].
\end{equation}
Then, from Theorem \ref{key} we obtain the following.
\begin{lemma}\label{key2}
  Let $\tau \in [b_0, b_1]$. Then there are $D^1_a, D^1_b \in   I_1$ such that
 \begin{equation*}
   \begin{split}
    \TTT_1(D^1_a,\tau) =     T_{2k_1}+B_{2k_1} ,  \qquad  \TTT_1(D^1_b,\tau) =     T_{2k_1}-B_{2k_1}.
   \end{split}
 \end{equation*}
 Hence the image of the function $\TTT_1(\cdot, \tau) :  I_1  \to \R$
 contains the closed interval $[\SS_1-B_{2k_1} , \SS_1+B_{2k_1}]$.
\end{lemma}
\begin{proof}
Let us set $D^1_A=\textrm{exp}\left(-\frac{5(\Delta_1+ B_{2k_1})}{ 2\sTfwd} \right)$;
  from Theorem \ref{key} we find
  $$ \TTT_1(D^1_A,\tau)-\tau  \ge  |\ln(D^1_A)|(\sTfwd- \mu) \ge \frac{5}{4} (\Delta_1+ B_{2k_1}) >\Delta_1+B_{2k_1}$$
  if $\mu \le \mu_0 \le\sTfwd/2$,  i.e. $\TTT_{1}(D_A^1,\tau)\geq T_{2k_1}+B_{2k_1}$.
  Further, again from Theorem \ref{key},  \eqref{defDej} and \eqref{defK0-new}, we find that
   $D^1_B= \ep^{(1+\nu)/\und{\si}} $ is such that
  $$ \TTT_1(D^1_B,\tau)-\tau \le  |\ln(D^1_B)| (\sTfwd+ \mu) =  |\ln(\ep^{(1+\nu)/ \und{\si}} )| (\sTfwd+ \mu)\le
     \Delta_1- B_{2k_1} ,    $$
     if $\mu \le \mu_0 \le\sTfwd/2$,  i.e. $\TTT_{1}(D_B^1,\tau)\leq T_{2k_1}-B_{2k_1}$.
     The assertion follows using the continuity of $\TTT_1(\cdot, \tau)$.
\end{proof}

Now we set
\begin{equation}\label{I1hat}
\begin{gathered}
\hat{I}_1 =\hat{I}_1^{\ee^+}(\tau, \TT^+):=   \left\{ d \in I_1  \mid  \beta_{2k_1} \le \TTT_1(d,\tau) \le \beta'_{2k_1} \right\}.
\end{gathered}
\end{equation}
Notice that $[\beta_{2k_1} , \beta'_{2k_1}] \subset [\SS_1-B_{2k_1}, \SS_1+B_{2k_1}]$,
so from the continuity of  $\TTT_1(\cdot,\tau)$ and from Lemma \ref{key2} we deduce that    $\hat{I}_1$ is non-empty and closed.
\begin{remark}\label{closest}
  A priori $\hat{I}_1$ may be disconnected, however we can find a closed interval $\check{I}_1^{\ee^+}(\tau, \TT^+)=\check{I}_1\subset
  \hat{I}_1$
   with the following property
   \begin{equation}\label{prop1}
     \begin{split}
 &   \TTT_1(\cdot,\tau):  \check{I}_1 \to  [\beta_{2k_1} , \beta'_{2k_1}] \qquad \textrm{is surjective}.
                                 \end{split}
   \end{equation}
    Further, we can assume w.l.o.g. that $\check{I}_1$ is the ``interval closest to $0$ satisfying property
   \eqref{prop1}''. Namely notice that by construction there is $a<b$, $a,b \in \hat{I}_1$  such that $\TTT_1(a,\tau)=\beta'_{2k_1}$ and
   $\TTT_1(b,\tau)=\beta_{2k_1}$, then we set
$$b'_1= \min \{d \in  \hat{I}_1 \mid  \TTT_1(d,\tau)=\beta_{2k_1}   \}, \qquad   a'_1= \max \{ d  \in \hat{I}_1 \mid d< b'_1, \,
\TTT_1(d,\tau)=\beta'_{2k_1}  \}$$  and we define $\check{I}_1=[a_1', b_1']$
(so it is the   smallest   interval with this property).
\end{remark}
\begin{lemma}\label{ep1}
  There are $A^-, A^+ \in \check{I}_1 $ such that $d_1(A^-,\tau)=-\ep^{(1+\nu)/\und{\si}}$ and $d_1(A^+,\tau)=\ep^{(1+\nu)/\und{\si}}$.
\end{lemma}

\begin{proof}
The assertion follows from property \eqref{prop1},  \assump{P1} and   Lemma \ref{key1}.
\end{proof}

Let us set
\begin{equation}\label{hatA}
  \begin{array}{l}
    \check{A}^-:= \min \{ d \in \check{I}_1^{\ee^+} \mid  d_1(d,\tau)= -\ep^{(1+\nu)/\und{\si}}\}, \\
    \check{A}^+:= \min \{ d \in \check{I}_1^{\ee^+} \mid  d_1(d,\tau)= \ep^{(1+\nu)/\und{\si}}\} ,\\
    \underline{D}_1= \min \{ \check{A}^-; \check{A}^+ \}, \qquad \overline{D}_1= \max \{ \check{A}^-; \check{A}^+ \}
  \end{array}
\end{equation}
and
\begin{equation}\label{J1e}
  J_1=J_1^{\ee^+}(\tau,\TT^+):=[\underline{D}_1,\overline{D}_1] \subset I_1.
\end{equation}
Further, we set
 \begin{equation}\label{defalpha1}
   \al_{k_1}=\al_{k_1}^{\ee^+}(\ep,\tau, \TT^+)= \TTT_1(d,\tau)- \SS_1.
 \end{equation}
 Then we have the following.

 \begin{remark}\label{time1.missed}
   Let $d \in J_1$ and $\tau  \in [b_0,b_1]$, then
   $$ \{ T_1; T_{2k_1-1}  \} \subset [\tau+ T_b ,  \TTT_1(d,\tau)-T_a].  $$
 \end{remark}
\begin{proof}[Proof of Remark \ref{time1.missed}]
Notice that, by \eqref{TandKnunew},  \eqref{TandKnunew_tau+}, \eqref{defK0-new} and  \eqref{TaTb},
$$T_1 \ge  \tau+ B_0 + K_0|\ln(\ep)| \ge  b_1 + \frac{|\ln(\ep)|}{| \la_s^+ |} \ge \tau + T_b,$$
$$ T_1 \le T_{2} - B_{2} - K_0|\ln(\ep)| \le \beta_2 - \frac{|\ln(\ep)|}{ \la_u^- } \le \TTT_1(d,\tau)- T_a,$$
i.e.,  $T_1  \in[\tau+ T_b ,  \TTT_1(d,\tau)-T_a]$.

  Observe now that $B_{2k_1} = \beta'_{2k_1}-\beta_{2k_1}$, see
  \eqref{defik}; then (recalling that $\beta'_0=b_1$ and $B_0=b_1-b_0$)
   again by  \eqref{TandKnunew}, \eqref{defK0-new} and  \eqref{TaTb}
  we find
 \begin{equation*}
   \begin{split}
&T_{2k_1-1} \ge T_{2k_1-2} + B_{2k_1-2} +  K_0|\ln(\ep)| \ge  \beta'_{2k_1-2} + \frac{|\ln(\ep)|}{| \la_s^+ |} \ge \tau + T_b,
   \end{split}
 \end{equation*}
and
 \begin{equation*}
   \begin{split}
& T_{2k_1-1}
\le T_{2k_1} - B_{2k_1} -  K_0|\ln(\ep)| \le \beta_{2k_1} - \frac{|\ln(\ep)|}{| \la_u^- |} \le \TTT_1(d,\tau)- T_a,
   \end{split}
 \end{equation*}
i.e.,  $ T_{2k_1-1}   \in[\tau+ T_b ,  \TTT_1(d,\tau)-T_a]$.
\end{proof}

\begin{lemma}\label{l.step1}
  Whenever $d \in J_1$ the trajectory $\x(t,\tau; \Q_s(d,\tau))$ satisfies $\bs{C^+_{\ee^+}}$ for any $\tau \le t \le \TTT_1(d,\tau)$,
  i.e. for any $\tau \le t \le T_{2 k_1} +\al_{k_1}$.
\end{lemma}
\begin{proof}
 Let $d \in J_1$. Assume first $k_1=1$ so that $\ee^+_1=1$  and $\TTT_1(d,\tau)=T_2+\al_{1}$. Then applying twice Proposition
 \ref{forPropC}, we see that
  \begin{eqnarray}
    \| \x(t,\tau; \Q_s(d,\tau))- \ga(t-\tau) \|  \le \frac{c^*}{2} \ep, & \forall  t \in [\tau, \TTT_{\frac{1}{2}}(d,\tau)], \label{firsthalf}
    \\
    \| \x(t,\tau; \Q_s(d,\tau))- \ga(t-T_2-\al_{1} )\|  \le \frac{c^*}{2} \ep, & \forall  t \in [\TTT_{\frac{1}{2}}(d,\tau), T_2+\al_{1}].
    \label{secondhalf}
  \end{eqnarray}
  From Lemma  \ref{forgottentimes} and Remark \ref{time1.missed} we see that $\{\TTT_{\frac{1}{2}}(d,\tau), T_1 \} \subset [\tau+ T_b ,
  \TTT_1(d,\tau)-T_a]$, hence
  from \eqref{firsthalf} we find that
  \begin{equation}\label{eq.goal}
    \| \x(t,\tau; \Q_s(d,\tau))- \ga(t-\tau) \|  \le c^* \ep
  \end{equation}
  holds for any $\tau \le t \le \tau+ T_b$, and from
  Lemma \ref{forgotten.origin} we see that \eqref{eq.goal} holds for any  $t \in [\tau+T_b, T_1]$, too.
   So \eqref{eq.goal} holds for any $t \in [\tau, T_1]$.

  Now, from
  Lemma \ref{forgotten.origin}, for any   $t \in  [T_1, \TTT_1(d,\tau)-T_a] \subset [\tau+T_b, \TTT_1(d,\tau)-T_a]$
  we find
    \begin{equation}\label{eq.goal2}
     \| \x(t,\tau; \Q_s(d,\tau))- \ga(t-T_2-\al_{1} )\|  \le c^* \ep,
  \end{equation}
  and from \eqref{secondhalf} we see that \eqref{eq.goal2} holds for any $t \in [\TTT_1(d,\tau)-T_a,  \TTT_1(d,\tau) ]$.
  So  \eqref{eq.goal2}   holds for any $t \in [T_1,  \TTT_1(d,\tau)]$ and
  when $k_1=1$ the lemma is proved.

   Now assume $k_1 \ge 2$ so that $\ee^+_j=0$ for any $j \in \{1, \ldots, k_1-1\}$.
  Hence
  $$\tau+T_b<T_1<T_2 < \ldots < T_{2k_1-1}< \TTT_1(d,\tau)-T_a,$$
  $$\{\TTT_{\frac{1}{2}}(d,\tau), T_1, T_{2k_1-1} \} \subset [\tau+ T_b ,  \TTT_1(d,\tau)-T_a]$$
  by Remark \ref{time1.missed}.
   So, again from Proposition \ref{forPropC} it follows that \eqref{firsthalf} holds, thus  \eqref{eq.goal} holds
  when $t \in  [\tau, \tau+T_b]\subset  [\tau, \TTT_{\frac{1}{2}}(d,\tau)]$, and from   Lemma \ref{forgotten.origin}
  we see that \eqref{eq.goal} holds
  when $t \in  [ \tau+T_b, T_1]\subset  [\tau+T_b, \TTT_1(d,\tau)-T_a]$.
   Hence \eqref{eq.goal} holds for any $t \in [\tau, T_1]$. Analogously,
   Proposition~\ref{forPropC} implies that
    \begin{equation}
    \| \x(t,\tau; \Q_s(d,\tau))- \ga(t-T_{2k_1}-\al_{k_1} )\|  \le \frac{c^*}{2} \ep, \forall  t \in [\TTT_{\frac{1}{2}}(d,\tau),
    T_{2k_1}+\al_{k_1}], \label{secondhalf.secondcase}
  \end{equation}
so  we see that
    \begin{equation}\label{eq.goal2.2}
     \| \x(t,\tau; \Q_s(d,\tau))- \ga(t-T_{2k_1}-\al_{k_1} )\|  \le c^* \ep
  \end{equation}
 holds
  when $t \in  [\TTT_1(d,\tau)-T_a, \TTT_1(d,\tau)]\subset  [  \TTT_{\frac{1}{2}}(d,\tau), \TTT_1(d,\tau)]$, and from
Lemma~\ref{forgotten.origin}
  we see that  \eqref{eq.goal2.2} holds
  when $t \in  [  T_{2k_1-1}, \TTT_1(d,\tau)-T_a]\subset     [\tau+T_b, \TTT_1(d,\tau)-T_a]$.
    Thus, \eqref{eq.goal2.2} holds for any $t \in [T_{2k_1-1},  \TTT_1(d,\tau)]$.

Finally  from Lemma \ref{forgotten.origin} we see that
  $$\| \x(t,\tau; \Q_s(d,\tau))\| \le c^* \ep  \qquad  \textrm{for any $t \in [T_1, T_{2k_1-1}] \subset  [\tau+T_b,\TTT_1(d,\tau)-T_a]$}.  $$
This shows that $\x(t,\tau; \Q_s(d,\tau))$ has property $\bs{C^+_{\ee^+}}$
 for any $t \in [\tau, \TTT_1(d,\tau)]$ and the lemma is proved.
\end{proof}

Summing up we have shown the following.
\begin{proposition}\label{p.step1}
  Let the assumptions of Theorem  \ref{restatebis} \textbf{a)}  be satisfied, and let
  $J_1$ be as in \eqref{J1e}.
  Fix $\tau\in[b_0,b_1]$. Then, the function
  $$ \TTT_1(\cdot ,\tau): J_1 \to  [\beta_{2k_1}  , \beta'_{2k_1}]$$
is well defined and $C^r$, while the function
$$ d_1(\cdot ,\tau): J_1 \to  \R$$
is  $C^r$ and its image contains $[-\ep^{(1+\nu)/ \und{\si}} , \ep^{(1+\nu)/ \und{\si}}]$.

Further,   the trajectory $\x(t,\tau; \Q_s(d,\tau))$ satisfies $\bs{C^+_{\ee^+}}$
for $t \in [\tau,  \TTT_1(d,\tau)]\supset [\tau, \beta_{2k_1}]$.  Finally, by construction,
$|\al_{k_1}| \le B_{2k_1}$.
\end{proposition}

\subsubsection{Iteration of the scheme: trajectories performing a prescribed number of loops}\label{s.iteration}
Our goal now is  to perform an iterative scheme in order to define  a family of
nested compact intervals  $J_n= J^{\ee^+}_n(\tau,\TT^+)$
such that  $J_n \supset J_{n+1} \ne  \emptyset$ for any $n>0$ and having the following properties.

For any  $i=1,\dots,n$ we can define $C^r$ functions
$$
  \TTT_i(\cdot ,\tau): J_i \to  [\beta_{2k_i}, \beta_{2k_{i}}'] \, , \qquad   d_i(\cdot ,\tau): J_i \to  \R,
$$
\begin{equation}\label{Tn}
\begin{split}
d_{i}(d,\tau):= & d_{1}(D,A), \quad \textrm{where $D=d_{i-1}(d,\tau)$, and $A=\TTT_{i-1}(d,\tau)$},\\
\TTT_{i}(d,\tau):= & \TTT_{1}(D,A), \quad \textrm{where $D=d_{i-1}(d,\tau)$, and $A=\TTT_{i-1}(d,\tau)$},\\
\TTT_{i-\frac{1}{2}}(d,\tau):= & \TTT_{\frac{1}{2}}(D,A), \quad \textrm{where $D=d_{i-1}(d,\tau)$, and $A=\TTT_{i-1}(d,\tau)$}
\end{split}
\end{equation}
 with $d_0(d,\tau):=d$, $\TTT_0(d,\tau):=\tau$,
such that, if
$d \in J_n$, the trajectory $\x (t,\tau; \vec{Q}_s(d,\tau))$ performs $n$ loops close to $\bs{\Gamma}$ when $t \in [\tau,  \TTT_n(d,\tau)]$,
and intersects (transversely) $L^0$ exactly at $t=\tau$ and at $t= \TTT_i(d ,\tau)$ and $L^{\inn}$ at $\TTT_{i-\frac{1}{2}}(d,\tau)$ for $i=1,
\ldots, n$.

\begin{proposition}\label{p.fwdn}
Let $\tau \in [b_0,b_1]$,  $\TT^+=(T_j^+) $ and $\ee^+ \in \EE^+$ be fixed as in  Theorem \ref{restatebis} \textbf{a)};
  then for any $n \le j^+_\SS(\ee^+)$ (resp.\ for any $n \in \N$ if $\ee^+ \in \hat{\EE}^+$, see \eqref{efuture}),
there is a compact interval $J_n= J_n^{\ee^+}(\tau,\TT^+)$ such that the function
$$ \TTT_n(\cdot ,\tau): J_n \to  [\beta_{2k_n}, \beta'_{2k_n}]$$
is well defined and $C^r$, while the function
$$ d_n(\cdot ,\tau): J_n \to  \R$$
is $C^r$ and its image contains $[-\ep^{(1+\nu)/ \und{\si}} , \ep^{(1+\nu)/ \und{\si}}]$.

Further, if
  \begin{equation}\label{defaln}
    \al_{k_n}=\al_{k_n}^{\ee^+}(\ep,d,\tau, \TT^+)= \TTT_n(d,\tau)- \SS_n,
  \end{equation}
the trajectory $\x(t,\tau; \Q_s(d,\tau))$ satisfies $\bs{C^+_{\ee^+}}$
for any $t \in [\tau,  \TTT_n(d,\tau)]\supset [\tau, \beta_{2k_n}]$,  and it crosses
$L^{\inn}$ at $\TTT_{i-\frac{1}{2}}(d,\tau)$ for $i=1, \ldots, n$.   Finally, by construction,
$|\al_{k_n}| \le B_{2k_n}$.
\end{proposition}

The proof of the proposition will follow from several lemmas.

We proceed by induction: the $n=1$ case follows from Proposition \ref{p.step1}, so now we show that the step $n$  follows
from the  step $n-1$.
Assume that $J_{n-1}$ is well defined.
Then we set  $D^n_\# := \textrm{exp}\left(\textstyle\frac{-5 (\Delta_n+B_{2k_n}+B_{2k_{(n-1)}}) }{2\sTfwd}\right)$  and
\begin{equation}\label{In}
    I_n= I^{\ee^+}_n(\tau,\TT^+):=
  \left\{ d \in J_{n-1} \mid    D^n_\#   \le d_{n-1}(d,\tau)    \le \ep^{(1+\nu)/ \und{\si}} \right\}.
\end{equation}
 For any $d \in I^{\ee^+}_n$ we define $d_{n}(d,\tau)$, $\TTT_{n}(d,\tau)$  and $\TTT_{n-\frac{1}{2}}(d,\tau)$ according to \eqref{Tn} with
 $i=n$,
so that $\TTT_n(d,\tau)>\TTT_{n-\frac{1}{2}}(d,\tau)>\TTT_{n-1}(d,\tau)$.

 \begin{lemma}\label{key2-stepn}
  Let $\tau \in\R$. Then there are $D_a^n, D_b^n \in   I_n$ such that
 \begin{equation*}
   \begin{split}
    \TTT_n(D_a^n,\tau) =     \SS_n+B_{2k_n} \,,\qquad \qquad  \TTT_n(D_b^n,\tau) =    \SS_n-B_{2k_n}.
   \end{split}
 \end{equation*}
 Hence the image of the function $\TTT_n(\cdot, \tau) :  I_n \to \R$
 contains the closed interval $[\beta_{2k_n} , \beta'_{2k_n}]$.
\end{lemma}
\begin{proof}
  Let $D^n_A\in  J_{n-1} $ be such that $d_{n-1}(D^n_A,\tau)=D^n_\#$.
  By construction   $D^n_A\in I_{n} \subset J_{n-1}$, hence $\TTT_{n-1}(D^n_A,\tau) \in [\beta_{2k_{(n-1)}}, \beta'_{2k_{(n-1)}}]$.
Using this fact and Theorem \ref{key},  we find
  \begin{gather*}
  	\TTT_n(D^n_A,\tau)= \TTT_1(d_{n-1}(D^n_A,\tau),\TTT_{n-1}(D^n_A,\tau))-\TTT_{n-1}(D^n_A,\tau)+\TTT_{n-1}(D^n_A,\tau) \\
  \ge \sTfwd\frac{|\ln(d_{n-1}(D^n_A,\tau))|}{2}+ \beta_{2k_{(n-1)}} = \frac{5}{4}  (\Delta_n+B_{2k_n}+B_{2k_{(n-1)}}) +\beta_{2k_{(n-1)}}
  \\ >
  \SS_{n}+B_{2k_n} .
  \end{gather*}

  Now, let $D^n_B \in  J_{n-1} $ be such that $d_{n-1}(D^n_B,\tau)= \ep^{(1+\nu)/\und{\si}}$.
  Since $D^n_B \in J_{n-1}$, we get  $\TTT_{n-1}(D^n_B,\tau) \in [\beta_{2k_{(n-1)}} , \beta'_{2k_{(n-1)}}]$;
     using this fact and   Theorem \ref{key},  with \eqref{defDej} and \eqref{defK0-new}, since  $ \mu \le \mu_0 \le \sTfwd/2$
  we find
  \begin{gather*}
  	\TTT_n(D^n_B,\tau) = \TTT_1(d_{n-1}(D^n_B,\tau),\TTT_{n-1}(D^n_B,\tau))-\TTT_{n-1}(D^n_B,\tau)+\TTT_{n-1}(D^n_B,\tau)\\
    \le \frac{3}{2} \sTfwd |\ln(d_{n-1}(D^n_B,\tau))| +\beta'_{2k_{(n-1)}} < K_0(1+\nu) |\ln(\ep)|+ \beta'_{2k_{(n-1)}}\\ < \Delta_n
    -B_{2k_n}-B_{2k_{(n-1)}}+\beta'_{2k_{(n-1)}}<
    \SS_{n}-B_{2k_n}.
    \end{gather*}
  So, the lemma follows from the  continuity of $\TTT_n(\cdot,\tau)$.
\end{proof}

Then we define
\begin{equation}\label{Inhat}
\begin{split}
\hat{I}_n= \hat{I}_n^{\ee^+}(\tau,\TT^+):= &   \left\{ d \in I_{n}  \mid
  \beta_{2k_n} \le \TTT_{n}(d,\tau) \le \beta'_{2k_n}   \right\},
\end{split}
\end{equation}
which is closed, non-empty and
 \begin{equation}\label{propn}
 \begin{split}
 &  \TTT_n(\cdot,\tau): \hat{I}_n \to [\beta_{2k_n} , \beta'_{2k_n}]   \qquad \text{is surjective.}
                                 \end{split}
 \end{equation}
Reasoning as in Remark \ref{closest} we denote by $\check{I}_n=\check{I}_n^{\ee^+}(\tau,\TT^+)$ the
 ``closed interval closest to $0$'' having  property \eqref{propn}.

 Arguing as in Lemma \ref{ep1}, we obtain the following.
\begin{remark}\label{epn}
  There are $A^-_n, A^+_n \in \check{I}_n $ such that $d_n(A^-_n,\tau)=-\ep^{(1+\nu)/\und{\si}}$ and
  $d_n(A^+_n,\tau)=\ep^{(1+\nu)/\und{\si}}$.
\end{remark}
Then we set
\begin{equation}\label{hatAn}
  \begin{array}{l}
    \check{A}^-_n:= \min \{ d \in \check{I}_n^{\ee^+} \mid  d_n(d,\tau)= -\ep^{(1+\nu)/\und{\si}}\}, \\
    \check{A}^+_n:= \min \{ d \in \check{I}_n^{\ee^+} \mid  d_n(d,\tau)= \ep^{(1+\nu)/\und{\si}}\} ,\\
    \underline{D}_n= \min \{ \check{A}^-_n; \check{A}^+_n \}, \qquad \overline{D}_n= \max \{ \check{A}^-_n; \check{A}^+_n \},
  \end{array}
\end{equation}
and   we denote by
 \begin{equation}\label{Jne}
  J_n=J_n^{\ee^+}(\tau,\TT^+):=[\underline{D}_n;\overline{D}_n]  \subset \check{I}_n   \subset J_{n-1}.
\end{equation}

Further,  applying  Lemma \ref{forgottentimes}, we get the following.
\begin{lemma}\label{forgottentimesn}
  Let $d \in J_n$  and $c^*>0$ be as in Proposition \ref{forPropC}. Then
  \begin{eqnarray}
    &    \TTT_{n-1}(d,\tau) +T_b  \le  \TTT_{\frac{1}{2}}(d_{n-1}(d,\tau),\TTT_{n-1}(d,\tau)) \le \TTT_n(d,\tau)-T_a,  \label{Thalfna}   \\
    &     \|x(t,\tau ; \Q_s(d,\tau))\| \le  \frac{3}{4}  c^* \ep, \qquad  \forall t \in [\TTT_{n-1}(d,\tau) +T_b
    ,\TTT_n(d,\tau)-T_a].\label{Thalfnb}
  \end{eqnarray}
\end{lemma}
\begin{proof}
  The inequality \eqref{Thalfna} follows directly from Lemma \ref{forgottentimes}, while \eqref{Thalfnb} can be obtained by
  repeating the argument in the second part of the proof of Lemma~\ref{forgottentimes}.
\end{proof}

Now, for any $d \in J_n$ and $j=1, \ldots, n$, let  $\al_{k_j}$ be as in \eqref{defaln}.
Notice that $|\al_{k_j}|\le \beta'_{2k_j}-\beta_{2k_j}=B_{2k_j}$.

  Repeating the argument of Lemma \ref{time1.missed} we prove the following.
   \begin{remark}\label{timen.missed}
   Observe that $T_{2k_{(n-1)}+1}  \le  T_{2k_n-1}$ and they are equal
  when $k_n=k_{(n-1)}+1$, i.e.~when we have two consecutive 1s in $\ee^+$.   Let $d \in J_n$ and $\tau \in [b_0, b_1]$, then
   $$  [ T_{2k_{(n-1)}+1} , T_{2k_n-1}  ] \subset [\TTT_{n-1}(d,\tau)+ T_b ,  \TTT_n(d,\tau)-T_a].$$
 \end{remark}
\begin{proof}[Proof of Remark \ref{timen.missed}]
 The first part is obvious. Now
since $|\alpha_{k_{(n-1)}}| \le B_{2k_{(n-1)}}$, then by \eqref{TandKnunew},
$K_0>\frac{1}{|\la_s^+|}$, and $K_0>\frac{1}{\la_u^-}$
we find
$$  T_{2k_{(n-1)}+1}  \ge  T_{2k_{(n-1)}} + B_{2k_{(n-1)}} + K_0(1+\nu)|\ln(\ep)| $$
$$\ge  \beta'_{2k_{(n-1)}} + \frac{|\ln(\ep)|}{| \la_s^+|} \ge \TTT_{n-1}(d,\tau) + T_b.$$

  Analogously   we find
$$   T_{2k_n-1}  \le T_{2k_n} - B_{2k_n} - K_0(1+\nu)|\ln(\ep)|$$
$$ \le \beta_{2k_n} - \frac{|\ln(\ep)|}{ \la_u^-} \le \TTT_n(d,\tau)- T_a,$$
i.e.,  $ [T_{2k_{(n-1)}+1} , T_{2k_n-1}] \subset [\TTT_{n-1}(d,\tau)+ T_b ,  \TTT_n(d,\tau)-T_a]$, where the first interval may reduce to a
singleton.
\end{proof}

  \begin{lemma}\label{l.stepn}
  Whenever $d \in J_n$ the trajectory $\x(t,\tau; \Q_s(d,\tau))$ satisfies $\bs{C^+_{\ee^+}}$ for any $\tau \le t \le \TTT_n(d,\tau)$,
  i.e. for any $\tau \le t \le \SS_n +\al_{k_n}$.
\end{lemma}
\begin{proof}
 We prove the lemma by induction in $n$. The $n=1$ case follows from Lemma \ref{l.step1}.
 Let $n>1$ and $d \in J_n \subset J_{n-1}$; since $d \in J_{n-1}$ we know from the inductive assumption that
 $\x(t,\tau; \Q_s(d,\tau))$ satisfies $\bs{C^+_{\ee^+}}$ for any $\tau \le t \le \TTT_{n-1}(d,\tau)$.
 From Lemma \ref{forgottentimesn} and Remark \ref{timen.missed}   we see that
\begin{equation}\label{timesinclusion}
    \{ \TTT_{n-\frac{1}{2}}(d,\tau); T_{2k_{(n-1)}+1}; T_{2k_{n}-1}  \} \subset [\TTT_{n-1}(d,\tau)+ T_b  ,
  \TTT_{n}(d,\tau)-T_a].
\end{equation}

  Assume first $k_n=k_{(n-1)}+1$ so that we have two consecutive $1$s in
 $\EE^+$, i.e. $\ee^+_{k_{(n-1)}}= \ee^+_{k_{(n-1)}+1}=1$.
   Note that
  \begin{gather*}
  	x(t,\tau;\Q_s(d,\tau)) \equiv x(t,\TTT_1(d,\tau);\Q_s(d_1(d,\tau),\TTT_1(d,\tau)))
  	\equiv \cdots\\
  	\equiv x(t,\TTT_{n-1}(d,\tau);\Q_s(d_{n-1}(d,\tau),\TTT_{n-1}(d,\tau)))
  \end{gather*}
  for any $t \in \R$.

  Then applying twice  Proposition \ref{forPropC}  we find
  \begin{equation}\label{firsthalfn}
    \| \x(t,\tau; \Q_s(d,\tau))- \ga(t-\SS_{n-1}-\al_{k_{(n-1)}}) \|  \le c^*  \ep
  \end{equation}
for any $t \in [\SS_{n-1}+\al_{k_{(n-1)}},\TTT_{n-\frac{1}{2}}(d,\tau)]$,
  \begin{equation}    \label{secondhalfn}
    \| \x(t,\tau; \Q_s(d,\tau))- \ga(t-\SS_{n}-\al_{k_n} )\|  \le c^*  \ep
  \end{equation}
for any $t \in [\TTT_{n-\frac{1}{2}}(d,\tau), \SS_{n}+\al_{k_n}]$.
From \eqref{timesinclusion} it follows that
 $\x(t,\tau; \Q_s(d,\tau))$ has property
  $\bs{C_{\ee^+}^+}$  whenever $t \in [\TTT_{n-1}(d,\tau),  T_{2k_{(n-1)}+1} ]$.

   Now assume $k_n>k_{(n-1)}+1$ so that $\ee^+_j=0$ for any $j \in \{k_{(n-1)}+1, \ldots, k_{n}-1\}$;
 from \eqref{timesinclusion}   we find
  \begin{gather*}
  	 \TTT_{n-1}(d,\tau)+T_b<T_{2k_{(n-1)}+1}<T_{2k_{(n-1)}+2} < \ldots < T_{2k_n-1}\\
  	< \TTT_1(d_{n-1}(d,\tau), \TTT_{n-1}(d,\tau))-T_a.
  \end{gather*}
  Now, from Proposition \ref{forPropC}
  and Lemma \ref{forgotten.origin} we see   that  \eqref{firsthalfn} holds
  respectively  for any  $t \in[\TTT_{ n-1}(d,\tau) , \TTT_{ n-1}(d,\tau)+ T_b]$ and for any $t \in
  [\TTT_{n-1}(d,\tau)+ T_b, T_{2k_{(n-1)}+1} ]\subset [\TTT_{n-1}(d,\tau)+ T_b, \TTT_{n}(d,\tau)-T_a] $.
  So property $\bs{C^+_{\ee^+}}$ holds for any $t \in  [\TTT_{n-1}(d,\tau), T_{2k_{(n-1)}+1}  ]$.

  Then
  from Lemma \ref{forgotten.origin} we see that
  $$\| \x(t,\tau; \Q_s(d,\tau))\| \le c^* \ep  \qquad  \textrm{for any
  $t \in [\TTT_{n-1}(d,\tau)+T_b, \TTT_{n}(d,\tau)-T_a]$},  $$
  so in particular also when  $t \in [T_{2k_{(n-1)}+1} , T_{2k_{n}-1} ]$.
   So property $\bs{C^+_{\ee^+}}$ holds for any $t \in  [T_{2k_{(n-1)}+1} , T_{2k_{n}-1} ]$.

 Now again from Proposition \ref{forPropC}  and  Lemma \ref{forgotten.origin}
 we see   that  \eqref{secondhalfn} holds
  respectively  for any  $t \in
  	[\TTT_{ n}(d,\tau)-T_a, \TTT_{ n}(d,\tau)] $       and for any  $t \in[ T_{2k_{n}-1} ,\TTT_{ n}(d,\tau)-T_a]
  \subset [ \TTT_{ n-1}(d,\tau)+ T_b, \TTT_{ n}(d,\tau)-T_a ]$.
  So we have shown property $\bs{C^+_{\ee^+}}$  for any $t \in  [T_{2k_{n}-1},\TTT_n(d,\tau) ]$
  and the proof is concluded.
\end{proof}

\begin{proof}[\textbf{Proof of Proposition \ref{p.fwdn}}]
 Notice that by construction $|\alpha_{k_n}| \le B_{2k_n}$, then the proof follows from  Lemma  \ref{l.stepn}.
\end{proof}

Let $\ee^+ \in \hat{\EE}^+$, we denote by
\begin{equation}\label{Jinfty}
\begin{split}
  J_{+\infty}= J_{+\infty}^{\ee^+}(\tau, \TT^+):= \cap_{n=1}^{+\infty} J_{n} , \qquad d_*=d_*^{\ee^+}(\tau, \TT^+)= \min \{d \in
  J_{+\infty}\}.
  \end{split}
\end{equation}
 Notice that   $J_{+\infty} $
is a non-empty compact connected interval since it is the countable intersection of non-empty compact intervals, one contained in
the other. So the minimum $d_*$ exists.

 Now we are in the position to prove the following.
 \begin{proposition}\label{newhalffuture}
  Theorem \ref{restatebis} part \textbf{a)} holds true.
\end{proposition}
\begin{proof}
Let $\ee^+ \in \hat{\EE}^+$, then the result is proved by choosing
$\xxi =\Q_s(d,\tau)$ and $d \in J_{+\infty}$.

Assume
 now that $\ee^+ \in \EE^+_0$, then we can apply  Proposition \ref{p.fwdn} for any $1 \le n \le j^+_{\SS}= j^+_{\SS}(\ee^+)$
 see \eqref{efuture}.
In particular we find a compact interval  $J_{j^+_{\SS}}:=J_{j^+_{\SS}}^{\ee^+}(\tau, \TT^+)$ such that
the function
$$ d_{j^+_{\SS}}(\cdot,\tau): J_{j^+_{\SS}} \to \R$$
is $C^r$ and its image contains the interval $[-\ep^{(1+\nu)/\und{\si}} , \ep^{(1+\nu)/\und{\si}}]$;
further the function
$$ \TTT_{j^+_{\SS}} (\cdot,\tau): J_{j^+_{\SS}}  \to [\beta_{2k_{j^+_{\SS}}}, \beta'_{2k_{j^+_{\SS}}}] $$
is well defined and $C^r$.
Hence
$$I_{+\infty}= I_{+\infty}^{\ee^+}(\tau, \TT^+):= \{d \in J_{j^+_{\SS}}    \mid d_{j^+_{\SS}} (d,\tau)=0 \}$$
is a closed non-empty set. Let us denote by $d_*= \min I_{+\infty}$ and by $J_{+\infty} =J_{+\infty}^{\ee^+}(\tau, \TT^+)$ the largest
connected component
of $I_{+\infty}$ containing $d_*$.
Then $\x(t,\tau; \Q_s(d,\tau))$ has property $\bs{C^+_{\ee^+}}$    for any $t \in [\tau,
\TTT_{j^+_{\SS}}(d,\tau)]$, whenever  $d \in
J_{+\infty}$. Further  $\x(\TTT_{j^+_{\SS}}(d,\tau),\tau; \Q_s(d,\tau))= \P_s(\TTT_{j^+_{\SS}}(d,\tau))$, hence
$\x(t,\tau; \Q_s(d,\tau)) \in \tilde{W}^s(t) \subset \Om^+$ for any $t \ge \TTT_{j^+_{\SS}}(d,\tau)$.
So from Proposition~\ref{forPropC}, recalling that
$\alpha_j= \TTT_{j}(d,\tau)-T_{2j}^{\ee^+}$ for any $1 \le j \le j^+_{\SS}$,
 we get
$$ \|\x(t,\tau; \Q_s(d,\tau))-      \ga^+(t-T_{2j^+_{\SS}}(d,\tau)-\alpha_{j^+_{\SS}}) \| \le \frac{c^*}{2} \ep$$
for any $t \ge T_{2j^+_{\SS}}(d,\tau)+\alpha_{j^+_{\SS}}$.
Moreover from Lemma \ref{forgottentimes} we have
$$\|\x(t,\tau; \Q_s(d,\tau))\| \le c^* \ep$$
for any
 $t \ge T_{2j^+_{\SS}+1}(d,\tau)  \ge  \TTT_{j^+_{\SS}}(d,\tau) + T_b$.
Hence we see that $ \x(t,\tau; \Q_s(d,\tau))$ has property $\bs{C^+_{\ee^+}}$ for any $t \ge \tau$, whenever
$d \in J_{+\infty}$ and we conclude the proof.
\end{proof}

  From Remark \ref{r.in1} we immediately find the following.
\begin{remark}\label{r.in2}
Let $\tilde{V}(t)$ be as in Remark \ref{r.in1};
let the assumptions of Theorem \ref{restatebis} \textbf{a)} be satisfied, then
 for any $d \in J_{+\infty}$ we have $\x(t,\tau; \Q_s(d,\tau)) \in \tilde{V}(t)$ whenever
$t \ge \tau$.
\end{remark}

\subsection{Chaotic patterns in forward time: setting \eqref{TandKnu}
}\label{proof.futurebis}
In this section we adapt the argument of Section \ref{proof.future} to reconstruct the set
$J_{+\infty}^{\ee^+}(\tau, \TT^+)$ when $\TT^+$ satisfies \eqref{TandKnu}, \eqref{TandKnu_tau+} instead of \eqref{TandKnunew},
\eqref{TandKnunew_tau+},
then we prove the estimates concerning
the
$\alpha_j$ to conclude the proof of Theorem \ref{restatefuture}.

In the whole subsection  we always assume
 that the hypotheses of Theorem~\ref{restatefuture} are satisfied in the weaker setting of case \textbf{c)} unless differently stated;
in fact
we simply need to repeat all the estimates performed in Section \ref{proof.future} involving $\Delta_j$, defined in \eqref{defSS}, by
replacing
$B_{j}$ by $\l1$
for any $j$,
so we will be rather sketchy.
{}From \eqref{TandKnu} we see that   $\Delta_j$ satisfies
\begin{equation}\label{defDejbis}
\Delta_j \ge 2 K_0(1+\nu) |\ln(\ep)| + 2 \l1 .
\end{equation}
 Repeating \eqref{Delta1} we find that
$\textrm{exp}\left(-\frac{5(\Delta_1+ \l1)}{ 2\sTfwd} \right) \le  \ep^{\frac{1+\nu}{\und{\si}}}$, so we can define
\begin{equation}\label{I1bis}
 I^{\Lambda}_1=  I_1^{\Lambda, \ee^+}(\tau, \TT^+) := \left[ \textrm{exp}\left(-\frac{5(\Delta_1+ \l1)}{2 \sTfwd}     \right) ,
 \ep^{(1+\nu)/\und{\si}}\right].
\end{equation}
Then from Theorem \ref{key} we obtain the analogue of Lemma \ref{key2}.
\begin{lemma}\label{key2bis}
  Let $\tau \in [b_0, b_1]$. Then there are $D^{\Lambda,1}_a, D^{\Lambda,1}_b \in   I^{\Lambda}_1$ such that
 \begin{equation*}
   \begin{split}
    \TTT_1(D^{\Lambda,1}_a,\tau) =     T_{2k_1}+\l1 ,  \qquad  \TTT_1(D^{\Lambda,1}_b,\tau) =     T_{2k_1}-\l1.
   \end{split}
 \end{equation*}
 Hence the image of the function $\TTT_1(\cdot, \tau) :  I^{\Lambda}_1  \to \R$
 contains the closed interval $[\SS_1-\l1 , \SS_1+\l1]$.
\end{lemma}
Then, we define the closed (possibly disconnected) set
\begin{equation}\label{I1hatbis}
\begin{gathered}
\hat{I}^{\Lambda}_1 =\hat{I}_1^{\Lambda, \ee^+}(\tau, \TT^+):=   \left\{ d \in I_1^\Lambda  \mid  T_{2k_1}-\l1 \le \TTT_1(d,\tau) \le
T_{2k_1}+\l1
\right\}.
\end{gathered}
\end{equation}
 Arguing as in Remark \ref{closest} we can select ``the compact interval closest to the origin"
$\check{I}^{\Lambda,\ee^+}_1(\tau, \TT^+)=\check{I}^{\Lambda}_1\subset \hat{I}^{\Lambda}_1$
   with the following property
   \begin{equation*}\label{prop1lambda}
     \begin{split}
 &   \TTT_1(\cdot,\tau):  \check{I}^{\Lambda}_1 \to  [T_{2k_1}-\l1 , T_{2k_1}+\l1] \qquad \textrm{is surjective}.
                                 \end{split}
   \end{equation*}
   Then, repeating word by word the proof of Lemma \ref{ep1} we obtain the following.
\begin{lemma}\label{ep1bis}
  There are $A^{\Lambda,-}, A^{\Lambda,+} \in \check{I}^{\Lambda}_1 $ such that $d_1(A^{\Lambda,-},\tau)=-\ep^{(1+\nu)/\und{\si}}$ and
  $d_1(A^{\Lambda,+},\tau)=\ep^{(1+\nu)/\und{\si}}$.
\end{lemma}

Let us set
\begin{equation}\label{hatAbis}
  \begin{array}{l}
    \check{A}^{\Lambda,-}:= \min \{ d \in \check{I}_1^{\Lambda,\ee^+} \mid  d_1(d,\tau)= -\ep^{(1+\nu)/\und{\si}}\}, \\
    \check{A}^{\Lambda,+}:= \min \{ d \in \check{I}_1^{\Lambda,\ee^+} \mid  d_1(d,\tau)= \ep^{(1+\nu)/\und{\si}}\} ,\\
    \underline{D^\Lambda_1}= \min \{ \check{A}^{\Lambda,-}; \check{A}^{\Lambda,+} \}, \qquad \overline{D^\Lambda_1}= \max \{
    \check{A}^{\Lambda,-}; \check{A}^{\Lambda,+} \}
  \end{array}
\end{equation}
and
\begin{equation}\label{J1ebisla}
  J_1^{\Lambda}=J_1^{\Lambda,\ee^+}(\tau,\TT^+):=[\underline{D^{\Lambda}_1}, \overline{D^{\Lambda}_1}] \subset I^{\Lambda}_1.
\end{equation}
Then we define $\alpha_{k_1}$ as in \eqref{defalpha1}, i.e.
 \begin{equation*}
   \al_{k_1}=\al_{k_1}^{\ee^+}(\ep,\tau, \TT^+)= \TTT_1(d,\tau)- \SS_1.
 \end{equation*}

Repeating the argument of Remark \ref{time1.missed} and the proof of Lemma \ref{l.step1} with no changes, we
find the following.
\begin{lemma}\label{l.step1bis}
  Whenever $d \in J^{\Lambda}_1$ and $\tau \in [b_0,b_1]$, then
   $$ [ T_1, T_{2k_1-1}  ]  \subset [\tau+ T_b ,  \TTT_1(d,\tau)-T_a].  $$
    Further, the trajectory $\x(t,\tau; \Q_s(d,\tau))$ satisfies $\bs{C^+_{\ee^+}}$ for any
     $\tau \le t \le T_{2 k_1} +\al_{k_1}$.
\end{lemma}

\begin{proposition}\label{p.step1bis}
  Let the assumptions of Theorem \ref{restatefuture} be satisfied, and let
  $J^{\Lambda}_1$ be as in \eqref{J1ebisla}.
  Fix $\tau\in[b_0,b_1]$. Then, the function
  $$ \TTT_1(\cdot ,\tau): J_1^{\Lambda} \to  [\aup_{k_1} , \adown_{k_1}]$$
is well defined and $C^r$, while the function
$$ d_1(\cdot ,\tau): J_1^{\Lambda} \to  \R$$
is  $C^r$ and its image contains $[-\ep^{(1+\nu)/ \und{\si}} , \ep^{(1+\nu)/ \und{\si}}]$.

Further,   the trajectory $\x(t,\tau; \Q_s(d,\tau))$ satisfies $\bs{C^+_{\ee^+}}$
for $t \in [\tau,  \TTT_1(d,\tau)]\supset [\tau, T_{2k_1}-\l1]$  and $|\al_{k_1}^{\ee^+}| \le \l1$.
\end{proposition}

Iterating the argument we obtain the following.
\begin{proposition}\label{p.fwdnbis}
Let $\tau \in [b_0,b_1]$,  $\TT^+=(T_j^+) $ and $\ee^+ \in \EE^+$ be fixed as in Theorem \ref{restatefuture};
  then for any $n \le  j^+_\SS(\ee^+)$ (resp.\ for any $n \in \N$ if $\ee^+ \in \hat{\EE}^+$, see \eqref{efuture}),
there is a compact interval $J_n^{\Lambda}= J_n^{\Lambda,\ee^+}(\tau,\TT^+)$ such that the function
$$ \TTT_n(\cdot ,\tau): J_n^{\Lambda}\to  [\aup_{k_n}, \adown_{k_n}]$$
is well defined and $C^r$, while the function
$$ d_n(\cdot ,\tau): J^{\Lambda}_n \to  \R$$
is $C^r$ and its image contains $[-\ep^{(1+\nu)/ \und{\si}} , \ep^{(1+\nu)/ \und{\si}}]$.

Further, if we set    $\al_{k_n}^{\ee^+}(\ep)=\TTT_n(d,\tau)-\SS_{n}$,  the trajectory $\x(t,\tau; \Q_s(d,\tau))$ satisfies $\bs{C^+_{\ee^+}}$
for any $t \in [\tau,  \TTT_n(d,\tau)]$  and $|\al_{k_n}^{\ee^+}| \le \l1$.
\end{proposition}

For the proof of Proposition \ref{p.fwdnbis} we proceed again by induction: the $n=1$ case
 follows from Proposition \ref{p.step1bis}, so we show that the step $n$  follows from the step $n-1$.
Assume that $J^{\Lambda}_{n-1}$ is well defined.
Then we set $D^{\Lambda,n}_\# := \textrm{exp}\left(\textstyle\frac{-5 (\Delta_n+2 \l1) }{2\sTfwd}\right)$ and
\begin{equation}\label{Inbis}
    I^{\Lambda}_n= I^{\Lambda, \ee^+}_n(\tau,\TT^+):=
  \left\{ d \in J^{\Lambda}_{n-1} \mid    D^{\Lambda, n}_\#   \le d_{n-1}(d,\tau)    \le \ep^{(1+\nu)/ \und{\si}} \right\}.
\end{equation}
For any $d \in I^\Lambda_n$ we define $d_{n}(d,\tau)$, $\TTT_{n}(d,\tau)$, and $\TTT_{n-\frac{1}{2}}(d,\tau)$  as in
 \eqref{Tn},
so that $\TTT_n(d,\tau)>\TTT_{n-\frac{1}{2}}(d,\tau)>\TTT_{n-1}(d,\tau)$.

Then, as in Lemma \ref{key2-stepn} we prove the following.

 \begin{lemma}\label{key2-stepnbis}
  Let $\tau \in\R$. Then there are $D_a^{\Lambda,n}, D_b^{\Lambda,n} \in   I^{\Lambda}_n$ such that
 \begin{equation*}
   \begin{split}
    \TTT_n(D_a^{\Lambda,n},\tau) =     \SS_n+\l1 \,,\qquad \qquad  \TTT_n(D_b^{\Lambda,n},\tau) =    \SS_n-\l1.
   \end{split}
 \end{equation*}
 Hence the image of the function $\TTT_n(\cdot, \tau) :  I^{\Lambda}_n \to \R$
 contains the closed interval $[\aup_{k_n} , \adown_{k_n}]$.
\end{lemma}

 Then, similarly to \eqref{Inhat} we define
\begin{equation}\label{Inhatbis}
\begin{split}
\hat{I}^{\Lambda}_n= \hat{I}_n^{\Lambda,\ee^+}(\tau,\TT^+):= &   \left\{ d \in I^{\Lambda}_{n}  \mid
  \aup_{k_n} \le \TTT_{n}(d,\tau) \le \adown_{k_n}   \right\},
\end{split}
\end{equation}
which is closed, non-empty and
 \begin{equation}\label{propnbis}
 \begin{split}
 &  \TTT_n(\cdot,\tau): \hat{I}^{\Lambda}_n \to [\aup_{k_n} , \adown_{k_n}]   \qquad \textrm{is surjective.}
\end{split}
 \end{equation}
Reasoning as in Remark \ref{closest} we denote by $\check{I}^{\Lambda}_n=\check{I}_n^{\Lambda,\ee^+}(\tau,\TT^+)$ the
 ``closed interval closest to $0$'' having  property  \eqref{propnbis}.

 Arguing as in Lemma \ref{ep1bis}, we obtain the following.
\begin{remark}\label{epnbis}
  There are $A^{\Lambda,-}_n, A^{\Lambda,+}_n \in \check{I}^{\Lambda}_n $ such that $d_n(A^{\Lambda,-}_n,\tau)=-\ep^{(1+\nu)/\und{\si}}$ and
  $d_n(A^{\Lambda,+}_n,\tau)=\ep^{(1+\nu)/\und{\si}}$.
\end{remark}
Then we set
\begin{equation}\label{hatAnbis}
  \begin{array}{l}
    \check{A}^{\Lambda,-}_n:= \min \{ d \in \check{I}_n^{\Lambda,\ee^+} \mid  d_n(d,\tau)= -\ep^{(1+\nu)/\und{\si}}\}, \\
    \check{A}^{\Lambda,+}_n:= \min \{ d \in \check{I}_n^{\Lambda,\ee^+} \mid  d_n(d,\tau)= \ep^{(1+\nu)/\und{\si}}\} ,\\
    \underline{D^\Lambda_n}= \min \{ \check{A}^{\Lambda,-}_n; \check{A}^{\Lambda,+}_n \},
     \qquad \overline{D^\Lambda_n}= \max \{ \check{A}^{\Lambda,-}_n; \check{A}^{\Lambda,+}_n \},
  \end{array}
\end{equation}
and   we denote by
 \begin{equation}\label{Jnebis}
  J_n^{\Lambda}=J_n^{\Lambda, \ee^+}(\tau,\TT^+):=[\underline{D^\Lambda_n},\overline{D^\Lambda_n}]  \subset \check{I}^{\Lambda}_n
  \subset J_{n-1}^{\Lambda}.
\end{equation}

Further, applying Lemma \ref{forgottentimes}, we get the following.
\begin{lemma}\label{forgottentimesnbis}
  Let $d \in J_n^{\Lambda}$  and $c^*>0$ be as in Proposition \ref{forPropC}. Then
  \begin{eqnarray}
    &    \TTT_{n-1}(d,\tau) +T_b  \le  \TTT_{\frac{1}{2}}(d_{n-1}(d,\tau),\TTT_{n-1}(d,\tau)) \le \TTT_n(d,\tau)-T_a,  \label{Thalfnabis}   \\
    &     \|x(t,\tau ; \Q_s(d,\tau))\| \le  \frac{3}{4}  c^* \ep, \quad  \forall t \in [\TTT_{n-1}(d,\tau) +T_b
    ,\TTT_n(d,\tau)-T_a],\label{Thalfnbbis}\\
   & \{ T_{2k_{(n-1)}+1} , T_{2k_n-1}  \} \subset [\TTT_{n-1}(d,\tau)+ T_b ,  \TTT_n(d,\tau)-T_a].
  \end{eqnarray}
\end{lemma}
Then, arguing as in Lemma \ref{l.stepn} we find the following.
  \begin{lemma}\label{l.stepnbis}
  Whenever $d \in J_n^{\Lambda}$ the trajectory $\x(t,\tau; \Q_s(d,\tau))$ satisfies $\bs{C^+_{\ee^+}}$ for any $\tau \le t \le
  \TTT_n(d,\tau)$,
  i.e. for any $\tau \le t \le \SS_n +\al_{k_n}$.
\end{lemma}

If  $\ee^+ \in \hat{\EE}^+$ we set
\begin{equation}\label{Jinftyla}
\begin{split}
  &J_{+\infty}^{\Lambda} = J^{\Lambda,\ee^+}_{+\infty}(\tau, \TT^+):= \cap_{n=1}^{+\infty}J_n^{\Lambda,
  \ee^+}(\tau, \TT^+) , \\
   & a_*^\Lambda=a_*^{\Lambda,\ee^+}(\tau, \TT^+)= \min \{a \in J_{+\infty}^{\Lambda}\}.
  \end{split}
\end{equation}
Notice that  we are intersecting infinitely many nested compact intervals, so $J_{+\infty}^{\Lambda}$
is a  non-empty compact and connected set, so the minimum $a_*^{\Lambda}$ exists.

\begin{proof}[\textbf{Proof of Proposition \ref{p.fwdnbis}}]
By construction the functions  $ \TTT_n(\cdot ,\tau): J^{\Lambda}_n \to  [\aup_{k_n} , \adown_{k_n}]  $
 and
$ d_n^{\Lambda}(\cdot ,\tau): J_n^{\Lambda} \to  \R$
are $C^r$  and the image of the latter contains $[-\ep^{(1+\nu)/ \und{\si}} , \ep^{(1+\nu)/ \und{\si}}]$.

Further from Lemma \ref{l.stepnbis} we know that $\x(t,\tau; \Q_s(d,\tau))$ has property $\bs{C^+_{\ee^+}}$,
 and by construction $|\al_{k_n}^{\ee^+} | \le |a^{\downarrow}_{k_n}-a^{\uparrow}_{k_n}| \le \l1$.
\end{proof}

\begin{remark}
  We think that, by asking for some non-degeneracy of the zeros of the Melnikov function, it might be possible to show
  that $J_{+\infty}^{\Lambda}= \{d_* \}$ is a singleton if $\ee^+ \in \hat{\EE}^+$, as in the smooth case.
\end{remark}

We emphasize that Remark \ref{r.in2} holds in this setting, too.

 \begin{lemma}\label{estimate.aln}
Let  $\f^{\pm}$ and $\g$ be $C^r$ with $r   >1$ and $d \in J^{\Lambda}_{+\infty}=J^{\Lambda, \ee^+}_{+\infty}(\tau,\TT^+)$,
then $|\MM(\TTT_n(d,\tau))|
\le \omega_T(\ep)$ where $\omega_T(\cdot)$ is a continuous and strictly increasing
function such that
$\omega_T(0)=0$.

Further, if $r \ge 2$ then we can find $C_M>1$ such that $\omega_T(\ep)= C_M  \ep $.
\end{lemma}
\begin{proof}
Let $n \in \mathbb{N}$; since  $d \in J_{n}^\Lambda$ from Remark \ref{rMelni} we see that
\begin{equation}\label{dis.a}
\dist (\P_s(\TTT_n(d,\tau)), \P_u(\TTT_n(d,\tau)))= c \ep [\MM(\TTT_n(d,\tau))+\omega(\TTT_n(d,\tau),\ep)]
\end{equation}
where   $|\omega(\tau,\ep)| \le \bar{\omega}(\ep)$, and   $\bar{\omega}(\cdot)$ is a continuous and monotone increasing
function such that $\bar{\omega}(0)=0$.
Recall that, see \eqref{Q1T1},
\begin{gather*}
	\x(\TTT_n(d,\tau),\tau; \Q_s(d,\tau))
	=\x(\TTT_n(d,\tau),\TTT_1(d,\tau); \Q_s(d_1(d,\tau),\TTT_1(d,\tau)))
	=\cdots\\
	=\x(\TTT_n(d,\tau),\TTT_{n-1}(d,\tau); \Q_s(d_{n-1}(d,\tau),\TTT_{n-1}(d,\tau)))
	=\PPP_1(d_{n-1}(d,\tau), \TTT_{n-1}(d,\tau))
\end{gather*}
 and
$0<d_{n-1}(d,\tau) \le \ep^{\frac{1+\nu}{\und{\si}}}$ since $d \in J_n^{\Lambda} \subset I^{\Lambda}_{n}$ (see  \eqref{Inbis}).
Thus by Theorem \ref{key}, we find
\begin{equation}\label{dis.b}
\begin{split}
 & \dist (\x(\TTT_n(d,\tau),\tau; \Q_s(d,\tau)), \P_u(\TTT_n(d,\tau)))  \le   d_{n-1}(d,\tau)^{\sfwd - \mu}  \\ &
 \le\ep^{\frac{1+\nu}{\und{\si}}(\sfwd - \mu)} \le\ep^{\frac{1+\nu_0}{\und{\si}}(\sfwd - \mu_0)} = \ep^{2+\bar{\nu}}
\end{split}
\end{equation}
where $\bar{\nu}=\frac{1+\nu_0}{\und{\si}}(\sfwd - \mu_0)-2\geq
	\frac{3\bar{\si}}{\und{\si}}(\sfwd_+\sfwd_--\frac{1}{4}\und{\si}^2)-2
	\geq 3\bar{\si}(\sfwd_--\frac{1}{4}\und{\si})-2\geq 3(1-\frac{1}{4})-2=\frac{1}{4}>0$
cf.~\eqref{mu0} and \eqref{defsigma}.

Further by construction
\begin{equation}\label{dis.c}
\dist (\x(\TTT_n(d,\tau),\tau; \Q_s(d,\tau)), \P_s(\TTT_n(d,\tau))) = |d_n(d,\tau)| \le  \ep ^{(1+\nu)/\und{\si}}\le \ep^{1+\nu}.
\end{equation}
Hence, from \eqref{dis.a}, \eqref{dis.b}, \eqref{dis.c}, we find
\begin{equation*}
\begin{split}
   & c \ep| \MM(\TTT_n(d,\tau))+\omega(\TTT_n(d,\tau),\ep)|
   =  | \dist (\P_u(\TTT_n(d,\tau)),  \x(\TTT_n(d,\tau),\tau; \Q_s(d,\tau)))-\\
   &  -\dist (\P_s(\TTT_n(d,\tau)),  \x(\TTT_n(d,\tau),\tau; \Q_s(d,\tau)))|  \le  2 \ep^{2}.
   \end{split}
\end{equation*}
Thus we conclude setting $\omega_T(\ep)=   \bar{\omega}(\ep)+ \frac{2}{c} \ep$.

If $r \ge 2$ the lemma follows observing that there is $\bar{C}_M \ge 1$ such that $\bar{\omega}(\ep) \le \bar{C}_M \ep$
and setting $C_M=\bar{C}_M+ 2 /c$.
\end{proof}

 \begin{remark}
  From an analysis of the latter proof it follows that Lemma \ref{estimate.aln} applies not only to the setting of Theorem
  \ref{restatefuture}, but also to the setting of Theorem~\ref{restatebis}.
\end{remark}

 Now we proceed to estimate $|\al_{k_n}(\ep)|$,  see \eqref{defaln}. We recall that if  \eqref{TandKnunew} holds and we only
assume \assump{P1}, then we are just able to show that $|\al_{k_n}(\ep)|< B_{ 2k_n}$,  and the sequence $B_n$ might be unbounded.
However assuming some very weak non-degeneracy condition, even relaxing slightly \eqref{TandKnunew} in \eqref{TandKnu}, we can get  better
estimates.
In particular if we assume \eqref{minimal} then $|\al_{k_n}(\ep)|$ is bounded, and if we assume \eqref{isogen} and $\lzero=0$ then
$|\al_{k_n}(\ep)| \to 0$ as $\ep \to 0$.

\begin{lemma}\label{lemma.alphan}
 Assume \assump{P1} and consider a sequence $\TT=(T_n)$ satisfying \eqref{TandKnu}, \eqref{TandKnu_tau+}. Let $d \in
 J_{+\infty}=J_{+\infty}^{\ee^+}(\tau,\TT^+)$,
 then
 for any $n \in \mathbb{N}$ we have the following:
  \begin{enumerate}
   \item if \eqref{minimal} holds and $r>1$ then $|\al_{k_n}^{\ee^+}(\ep)| \le \l1$;
      \item if \eqref{isogen} holds and $r>1$ then there is a monotone increasing and continuous function $\omega_{\alpha}(\ep)$ such that
    $\omega_{\alpha}(0)=0$ and
    $|\al_{k_n}^{\ee^+}(\ep)| \le \lzero+\omega_{\alpha}(\ep)$. In particular if $\lzero=0$, i.e. \eqref{isolated}
    holds, then $|\al_{k_n}^{\ee^+}(\ep)| \to 0$ as $\ep \to 0$;
     \item if \eqref{nondeg} holds and $r \ge 2$ then $|\al_{k_n}^{\ee^+}(\ep)| \le
        c_{\alpha}\, \ep$, where   $c_{\alpha}:=\frac{2 C_M}{C}> 0$ and $C_M>1$ is as in Lemma \ref{estimate.aln}.
  \end{enumerate}
\end{lemma}
\begin{proof}
In setting 1 by construction we see that $\TTT_n(d,\tau) \in [ \aup_{k_n}, \adown_{k_n}]$ hence
$|\al_{k_n}^{\ee^+}(\ep)| \le  \adown_{k_n}-\aup_{k_n} \le \l1$.

Now we consider
 setting 2;  assume first $\TTT_n(d,\tau) \in [T_{2 k_n}+\lzero, \adown_{k_n}]$ so that $\al_{k_n}^{\ee^+}(\ep)-\lzero\geq 0$.
 Then, from   \eqref{isogen}, we find
 $$|\MM(\TTT_n(d,\tau)  )|=|\MM(T_{2 k_n}+ \al_{k_n}^{\ee^+}(\ep))| \ge \omega_M(  \al_{k_n}^{\ee^+}(\ep)-\lzero ).$$
Moreover from Lemma \ref{estimate.aln} we find
$$|\MM(T_{2 k_n}+ \al_{k_n}^{\ee^+}(\ep))|= |\MM(T_{2 k_n}+
 \al_{k_n}^{\ee^+}(\ep))- \MM (T_{2 k_n})|  \le   \omega_T(\ep).$$
So, setting $ \omega_{\alpha}(\ep)=\omega_M^{-1}[\omega_T(\ep)]$ we find
$|\al_{k_n}^{\ee^+}(\ep)| \le \lzero+\omega_{\alpha}(\ep)$.
  Since
$\omega_T(\ep)$ is independent of $n$,   the same holds for $\omega_{\alpha}(\ep)$ as well.
The proof when  $\TTT_n(d,\tau) \in [\aup_n, T_{2 k_n}-\lzero]$ is analogous and it is omitted.
 Further  when $\TTT_n(d,\tau) \in [ T_{2 k_n}-\lzero , T_{2 k_n}+\lzero]$ we get $|\al_{k_n}^{\ee^+}(\ep)| \le \lzero$,
and the claim follows.

In setting 3, using Lemma \ref{estimate.aln} we find
$$
\left|\MM'(T_{2 k_n}) \frac{\al_{k_n}^{\ee^+}(\ep)}{2}  \right|\le |\MM'(T_{2 k_n}) \al_{k_n}^{\ee^+}(\ep) + o( \al_{k_n}^{\ee^+}(\ep))|
=|\MM(\TTT_n(d,\tau))| \le C_M\ep.
$$
Since $|\MM'(T_{2 k_n})|>C$, the claim follows  setting  $c_{\alpha}=\frac{2 C_M}{C}$.
\end{proof}

 From the argument of the proof of Lemma \ref{lemma.alphan} we also get the following.
 \begin{remark}
          Assume that there is $k$ such that $|\MM'(T_{2k})|=C>0$, then there is $c_{\al}^k>0$ (independent of $\ee^+$ and $\ep$)
          such that
          $|\al_{k}^{\ee^+}| \le c_{\al}^k \ep$ for any $\ee^+$ such that $\ee^+_k=1$.
          \\
          Further if there is $j$ such that  $2j+1\leq r$,  $\MM^{(l)}(T_{2k})=0$ for any $l=1, \ldots, 2j$ and
          $|\MM^{(2j+1)}(T_{2k})|=C>0$
          then there is $c_{\al}^k>0$ (independent of $\ee^+$ and $\ep$)
          such that
          $|\al_{k}^{\ee^+}| \le c_{\al}^k \ep^{1/(2j+1)}$ for any $\ee^+$ such that $\ee^+_k=1$.
        \end{remark}

\begin{proposition}\label{est.al}
 Theorem \ref{restatefuture} holds for any $\ee^+ \in \hat{\EE}^+$.
\end{proposition}
\begin{proof}
Let $d \in J^{\Lambda}_{+\infty}$; applying Proposition \ref{p.fwdnbis}   we obtain
that $\x(t,\tau; \Q_s(d,\tau))$ has property $\bs{C^+_{\ee^+}}$ for any $n \in \N$.
Further from Lemma \ref{lemma.alphan} we obtain the estimates on $\al_n^{\ee^+}(\ep)$.
\end{proof}

\begin{proof}[\textbf{Proof of Theorem \ref{restatefuture}}]
The proof of Theorem \ref{restatefuture} follows the lines of the proof of Proposition \ref{newhalffuture}, profiting also of the estimates
for the $\al_{k_n}^{\ee^+}(\ep)$ given in Lemma \ref{lemma.alphan}.
\end{proof}

\subsection{Construction of the chaotic patterns in backward time}\label{proof.past}

In this section we deduce    Theorem \ref{restatebis} part \textbf{b)} from  Theorem \ref{restatebis} part \textbf{a)}
and Theorem \ref{restatepast} from Theorem \ref{restatefuture}
using a standard inversion of time argument.

\begin{proof}[\textbf{Proof of Theorem \ref{restatebis} part \textbf{b)}}]
Let us set $\underline{\f}^{\pm}(\x)=- \f^{\pm}(\x)$ and $\underline{\g}(t,\x,\ep)=-\g(-t,\x,\ep)$.
Let $\x(t)$ be a solution of \eqref{eq-disc} then $\underline{\x}(t)=\x(-t)$ is a solution of
\begin{equation}\label{eq.mod}
  \dot{\x}=\underline{\f}^\pm(\x)+\ep\underline{\g}(t,\x,\ep),\quad \x\in\Om^\pm.
\end{equation}
  Further, if $\f^{\pm}$ and $\g$ satisfy the assumptions of  Theorem \ref{restatebis} part \textbf{b)}  then
  $\underline{\f}^{\pm}$ and $\underline{\g}$ satisfy the assumptions of  Theorem \ref{restatebis} part \textbf{a)}.

  Let $\TT^-=(T^-_j)$, $j\in\Z^-$ be a sequence satisfying \eqref{TandKnunew}
  for $j\in \Z^-$, $j \le -2$ and \eqref{TandKnunew_tau-} for some $\tau\in[b_0,b_1]$, and let $\ee^-=(\ee_j)$, $j\in \Z^-$; we set
  $\underline{\TT^+}=(\underline{T^+_j})$, $j \in \Z^+$, where $\underline{T^+_j}= -T^-_{-j}$, so
   that   $\underline{\TT^+}$ satisfies  \eqref{TandKnunew}  for $j \in \Z^+$ where $B_j=B_{-j}$ and \eqref{TandKnunew_tau+} with
   $\tau\in[-b_1,-b_0]$.
   Further we set $\underline{\ee^+}=(\ee_{-j})$, $j \in \Z^-$ so that $\underline{\ee^+} \in \EE^+$.

   Now we can apply Theorem \ref{restatebis}   \textbf{a)} to \eqref{eq.mod} with the sequences $\underline{\TT^+}$
   and $\underline{\ee^+} \in \EE^+$, and we get the existence of the
    compact set $X^+= X^+(\underline{\ee^+},\tau, \underline{\TT^+})$ and the sequence
     $\underline{\alpha_j}(\ep)= \underline{\al_j}^{\underline{\ee^+}}(\ep,\tau,\underline{\TT^+})$
such that   for any $\underline{\xxi}\in X^+$  the trajectory $\x(t,\tau ;  \underline{\xxi})$ has the   property
 $\bs{C_{\ee^+}^+}$.

Let us set $X^-(\ee^-,\tau, \TT^-) = X^+(\underline{\ee^+},\tau, \underline{\TT^+})$
and $\alpha_j(\ep)=-\underline{\alpha_{-j}}(\ep)$.
Going back to    the original system \eqref{eq-disc}, we see that   the trajectory $\x(t, \tau ; \underline{\xxi})$
has property  $\bs{C_{\ee^-}^-}$, so Theorem \ref{restatebis}  \textbf{b)} follows.
\end{proof}

 \begin{proof}[\textbf{Proof of Theorem \ref{restatepast}}]
In order to prove Theorem \ref{restatepast}   we can simply perform an inversion of time
argument exactly as in the proof of Theorem \ref{restatebis}    \textbf{b)} just above, and then apply   Theorem \ref{restatefuture}.
\end{proof}

 \begin{proof}[\textbf{Proof of Corollary \ref{r.allbis}}]
Let us set
\begin{equation}\label{tildecstarbis}
\tilde{c} = 2 \sup \{\|\dot{\ga}^+(t)\|; \|\dot{\ga}^-(s)\| \mid s<0<t \},
\end{equation}
and observe that $\tilde{c}>0$ is finite.
We start with the following observation.
\\ $\bullet$ \ \emph{Claim 1: for any $j \in \Z$,  $t\in\R$,  we get}
  \begin{equation}\label{easy}
     \| \ga(t+ \alpha_j(\ep))- \ga(t)\| \le  \tilde{c} |\alpha_j(\ep)|.
   \end{equation}
  In fact, assume first that $t+ \alpha_j(\ep)$ and $t$ are both negative; then
  $$ \| \ga(t+ \alpha_j(\ep))- \ga(t)\|=\| \ga^-(t+ \alpha_j(\ep))- \ga^-(t)\| \le  \tilde{c} |\alpha_j(\ep)|.$$
  In the same way we prove \eqref{easy} when   $t+ \alpha_j(\ep)$ and $t$ are both positive.

  Now assume $t<0<t+ \alpha_j(\ep)$. Then
  \begin{gather*}
  	\| \ga(t+ \alpha_j(\ep))- \ga(t)\|=\| \ga^+(t+ \alpha_j(\ep))- \ga^+(0) \| + \| \ga^-(0)-\ga^-(t)\| \\
  \leq \tilde{c}(\alpha_j(\ep)+t) +\tilde{c} |t| = \tilde{c}[\alpha_j(\ep)-|t| +|t|]
  \le  \tilde{c} |\alpha_j(\ep)|.
  \end{gather*}
  The case where $t+ \alpha_j(\ep)<0<t$ can be handled in the same way, so the claim is proved.

Now from Claim 1 we easily get Claim 2
\\
  $\bullet$ \ \emph{Claim 2:
 Let the assumptions Corollary \ref{r.allbis} be satisfied;
 \begin{itemize}
   \item assume \eqref{nondeg}, then   for any $j \in \Z$ we find
   \begin{equation}\label{easier}
 \sup_{t \in \R} \| \ga(t-T_{2j}- \alpha_j(\ep))- \ga(t-T_{2j})\|  \le  \tilde{c}   c_{\alpha} \ep
     \end{equation}
  where $c_{\alpha}>0$ is given in  Lemma \ref{lemma.alphan} point 3  and it is independent of $j$;
   \item assume
   \eqref{isolated}, then   for any $j \in \Z$ we find
 $$  \sup_{t \in \R} \| \ga(t-T_{2j}- \alpha_j(\ep))- \ga(t-T_{2j})\|   \le  \tilde{c}   \omega_{\alpha} (\ep ) =:
 \bar{\omega}_{\alpha} (\ep )$$
  where  $\omega_{\alpha} (\ep )$ is given in   Lemma \ref{lemma.alphan} point 2
  and it is independent of $j$.
 \end{itemize}}
Corollary \ref{r.allbis} follows from  Theorems \ref{restatefuture} and \ref{restatepast}
respectively,
 simply using Claim 2, the triangular inequality and choosing
 \begin{equation}\label{deftildecstar}
 \tilde{c}^*=c^*+ \tilde{c}   c_{\alpha}.
 \end{equation}
\end{proof}

\begin{proof}[\textbf{Proof of Corollary \ref{c.localization}}]
 The set $\aleph^+$ described in Theorem \ref{restatefuture} (see \eqref{XandSigma+}) is defined   by
 \begin{equation}\label{defSigma+}
   \aleph^+:= \{  \Q_s(d,\tau) \mid d \in J^{\Lambda}_{+\infty}, \, \ee \in \EE^+ \}.
 \end{equation}
 By construction $\aleph^+ \subset L^0(c^* \ep)$ and
 $$\|\Q_s(d,\tau)-  \P_s(\tau)\| \le \DDD(\Q_s(d,\tau), \P_s(\tau)) \le \ep^{\frac{1+\nu}{\und{\sigma}}},$$
 so the part of  Corollary  \ref{c.localization} concerning $\aleph^+$   immediately follows;
 the part concerning $\aleph^-$ is analogous.
\end{proof}

In fact from the argument of the proof of Proposition \ref{p.fwdn} we also get the following result.
\begin{proposition}\label{closeW}
  The trajectories $\x(t,\tau; \xxi)$ constructed via Theorems \ref{restatefuture} and \ref{restatebis} \textbf{a)}  are such that
  \begin{equation}\label{e.closeW}
    \x(t,\tau; \xxi) \in B(\tilde{W}(t), \ep^{(1+\nu)/2})
  \end{equation}
  for any $t \ge \tau$.

  Analogously the  trajectories $\x(t,\tau; \xxi)$ constructed via Theorems \ref{restatepast} and \ref{restatebis} \textbf{b)}
  are such that \eqref{e.closeW} holds for any $t \le \tau$.
\end{proposition}
\begin{proof}
   Let us focus on the trajectories $\x(t,\tau; \xxi)$ which have property $\bs{C_{\ee^+}^+}$ which have been constructed via   Theorem
   \ref{restatebis}
   \textbf{a)}, the other cases being analogous.
   These trajectories
   are in fact built up by  Proposition \ref{p.fwdn}
    by choosing $\xxi= \Q_s(d,\tau)$ and then observing
   that
   $$ \x(\TTT_n(d,\tau),\tau; \Q_s(d,\tau))=    \Q_s(d_n(d,\tau), \TTT_n(d,\tau)),$$
   where $d \in [0, \ep^{(1+\nu)/\und{\si}}]$ and $d_n(d,\tau) \in [0, \ep^{(1+\nu)/\und{\si}}]$, for any $n \le j_\SS^+(\ee^+)$.
   Then  from Theorem \ref{keymissed} we find
    \begin{equation*}
      \begin{split}
          &   \|\x(t, \TTT_n(d,\tau); \Q_s(d_n(d,\tau), \TTT_n(d,\tau)))-\x(t, \TTT_n(d,\tau); \P_s(\TTT_n(d,\tau)))\| \\
           &  \le  |d_n(d,\tau)|^{\sfwd_+ - \mu }
   \le \ep^{1+\nu -\mu (1+\nu)/\und{\si}} \le  \ep^{(1+\nu)/2}
      \end{split}
    \end{equation*}
  for any $\TTT_n(d,\tau) \le t \le \TTT_{n+1/2}(d,\tau)$,  and any   $n < j_\SS^+(\ee^+)$. Analogously
    \begin{equation*}
      \begin{split}
 &  \|\x(t, \TTT_n(d,\tau); \Q_s(d_n(d,\tau), \TTT_n(d,\tau)))-\x(t, \TTT_{n+1}(d,\tau);  \P_u(\TTT_{n+1}(d,\tau)))\| \\
 & \le  |d_n(d,\tau)|^{\sfwd_+ - \mu }
   \le \ep^{1+\nu -\mu (1+\nu)/\und{\si}} \le  \ep^{(1+\nu)/2}
  \end{split}
    \end{equation*}
    for any $\TTT_{n+1/2}(d,\tau)  \le t \le     \TTT_{n+1}(d,\tau)$, and any   $n < j_\SS^+(\ee^+)$.
   Then \eqref{e.closeW} follows.
\end{proof}

\label{proof.local}
Now we turn to consider the proof of Corollary \ref{localize1}; we
   develop the argument in the setting
 of Theorem \ref{restatebis}, but the proof works in the setting of Theorems \ref{restatefuture} and \ref{restatepast} with no  changes.

Using known arguments from exponential dichotomy theory, see e.g.\ \cite[\S 6.2]{CDFP}, we see that   there is $c_{\gamma}>1$ such that
\begin{equation}\label{est.stableunstablebis}
\begin{split}
   &   \|\x(\theta_s+\tau,\tau; \P_s(\tau))\| < c_{\gamma} \eu^{-\frac{2 | \la_s^+|}{3}\theta_s}, \\
  &   \|\x(\theta_u+\tau,\tau; \P_u(\tau))\| < c_{\gamma} \eu^{\frac{2 \la_u^- }{3}\theta_u}
\end{split}
\end{equation}
whenever $\theta_u \le 0 \le \theta_s$, for any $\tau \in \R$.

\begin{lemma}\label{est.mani}
 Let the assumptions of Theorem \ref{restatebis} be satisfied,
 let $d \in J_1$, $\tau \in [b_0, b_1]$; then
 $$ \|\x(T_1,\tau; \P_s(\tau))\| \le c_{\gamma} \ep^{\frac{2(1+\nu)}{3}},$$
 $$ \|\x(T_1,\TTT_1(d,\tau); \P_u(\TTT_1(d,\tau)))\| \le c_{\gamma} \ep^{\frac{2(1+\nu)}{3}}.$$
\end{lemma}

\begin{proof}  From
\eqref{TandKnunew},
we find
$$T_1-\tau =T_1-T_0+ T_0-\tau \ge K_0(1+\nu)|\ln(\ep)| +B_0+ T_0-\tau  \ge  K_0(1+\nu)|\ln(\ep)| .$$
Hence from \eqref{est.stableunstablebis}  and \eqref{defK0-new}  we find
$$\|\x(T_1,\tau; \P_s(\tau))\| \le c_{\gamma} \eu^{- \frac{2 |\la_s^+|}{3}(T_1-\tau)} \le  c_{\gamma} \ep^{\frac{2(1+\nu)}{3}}. $$
Analogously, by \eqref{TandKnunew},
$$\TTT_1(d,\tau)-   T_1  \ge T_2 - B_2 -T_1\ge K_0(1+\nu)|\ln(\ep)|    ; $$
then, from \eqref{est.stableunstablebis}  and \eqref{defK0-new}  we get
$$\|\x(T_1,\TTT_1(d,\tau); \P_u(\TTT_1(d,\tau)))\| \le c_{\gamma} \eu^{-\frac{ 2\la_u^-}{3}(\TTT_1(d,\tau)-   T_1)} \le  c_{\gamma}
\ep^{\frac{2(1+\nu)}{3}}.$$
\end{proof}

\begin{proposition}\label{p.localize1}
   Let the assumptions of Theorem \ref{restatebis} be satisfied; let $d \in J_1$, $\tau \in [b_0, b_1]$; then
   $$\|\x(T_1,\tau; \Q_s(d,\tau))\| \le    \ep^{\frac{1+\nu}{2}}.$$
\end{proposition}
\begin{proof}
  By construction $\TTT_{\frac{1}{2}}(d,\tau) \in [T_0, T_{2k_1}]$. Assume first that $\TTT_{\frac{1}{2}}(d,\tau) \ge T_1$ so that
  $[T_0,T_1] \subset [T_0, \TTT_{\frac{1}{2}}(d,\tau)]$.
  Then, from  estimate \eqref{keymissed.es-}  in Theorem \ref{keymissed}, using the fact that $d \in J_1 \subset J_0$ and Lemma
  \ref{est.mani}, we find
  \begin{equation}\label{e.localize1}
  \begin{split}
         \|\x(T_1,\tau; \Q_s & (d,\tau))\| \le \|\x(T_1,\tau; \P_s(\tau))\|
       \\ & +
     \|\x(T_1,\tau; \Q_s(d,\tau))-   \x(T_1,\tau; \P_s(\tau)) \|    \le (c_{\gamma}+1)  \ep^{\frac{2(1+\nu)}{3}}.
  \end{split}
  \end{equation}

  Now assume $\TTT_{\frac{1}{2}}(d,\tau) < T_1$ so that $[ T_1, \TTT_{1}(d,\tau)] \subset [ \TTT_{\frac{1}{2}}(d,\tau), \TTT_{1}(d,\tau)]$.
 As in the proof of Lemma \ref{forgottentimes}, let us set      $D_1=D_1(d,\tau):= \dist (\PPP_1(d,\tau),\P_u(\TTT_1(d,\tau)))$
 and observe that, by definition, see \eqref{Q1T1},
  $$\x(t,\tau; \Q_s(d,\tau)) \equiv   \x(t,\TTT_1(d,\tau); \PPP_1(d,\tau)) \equiv
   \x(t,\TTT_1(d,\tau); \Q_u(D_1,\TTT_1(d,\tau))), $$
   for any $t \in \R$.
    Further, recalling \eqref{twice}, we find   $D_1 \le \ep^{ \frac{3\sfwd_-}{4}(1+\nu)}$;
then, again from Theorem \ref{keymissed} (see \eqref{keymissed.eu+}),
   and Lemma  \ref{est.mani}, we find
  \begin{equation}\label{e.localize2}
  \begin{split}
       &  \|\x(T_1,\tau; \Q_s(d,\tau))\|   \le \|\x(T_1,\TTT_1(d,\tau);  \P_u(\TTT_1(d,\tau)))\|  \\
       & +
     \|\x(T_1,\TTT_1(d,\tau); \Q_u(D_1,\TTT_1(d,\tau)))-   \x(T_1,\TTT_1(d,\tau); \P_u(\TTT_1(d,\tau))) \|  \\
       & \le c_{\gamma}  \ep^{\frac{2(1+\nu)}{3}}+  D_1^{\sbwd_- -\mu} \le   c_{\gamma} \ep^{\frac{2(1+\nu)}{3}} + \ep^{\frac{9(1+\nu)}{16}}
       \le
       \ep^{\frac{1+\nu}{2}}
  \end{split}
  \end{equation}
and the assertion follows.
\end{proof}

Now we prove Corollary \ref{localize1}
\begin{proof}[\textbf{Proof of Corollary \ref{localize1}}]\label{proof1}
  The part of the proof of Corollary \ref{localize1} concerning $\aleph^+_1$
   follows from Proposition \ref{p.localize1} and Remark \ref{r.in2}; the part concerning
   $\aleph^-_{-1}$ follows from an inversion of time argument as in \S \ref{proof.past}.
  \end{proof}

\section{Semi-conjugacy with the Bernoulli shift}\label{proof.Bernoulli}

Let $\si:{\cal E}\to{\cal E}$ be the  (forward)  Bernoulli shift that is
$\si(\ee):=(\ee_{m+1})_{m\in\Z}$.
In this section, adapting a classical argument, we show that the  action of the forward flow of \eqref{eq-disc}
on the sets $\aleph^+$    constructed via Theorem  \ref{restatefuture} (see \eqref{XandSigma+})  is semi-conjugated  with the   forward
Bernoulli shift,
while the backward flow of  \eqref{eq-disc}
on  $\aleph^-$  constructed via Theorem \ref{restatepast} (see \eqref{XandSigma-})
is semi-conjugated with the backward Bernoulli shift.

 We obtain partial results also in the setting of   Theorem \ref{restatebis}.
In the whole section we follow quite closely \cite[\S 6]{BF11}, from which most of the ideas are borrowed.
See also \cite[\S 5]{BF12} for a survey.

Set
 \begin{equation*}\label{defE}
\begin{gathered}
\hat{\cal E}^+:=\left\{\ee\in{\cal E}^+ \mid   \sup\{m\in\Z^+ \mid \ee_m=1\}=\infty \right\}, \\
{\cal E}_0^+:=\left\{\ee\in{\cal E}^+ \mid  \sup\{m\in\Z^+ \mid \ee_m=1\}<\infty\right\}, \\
\hat{\cal E}^-:=\left\{\ee\in{\cal E}^- \mid \inf\{m\in\Z^- \mid \ee_m=1\}=-\infty\right\}, \\
{\cal \hat{E}}_0^-:=\left\{\ee\in{\cal E}^- \mid \inf\{m\in\Z^- \mid
\ee_m=1\}>-\infty\right\}.
\end{gathered}
\end{equation*}
Note that $\hat{\cal E}^+$, ${\cal E}_0^+$ are positively invariant while  $\hat{{\cal E}}^-$, ${\cal E}_0^-$ are
negatively invariant under the Bernoulli shift.
The set ${\cal E}$  becomes a
totally disconnected compact metric space with the distance
\begin{equation}\label{metric}
d(\ee',\ee'') = \sum_{m\in\Z}\frac{|\ee'_{m}-\ee''_{m}|}{2^{|m|+1}}\, ,
\end{equation}
 and the same happens to  its subsets  $\hat{\cal E}^+$, ${\cal E}_0^+$ and $\hat{\cal E}^-$, ${\cal E}_0^-$ if
we restrict the definition respectively to $\Z^+$ or to $\Z^-$.
Further   let us denote by $$\sigma^k= \stackrel{k \;\text{times}}{\overbrace{\sigma \circ \cdots \circ \sigma}}.$$

To fix the ideas let us consider the case of forward time and $\Z^+$, so let $\tau \in [b_0,b_1]$ and let
$\mathcal{T}^+= (T_m )$, $m\in\Z^+$,   be a  fixed sequence  of values satisfying \eqref{TandKnu}
and  \eqref{TandKnu_tau+}.
Following  \cite[\S 6]{BF11}
 we set
 $\mathcal{T}^{(k)}= ( T_{m+2k} )$  for any $k \in \Z$.

Let $X^+(\ee^+, \tau,\TT^+)$  and $\aleph^+(\tau, \TT^+)$
be the sets constructed via Theorem \ref{restatefuture};
we introduce the sets
\begin{equation}\label{defSk}
	\begin{gathered}
		\aleph_k^+ := \left\{ \xxi_k= \x(T_{k}, \tau; \xxi_0) \mid \xxi_0 \in X^+(\ee^+,\tau,\TT^+), \, \ee^+ \in \EE^+ \right\},\quad k \in
\Z^+,\\
		 \aleph_0^+:=\left\{ \xxi_0\in X^+(\ee^+,\tau,\TT^+)\mid \ee^+ \in \EE^+ \right\}.
	\end{gathered}
\end{equation}

Using \eqref{ej+=1w}   and \eqref{ej+=0w} we get the following.
 \begin{remark}\label{remtwo}
 Assume the hypotheses of Theorem \ref{restatefuture} \textbf{a)}, then using also Corollary \ref{r.allbis}  \textbf{a)} we find
\begin{equation}\label{e.alep}
   \aleph^+(\tau,\TT^+)=\aleph_0^+  \subset  B(\ga(0), \tilde{c}^* \ep)   , \qquad
\aleph_{2k}^+ \subset [B(\ga(0), \tilde{c}^* \ep) \cup B(\vec{0}, \tilde{c}^* \ep)]
\end{equation}
 for any $k \in \Z^+ $.

 Similarly  in the setting of Theorem \ref{restatefuture} \textbf{b)},  using also Corollary \ref{r.allbis}  \textbf{b)} we find
$$ \aleph^+(\tau,\TT^+)=\aleph_0^+  \subset  B(\ga(0), \tilde{c}^* \ep)   , \qquad
\aleph_{2k}^+ \subset [B(\ga(0), \omega(\ep)) \cup B(\vec{0}, \tilde{c}^* \ep)]$$
 for any $k \in \Z^+ $.

Finally  in the setting of Theorem \ref{restatefuture} \textbf{c)}, if we set $\bs{\Gamma^{\Lambda}}= \{ \ga(t)
\mid |t| \le \Lambda^1\}$ we find
$$ \aleph^+(\tau,\TT^+)=\aleph_0^+
  \subset B(\ga(0), \tilde{c}^* \ep)   , \qquad
\aleph_{2k}^+ \subset [B(\bs{\Gamma^{\Lambda}}, \tilde{c}^* \ep) \cup B(\vec{0}, \tilde{c}^* \ep)]$$
 for any $k \in \Z^+ $.
 \end{remark}

Further   $\x(t,\tau ; \xxi_0) \in B(\bs{\Gamma}, \sqrt{\ep})$
for any $t \ge \tau$; hence from Remark  \ref{data.smooth} we see that local
 uniqueness and continuous dependence on initial data is ensured for any
trajectory $\x(t, \tau; \xxi_0)$ such that $\xxi_0 \in \aleph_0^+$.
So we get the following.
\begin{remark}\label{remone}
  By construction,  $\xxi_0 \in \aleph^+_0$ if and only if  $\xxi_k= \x(T_{2k},\tau; \xxi_0) \in \aleph_{2k}^+$ and we have
  \begin{equation}\label{identity}
    \x(t,\tau; \xxi_0) \equiv \x(t,T_{2k}; \xxi_k)   \, , \quad \textrm{for any $t \in \R$}.
  \end{equation}
 \end{remark}

 Let $\xxi_0 \in X^+(\ee, \tau, \TT^+)$, then we set
  $\Psi_0(\xxi_0)=\ee$, so that
    $\Psi_0: \aleph_0^+ \to \EE^+$ is well defined and onto.

  Similarly, let $\xxi_k \in \aleph_{2k}^+$, then there is a  uniquely defined
  $\xxi_0\in \aleph_0^+$  such that $x(\tau, T_{2k}; \xxi_k)=\xxi_0$; further there is a uniquely defined
  $\ee \in \EE^+$ such that $\xxi_0 \in X^+(\ee, \tau,\TT^+)$. Let $\ee^k= \sigma^k(\ee)$, then we
  set $\Psi_k(\xxi_k)=\ee^k$, so that $\Psi_k: \aleph_{2k}^+ \to \EE$ is well defined and onto for any $k \ge 0$.

  \begin{remark}\label{unique.e}
 We   stress that  we might have
$\xxi_0' \ne \xxi_0''$, such that $\xxi_0',\xxi_0'' \in  X^+(\ee,\tau,\TT^+)$.
 So it follows that $\Psi_0$, and consequently  $\Psi_k$  for any $k>1$, might not be injective.
\end{remark}

\begin{proposition}\label{continuityine}
Let the assumptions of Theorem \ref{restatefuture}  hold, then
the map $\Psi_k: \aleph_{2k}^+ \to \EE^+$ is continuous.
\end{proposition}
\begin{proof}
Assume by contradiction that  $\Psi_k$ is discontinuous at some point  $\xxi \in \aleph_{2k}^+$. This means that  there is
 $\varpi>0$ such that for any $\delta>0$ we can find $\zz \in \aleph_{2k}^+$, $\| \zz- \xxi\| <\delta$  such that
$d(\ee, \ee')> \varpi$ where $\ee=(\ee_j)= \Psi_k( \xxi)$, $\ee'=(\ee_j')= \Psi_k (\zz)$.
From the definition of the distance $d$ in \eqref{metric} we get that there is $N>k$, independent of $\delta$, such that
$$\sum_{j \le N} |\ee_j-\ee'_j| \ge 1.$$
Hence there is $\ell \in \Z^+$, $\ell \le N$ such that $\ee_{\ell} \ne \ee'_{\ell}$.
 We assume for definiteness $\ee_{\ell} =1 $ and $\ee'_{\ell}=0$.

Assume the hypotheses of Theorem \ref{restatefuture} \textbf{c)}.

Let us consider the trajectories $\x(t, T_{2k} ; \xxi)$ and $\x(t, T_{2k} ; \zz)$ of \eqref{eq-disc}:
since $\| \xxi-\zz\| < \delta$,
 using the continuous dependence on initial data we can choose $\delta>0$ so that
\begin{equation}\label{no1new}
\sup \{ \|\x(t, T_{2k} ; \xxi)-\x(t, T_{2k} ; \zz)\| \mid t \in [ \tau; T_{2N}] \} <  K:=\min \left\{ \frac{\|\ga(t)\|}{4} \,\Big{|}\, |t| \le
\Lambda^1
\right\}.
\end{equation}
On the other hand, from  Theorem \ref{restatefuture} \textbf{c)} we find that $\|\x(T_{2\ell}, T_{2k} ; \xxi)-\ga(-\al_\ell^{\ee^+}(\ep))\|\le
\tilde c^* \ep$
and
$\|\x(T_{2\ell}, T_{2k} ; \zz)\|\le \tilde c^* \ep$.

So choosing
$2 \tilde{c}^* \ep\leq 2 \tilde{c}^* \ep_0 < K$ we get
\begin{equation*}
\begin{split}
& \|\x(T_{2\ell}, T_{2k} ; \xxi)  -\x(T_{2\ell}, T_{2k} ; \zz)\| \ge
  \|\ga(-\al_\ell^{\ee^+}(\ep))\| - \|\x(T_{2\ell}, T_{2k} ; \zz)\| \\
 &  -\|\x(T_{2\ell}, T_{2k} ; \xxi)- \ga(-\al_\ell^{\ee^+}(\ep))\| \ge   \|\ga(-\al_\ell^{\ee^+}(\ep))\|  - 2 \tilde{c}^* \ep  \ge K,
\end{split}
\end{equation*}
a contradiction, and the continuity follows.

Since the assumptions of Theorem \ref{restatefuture} \textbf{c)} are weaker than the ones of
Theorem \ref{restatefuture} \textbf{a)} and \textbf{b)} the lemma in these cases is proved, too.
\end{proof}

Now let us fix $\rho>0$ small enough, independent of $\ep$, so that in $B(\bs{\Gamma}, \rho)$ no sliding phenomena may take place.
Then from Remark \ref{data.smooth} we see that for any $k \in \Z^+$ the function $F_k: B(\bs{\Gamma}, \rho) \to \Omega$,
$F_k(\xxi)= \x(T_{2k+2}, T_{2k}; \xxi)$ is a homemomorphism onto its image; the same property holds for $F_0: B(\bs{\Gamma}, \rho) \to
\Omega$,
$F_0(\xxi)= \x(T_{2}, \tau; \xxi)$.
However, notice that $F_k$ is not a diffeomorphism for $k \ge 0$, since the flow of \eqref{eq-disc} is continuous in the domain but not
smooth.
Hence
by construction $F_k :  \aleph_{2k}^+ \to \aleph_{2(k+1)}^+$ is a well-defined homeomorphism too, for any $k \ge 0$.
We want to prove the following.
\begin{theorem} \label{thm.bernoulli}
Let the assumptions of   Theorem \ref{restatefuture}  be satisfied.
  Then for any $0<\ep \le \ep_0$ and any $k \in \Z^+ \cup \{0\}$, we find
\begin{equation}\label{eq.com}
\Psi_{k+1} \circ F_{k} (\xxi) = \sigma \circ \Psi_k (\xxi) \, , \qquad \xxi \in \aleph_{2k}^+,
\end{equation}
i.e., for all $k \ge 0$
the following diagram
commutes:
$$\xymatrix { {\aleph_{2k}^+} \ar[rrrrr]^{F_k}
\ar[d]_{\Psi_k} &&&&&
{\aleph_{2(k+1)}^+}\ar[d]^{\Psi_{k+1}}\\
{\EE^+} \ar[rrrrr]_{\sigma} &&&&& {\EE^+}}
$$
with the notation \eqref{defSk}.
Moreover, for all $k \ge 0$ the map $\Psi_k$ is continuous and onto.
\end{theorem}
 \begin{proof}
  We borrow the argument from the proof of Theorem 6.1 in \cite{BF11}.

 We have already shown that $\Psi_k$ is continuous and onto, so we just need to show that the diagram commutes.
 Let $\xxi_k \in \aleph_{2k}^+$ and let $\xxi_{k+1}=  F_k(\xxi_k)= \x(T_{2k+2},T_{2k}; \xxi_k)$, and $\xxi_0=\x(\tau, T_{2k}; \xxi_k) \in
 \aleph_0^+$.
 Let us denote by $\ee= \Psi_0(\xxi_0)$,  $\ee'= \Psi_k(\xxi_k)$, $\ee''= \Psi_{k+1}(\xxi_{k+1})$; then by construction
 $\ee''= \sigma^{k+1}(\ee)$ and $\ee'= \sigma^{k}(\ee)$ so that $\ee''= \sigma (\ee')$. Hence
\begin{equation*}
  \begin{split}
   \Psi_{k+1}\circ F_k(\xxi_k)=  \Psi_{k+1}(\xxi_{k+1})= \ee''= \sigma(\ee')=\sigma \circ \Psi_k (\xxi_k),
  \end{split}
\end{equation*}
so the diagram commutes.
 \end{proof}

Using the inversion of time argument of Section \ref{proof.past}, or simply repeating the argument in backward time, we reprove all the
results of this
subsection
 in the setting of Theorem \ref{restatepast}, in particular Remark \ref{remtwo}, Proposition \ref{continuityine} and
 Theorem~\ref{thm.bernoulli}
 hold for $\aleph^-$ and for $k \le 0$.

\smallskip

 \begin{remark}
  When $\g$ and consequently $\MM$ are $1$-periodic (and we choose a periodic sequence $T_k=t_0+ k\theta(\ep)$
  satisfying \eqref{TandKnu}),
   it can be proved that $\aleph_2^+=\aleph_{2k}^+$, $\Psi_1=\Psi_{k}$ and $F_1=F_k$ for any $k \in \Z^+$, so that the set $\aleph_2^+$ is
   invariant
   for the flow of $F_1$. \\ \indent
    On the contrary,
   even when $\g$ and $\MM$ are quasi-periodic or almost periodic, $\aleph_2$ differs from $\aleph_{2k}^+$ for any $k>1$
   (even if they are quite close since they satisfy \eqref{e.alep}) and consequently
   the endomorphisms $\Psi_k: \aleph_{2k}^+ \to \EE^+$ and the maps $F_k: \aleph_{2k}^+ \to \aleph_{2(k+1)}^+$ are all slightly different.
   In fact we have the same situation classically in literature when dealing with chaos in the real line with almost periodic perturbation,
   see, e.g.\ \cite{PS} and in particular the Remark at page 599 in \cite{PS}, i.e., $\aleph_2^+$ is not invariant under
   $F_1$. However
    in the periodic case we find, e.g., infinitely many periodic orbits
    and in the almost periodic case infinitely many almost
    periodic orbits.
   \\ \indent
   This kind of argument, relying on variable sets $\aleph_{2k}^+$ and variable endomorphisms $\Psi_k: \aleph_{2k}^+ \to \EE^+$
   which is needed in the almost periodic case,  as can be found  detailed in \cite{BF11}, is naturally generalized
   to our aperiodic (and one sided) setting.
\end{remark}

Now we briefly consider the setting of Theorem \ref{restatebis}. In this case
we can still define the sets $\aleph_k^+$ as in \eqref{defSk}. Further,    we can define the  mappings
$\Psi_k$ and $F_k$ as above and  we obtain that  $\Psi_k$ is onto and $F_k$ is a homeomorphism.   However
Remark \ref{remtwo} is replaced by the following weaker result.
 \begin{remark}\label{remthree}
 Assume the hypotheses of Theorem \ref{restatebis} \textbf{a)},  and set
 $\ov{B}_N= \max \{B_j \mid 1 \le j \le N \}$ and
  $\bs{\Gamma^N}= \{ \ga(t) \mid |t| \le \ov{B}_N \}$. Then
$$ \aleph^+(\TT^+)=\aleph_0^+  \subset B(\ga(0), \tilde{c}^* \ep)   , \qquad
\aleph_{2k}^+ \subset [B(\bs{\Gamma^{k}}, \tilde{c}^* \ep) \cup B(\vec{0}, \tilde{c}^* \ep)] \subset B(\bs{\Gamma}, \tilde{c}^* \ep)$$
 for any $k \in \Z^+ $.
 \end{remark}
 Notice that if $(B_j)$ is unbounded then the two sets $B(\bs{\Gamma^{k}}, \tilde{c}^* \ep)$
 and $B(\vec{0}, \tilde{c}^* \ep)$ may intersect, so $\Psi_k$ might not be continuous.

 So with the previous argument we can reprove \eqref{eq.com}  obtaining that  $\Psi_k$ is onto but possibly discontinuous, so we do not have a
 real
 semi-conjugation with the Bernoulli shift even if the diagram in Theorem \ref{thm.bernoulli} still commutes.

\appendix
\section{Formula  for the Melnikov function}

This appendix is devoted to correct the imprecise formula
for the Melnikov function $\MM(\al)$ in the $2$-dimensional case appeared
at page 757 in \cite{CaFr}.
Unfortunately the mistake has been repeated in \cite{CFPRimut} even if the formula was not explicitly used there.

Let us denote by $\mathcal{N}\bs{A}$ and by $\mathcal{R}\bs{A}$ respectively the nullspace and the range of the matrix $\bs{A}$.
First we recall that the condition \assump{F4$'$},
  \begin{description}
\item[\assump{F4$'$}] $\dim([\lspan(\vec{\nabla}G(\ga(0)))]^{\perp} \cap  \mathcal{N} \boldsymbol{P^-} \cap \mathcal{R}
    \boldsymbol{P^+})=n-2=0,$
\end{description}
 required in \cite{CaFr} follows from $\bs{K}$, since $\mathcal{N} \boldsymbol{P^-}=\lspan[\f^-(\ga(0))]$ and $\mathcal{R}
 \boldsymbol{P^+}=\lspan[\f^+(\ga(0))]$
 in the $2$-dimensional case.

 We denote by $\vec{\psi}=\bs{J} \frac{\na G (\ga(0))}{\|\na G (\ga(0))\|} $ where
 $\boldsymbol{J}=\left(\begin{smallmatrix}0&-1\\1&0\end{smallmatrix}\right)$.
 Further we denote by $\vec{w}^{\pm}= \bs{J} \frac{\f^{\pm}(\ga(0))}{ \|\f^{\pm}(\ga(0))\|}$.

 We start from the general formula for $\MM(\tau)$ developed in \cite{CaFr} at page 753 for the $n \ge 2$ case (which is correct) and we
 detail
 the reduction to the $2$-dimensional case:
 \begin{equation*}
    M(\al)=\int_{-\infty}^{+\infty} \vec{\psi}^*(t)\g(t+\al,\ga(t),0) dt\,
\end{equation*}
where
$$\vec{\psi}(t)= \left\{\begin{array}{ll} \vec{\psi}^-(t)=
[[\boldsymbol{X^-}(t)]^{-1}]^*  [\boldsymbol{R^-}]^* \vec{\psi} &
\text{if $t<0$}, \\ %
\vec{\psi}^+(t)=[[\boldsymbol{X^+}(t)]^{-1}]^* [\boldsymbol{R^+}]^* \vec{\psi} & \text{if $t \ge 0$,}
\end{array} \right.
$$
 and $\bs{X^{\pm}}(t)$ is the fundamental matrix of the variational system, i.e.\ the solution of $\bs{\dot{X}^{\pm}}=
 \bs{f^{\pm}_x}(\ga(t))\bs{X^{\pm}}$, respectively for $t\ge 0$ and $t \le 0$,
 such that $\bs{X^{\pm}}(0)=\bs{I}$ is the identity, and
 $\boldsymbol{R^{\pm}}$ is the projection with range $\mathcal{R}(\bs{R^{\pm}})=\lspan[\vec{\psi}]$ and nullspace
 $\mathcal{N}(\bs{R^{\pm}})=\lspan[\f^{\pm}(\ga(0))]$, i.e.
 $$\boldsymbol{R^\pm}=\bs{I}-\frac{\f^\pm(\ga(0))\na G(\ga(0))^*}{\na G(\ga(0))^*\f^\pm(\ga(0))}.$$
 Recall that $\mathcal{R}[(\bs{R^{\pm}})^*]$ is the orthogonal complement to $\mathcal{N}(\bs{R^{\pm}})$, so
  there are $C_{\perp}^{\pm} \in \R$   such that $[\bs{R^{\pm}}]^*\vec{\psi}= C_{\perp}^{\pm} \vec{w}^{\pm}$.
  Omitting  the dependence of $G$ and $\f^{\pm}$ on $\ga(0)$, we find
 \begin{equation*}
	\begin{split}
		&   ((\bs{R^{\pm}})^* \bs{J}\na G)^*  \,  \bs{J}\f^{\pm }     =
		  (\bs{J} \na G)^*  \,    \bs{R^{\pm}}\bs{J}\f^{\pm }     = \left(\bs{J} \na G \right)^{*}\left(\bs{I}-\frac{\f^\pm\na G^*}{\na
G^*\f^\pm}\right)\bs{J}\f^{\pm }\\
		&= [\bs{J}\na G]^{*}[\bs{J}\f^{\pm }]-\frac{[\bs{J}\na G]^{ *}\f^\pm \cdot \na G^*[\bs{J}\f^{\pm}]}{\na G^*\f^\pm}
		= \frac{(\na G^*\f^\pm)^2+[(\bs{J}\na G)^{ *}\f^\pm]^2}{\na G^*\f^\pm}\\
		&= \frac{\|\na G\|^2\|\f^\pm\|^2}{\na G^*\f^\pm},
	\end{split}
\end{equation*}
  where we  used $[\bs{J}\vec{v}]^*[\bs{J}\vec{w}]= \vec{v}^*\vec{w}$ and $\bs{J}\bs{J} \vec{v}= - \vec{v}$.
  Hence
  \begin{equation*}
    \begin{split}
        &  C_{\perp}^{\pm} =  ((\bs{R^{\pm}})^* \vec{\psi})^*  \,  \vec{w}^{\pm}   =
         \frac{\|\na G(\ga(0))\| \cdot \|\f^\pm(\ga(0))\|}{ \na G(\ga(0))^*  \f^\pm(\ga(0))} >0
    \end{split}
  \end{equation*}
by \assump{K}.

  Then, following \cite[page 253]{Pa84}, we see that  if $t<0<s$ we find
  $$\vec{\psi}^{-}(t)= C_{\perp}^{-} [(\bs{X^-}(t))^{-1}]^* \vec{w}^-= c_{\perp}^{-} \eu^{-\int_0^t\operatorname{div}\f^-(\ga(\tau))d\tau }
  \bs{J}\f^-(\ga(t)) \,, $$
 $$\vec{\psi}^{+}(s)= C_{\perp}^{+} [(\bs{X^+}(s))^{-1}]^* \vec{w}^+=  c_{\perp}^{+}\eu^{-\int_0^s\operatorname{div}\f^+(\ga(\tau))d\tau }
 \bs{J}\f^+(\ga(s))$$
 with
\begin{equation}\label{cperp}
	c_{\perp}^{\pm}=\frac{C_{\perp}^{\pm}}{\|\f^\pm(\ga(0))\|}
	=\frac{\|\na G(\ga(0))\|}{ \na G(\ga(0))^*  \f^\pm(\ga(0)) } >0.
\end{equation}
 Therefore the Melnikov function $\MM(\al)$ can be written as follows:
  \begin{equation}\label{melni.correct}
    \begin{split}
   \mathcal{M}(\al) &=    c_{\perp}^{-}\int_{-\infty}^{0} \eu^{-\int_0^t
			\tr\boldsymbol{\f_x^-}(\ga^-(s))ds} \left( \bs{J}\f^-(\ga(t)) \right)^*  \,
		 \g(t+\al,\ga^-(t),0)  dt \\
   & + c_{\perp}^{+} \int_{0}^{+\infty} \eu^{-\int_0^t
			\tr\boldsymbol{\f_x^+}(\ga^+(s))ds} \left( \bs{J} \f^+(\ga(t))\right)^*  \,
		 \g(t+\al,\ga^+(t),0) dt\\
   &= c_{\perp}^{-}\int_{-\infty}^{0} \eu^{-\int_0^t
			\tr\boldsymbol{\f_x^-}(\ga^-(s))ds} \f^-(\ga(t))
		\wedge \g(t+\al,\ga^-(t),0) dt\\
		&\quad {}+ c_{\perp}^{+} \int_{0}^{+\infty} \eu^{-\int_0^t
			\tr\boldsymbol{\f_x^+}(\ga^+(s))ds} \f^+(\ga(t))
		\wedge \g(t+\al,\ga^+(t),0) dt.
\end{split}
\end{equation}


\begin{thebibliography}{99}


	\bibitem{BF02C} {\sc  F. Battelli and M. Fe\v ckan}: {\it Chaos arising
		near a topologically transversal homoclinic set},
	Topol. Methods Nonlinear Anal. {\bf 20} (2002), no. 2, 195--215.

\bibitem{BF02M} {\sc  F. Battelli and M. Fe\v ckan}: {\it Some remarks on
the Melnikov function},
Electron. J. Differential Equations {\bf 2002}, no. 13, 29 pp.
	
	\bibitem{BF08} {\sc  F. Battelli and M. Fe\v ckan}: {\it Homoclinic trajectories in discontinuous systems},
	J. Dyn. Differ. Equations {\bf 20} (2008), no. 2,  337--376.
	
	\bibitem{BF10} {\sc  F. Battelli and M. Fe\v ckan}: {\it   An example of chaotic behaviour in presence of a sliding homoclinic orbit},
Ann. Mat. Pura Appl. \textbf{189}
(2010), no. 4, 615--642.
	
	\bibitem{BF11} {\sc  F. Battelli and M. Fe\v ckan}: {\it On the chaotic behaviour of discontinuous systems},
	J. Dyn. Differ. Equations {\bf 23} (2011), no. 3, 495--540.
	
	\bibitem{BF12} {\sc  F. Battelli and M. Fe\v ckan}: {\it  Nonsmooth homoclinic orbits, Melnikov functions and chaos in discontinuous
systems}, Phys. D \textbf{241}
(2012), no. 22, 1962--1975.


 \bibitem{BW} {\sc K. Burns and H. Weiss}: {\it A geometric criterion for positive topological entropy}, Comm.
Math. Phys. {\bf 172} (1995), 95--118.
	




	\bibitem{CDFP}  {\sc A. Calamai, J. Dibl\' ik, M. Franca, and M. Posp\' i\v sil }: {\it On the position of chaotic trajectories},  J. Dyn.
Differ. Equations {\bf 29}
(2017), no. 4, 1423--1458.

	\bibitem{CaFr} {\sc A. Calamai and M. Franca}: {\it Melnikov methods and homoclinic orbits in discontinuous systems},
	J. Dyn. Differ. Equations {\bf 25} (2013), no. 3, 733--764.

\bibitem{CFPRimut}  {\sc A. Calamai,  M. Franca, and M. Posp\' i\v sil}: {\it On the dynamics of non-autonomous systems in a neighborhood of a
    homoclinic
    trajectory},   Rend. Mat. Univ. Trieste, {\bf 56} (2024), art. no. 10, 67 pp.

    \bibitem{CFPdiscAr}  {\sc A. Calamai,  M. Franca, and M. Posp\' i\v sil}: {\it  A new construction for Melnikov chaos in piecewise-smooth
        planar systems},  preprint (2025), 51 pages,
https://doi.org/10.48550/arXiv.2507.15543
	


\bibitem{CHM} {\sc S.N. Chow, J.K. Hale, and J. Mallet-Paret}: {\it An example of bifurcation to homoclinic orbits},
J. Differential Equations {\bf 37} (1980), no. 3, 351--373.
	


  \bibitem{DE13}    {\sc L. Dieci, C. Elia and L. Lopez}: {\it A Filippov sliding vector field on an attracting co-dimension 2 discontinuity
      surface, and a limited
      loss-of-attractivity analysis}, J. Differential Equations 254 (2013), no. 4, 1800--1832.
	
	\bibitem{FrPo} {\sc  M. Franca and M. Posp\' i\v sil}: {\it New global bifurcation diagrams for piecewise smooth systems: Transversality
of homoclinic points does not
imply chaos},
	J.  Differ. Equations {\bf 266} (2019), 1429--1461.

\bibitem{mFaS}
{\sc M. Franca and A. Sfecci}: {\it On a diffusion model with absorption and production},
Nonlinear Anal. Real World Appl. \textbf{34} (2017), 41--60.



\bibitem{G1} {\sc J. Gruendler}: {\it Homoclinic solutions for autonomous ordinary differential equations with nonautonomous perturbations},
J. Differ. Equ. {\bf 122} (1995), no. 1, 1--26.


\bibitem{GZ23} {\sc O. Gjata and F. Zanolin}: {\it An application of the Melnikov method to a piecewise oscillator},
Contemp. Math. {\bf 4} (2023), no. 2, 249--269.

\bibitem{GH} {\sc J. Guckenheimer and P. Holmes}: {\it Nonlinear Oscillations,
Dynamical Systems, and Bifurcations of Vector Fields},
Springer--Verlag, New York, 1983.

\bibitem{HK} {\sc J. Hale and H. Ko\c{c}ak}: {\it Dynamics and Bifurcations}, Volume 3, Springer--Verlag, New York, 1991.


\bibitem{HL24} {\sc D. Hua and X. Liu}: {\it Limit cycle bifurcations near nonsmooth homoclinic cycle in discontinuous systems},
J. Dyn.
Differ. Equations (2024), https://doi.org/10.1007/s10884-024-10358-7.



\bibitem{Jsell}
{\sc R. Johnson}: {\it Concerning a theorem of Sell}, J. Differential Equations \textbf{30} (1978),
324--339.	

\bibitem{JPY} {\sc R. Johnson, X.B. Pan and Y.F. Yi}: {\it The Melnikov method and
elliptic equations with critical exponent},
Indiana Univ. Math. J. {\bf 43} (1994), no. 3, 1045--1077.


\bibitem{HuLi} Hua D, Liu X. \textit{Bifurcations of degenerate homoclinic solutions in discontinuous systems under non-autonomous
    perturbations}, Chaos
    \textbf{34(6)} (2024),   doi: 10.1063/5.0200037.

\bibitem{K2} {\sc P. Kuku\v cka}: {\it Melnikov method for discontinuous planar systems},
Nonlinear Anal. {\bf 66} (2007), no. 12, 2698--2719.

\bibitem{KRG} {\sc Yu.A. Kuznetsov, S. Rinaldi and A. Gragnani}: {\it One-parametric bifurcations in planar Filippov systems},
Int. J. Bifurcation Chaos {\bf 13} (2003), no. 8, 2157--2188.



\bibitem{LSZH16} {\sc S. Li, C. Shen, W. Zhang, and Y. Hao}: {\it The Melnikov method of heteroclinic orbits for a class of planar hybrid
    piecewise-smooth systems and
    application}, Nonlinear Dynam. \textbf{85} (2016), no. 2, 1091--1104.

\bibitem{LPT} {\sc J. Llibre, E. Ponce and A.E. Teruel}:
{\it Horseshoes near homoclinic orbits for piecewise linear
differential systems in $\R^3$},
Int. J. Bifurcation Chaos {\bf 17} (2007), no. 4, 1171--1184.


\bibitem{MV21} {\sc O. Makarenkov and F. Verhulst}: {\it Resonant periodic solutions in regularized impact oscillator},
J. Math. Anal. Appl. {\bf 499} (2021), Paper no. 125035, 17 pp.

\bibitem{Me}
{\sc V.K. Melnikov}:  {\it On the stability of the center for time-periodic perturbations},
Trans. Mosc. Math. Soc. {\bf 12} (1963), 1--56.

\bibitem{MS} {\sc K.R. Meyer and G.R. Sell}: {\it Melnikov transforms,
Bernoulli bundles, and almost periodic perturbations}, Trans.
Amer. Math. Soc. {\bf 314} (1989), 63--105.

\bibitem{Pa84} {\sc K.J. Palmer}: {\it Exponential dichotomies and transversal
homoclinic points}, J. Differential Equations {\bf 55} (1984), no. 2, 225--256.

\bibitem{Pa00} {\sc K.J. Palmer}: {\it Shadowing in Dynamical systems. Theory and Applications},
{\bf 501}, Kluwer Academic Publishers, Dordrecht, 2000.

\bibitem{PS} {\sc K.J. Palmer and D. Stoffer}: {\it Chaos in almost periodic systems},
Z. Angew. Math. Phys.  (ZAMP) {\bf 40} (1989), no. 4, 592--602.

 \bibitem{PZ}
{\sc  D. Papini and F. Zanolin}: {\it Periodic points and chaotic-like dynamics
of planar maps associated to nonlinear Hill's equations with
indefinite weight}, Georgian Math. J. {\bf 9} (2002),  339--366.



	
\bibitem{Sch} {\sc J. Scheurle}: {\it Chaotic solutions of systems with almost periodic forcing},
Z. Angew. Math. Phys.  (ZAMP) {\bf 37} (1986), no. 1, 12--26.

 \bibitem{SW}  {\sc R. Srzednicki and K. Wojcik}: {\it A geometric method for detecting chaotic dynamics},
J. Differ. Equ. {\bf 135} (1997), 66--82.

\bibitem{St} {\sc D. Stoffer}: {\it Transversal homoclinic points and hyperbolic sets for non-autonomous maps I, II},
Z. Angew. Math. Phys.  (ZAMP) {\bf 39} (1988), no. 4,  518--549 \textup{and} no. 6,  783--812.


\bibitem{WigBook} {\sc S. Wiggins}: {\it Introduction to Applied Nonlinear Dynamical Systems and Chaos},
Springer, New York, 2003.

\bibitem{Wig99} {\sc S. Wiggins}: {\it Chaos in the dynamics generated by sequences of maps,
with applications to chaotic advection in flows with aperiodic
time dependence}, Z. Angew. Math. Phys. (ZAMP) {\bf 50} (1999),  no. 4, 585--616.




\end{thebibliography}
\end{document}